\documentclass[smallextended]{svjour3}       
\smartqed  
\usepackage{graphicx}
\usepackage{booktabs}
\usepackage{diagbox}
\usepackage{amssymb}
\usepackage{amsmath}
\usepackage{float}
\usepackage{color}
\usepackage{bm}
\renewenvironment{proof}{\hspace{-\parindent}{Proof.}}{\hspace*{\fill}$\square$\qquad}

\journalname{Journal of Scientific Computing}
\begin{document}

\title{Splitting finite element approximations for  quasi-static electroporoelasticity equations}   
\thanks{The research is supported by National Natural Science Foundation of China grants 12171199 and 22341302; Jilin Provincial Department of Science and Technology grant 20240301017GX; National Key Research and Development Program of China grants 2020YFA0713602 and 2023YFA1008803; and Key Laboratory of Symbolic Computation and Knowledge Engineering of Ministry of Education of China housed at Jilin University.}

\titlerunning{Well-posedness \& FEM to electroporoelasticity PDEs}        

\author{Xuan Liu \and
        Yongkui Zou \and
        Ran Zhang \and
        Yanzhao Cao\and 
        Amnon J. Meir 
}

\authorrunning{X. Liu, Y. Zou, R. Zhang, Y. Cao and A. J. Meir} 

\institute{Xuan Liu \at
              School of Mathematics, Jilin University, Changchun 130012, People’s Republic of China \\
              \email{liuxuan23@mails.jlu.edu.cn}           
           \and
           Yongkui Zou \at
              School of Mathematics, Jilin University, Changchun 130012, People’s Republic of China \\
              \email{zouyk@jlu.edu.cn}
           \and
           Ran Zhang \at
              School of Mathematics, Jilin University, Changchun 130012, People’s Republic of China \\
              \email{zhangran@jlu.edu.cn}
           \and
           Yanzhao Cao \at
           Department of Mathematics and Statistics, Auburn University, Auburn, AL 36849, USA \\
           \email{yzc0009@auburn.edu}
            \and
           Amnon J Meir  \at
           Department of Mathematics, Southern Methodist University, Dallas, TX 75205, USA \\
           \email{ajmeir@smu.edu}
}

\date{Received: date / Accepted: date}

\maketitle

\begin{abstract}

The electroporoelasticity model, which couples Maxwell's equations with Biot's equations, plays a critical role in applications such as water conservancy exploration, earthquake early warning, and various other fields. This work focuses on investigating its well-posedness and analyzing error estimates for a splitting backward Euler finite element method. We first define a weak solution consistent with the finite element framework. Then, we prove the uniqueness and existence of such a solution using the Galerkin method and derive a priori estimates for high-order regularity. Using a splitting technique, we define an approximate splitting solution and analyze its convergence order. Next, we apply Nédélec’s curl-conforming finite elements, Lagrange elements, and the backward Euler method to construct a fully discretized scheme. We demonstrate the stability of the splitting numerical solution and provide error estimates for its convergence order in both temporal and spatial variables. Finally, we present numerical experiments to validate the theoretical results, showing that our method significantly reduces computational complexity compared to the classical finite element method.

\end{abstract}

\section{Introduction}\label{s3.1}
The seismoelectric coupling (including seismoelectric effect and electroseismic effect) \cite{Fr2005} describes the conversion law of seismic waves and electromagnetic fields occurring in elastic porous media \cite{Co2004}. Pride \cite{Pr1994} derived a theoretical model (i.e., electroporoelasticity model) by coupling Maxwell's equations \cite{BaLi2021,Li2007,LiCh2006} and Biot's equations \cite{Bi1962,CaChMe2013,CaChMe2014,ZhZhCaMe2020} to describe both electroseismic and seismoelectric phenomena in isotropic, homogeneous and electrolyte-saturated porous media (e.g., sedimentary rocks). Later on, Pride and Haartsen \cite{PrHa1996} generalized the model to anisotropic case. Seismoelectric coupling plays an important role in many fields such as oil-gas exploration \cite{WhZh2006}, water conservancy exploration \cite{KuMuRi2006}, earthquake early warning \cite{PrGa2005,PrMoGa2004}, environmental protection \cite{MiSoMo2007} and some other areas in geophysics \cite{BePe1998,Wh2005}, etc. Specifically, seismoelectric coupling waves possess high spatial resolution and are sensitive to the change of resistivity such that electroseismic prospecting is more effective to locate and identify oil and gas reservoirs than seismic exploration \cite{DoSa1960} or electromagnetic exploration \cite{Na1988}. Besides, electromagnetic waves have a much faster propagation speed than seismic waves, and they would travel to the ground earlier than seismic waves. If abnormal electromagnetic signals can be detected in advance, an earthquake warning may be issued.

Several studies have been conducted on numerical approximations of the electroporoelasticity equations. Haartsen and Pride \cite{HaPr1997} determined the full-waveform electroseismic response to a point source in a stratified porous medium. Haines and Pride \cite{HaPr2006} developed a finite difference algorithm to simulate seismoelectric coupling in arbitrarily heterogeneous porous media. Santos \cite{Sa2011} introduced finite element approximations and procedures for two electrokinetic coupling modes in a two-dimensional isotropic, fluid-saturated, poroviscoelastic domain. McGhee and Picard \cite{McPi2011-1,McPi2011-2} examined electroseismic waves in anisotropic, inhomogeneous and time-shift invariant media, establishing the well-posedness of an electroseismic model. Santos et al.\ \cite{SaZyGa2012} focused on an electroseismic model that neglects seismoelectric phenomenon, demonstrating the model's unique solution and providing a priori error estimates for a Galerkin semi-discretized procedure, along with an analysis of the unconditional stability of an implicit fully discretized finite element method (FEM). Hu and Meir \cite{HuMe2022} researched the standard finite element approximation to quasi-static electroporoelasticity equations.

The purpose of this paper is to establish a well-posedness theory for the quasi-static electroporoelasticity model and investigate the convergence order of its splitting backward Euler FEM approximation. In \cite{HuMe2023}, Hu and Meir studied a quasi-static electroporoelasticity model and proved the existence of a unique weak solution that is essentially bounded in time and square integrable in space. However, the test function used in their definition requires higher regularity and depends on the temporal variable, which makes it inconvenient for numerical analysis within this framework.

In this work, we first introduce a new definition of the solution to the quasi-static electroporoelasticity equations, which is consistent with the standard weak solution. This solution is $H^{1}$-smooth in all variables, and the test function is independent of the temporal variable. We then prove the uniqueness and existence of such a solution using the Galerkin method. Furthermore, we derive a priori estimates of higher-order regularity for the solution, which are crucial for performing error analysis of the numerical approximation.

The electroporoelasticity model consists of 10 physical unknowns in a three-dimensional spatial domain, including the electric and magnetic fields, the medium displacement, and the pressure, making it a problem of high numerical complexity. To address this challenge, we construct a splitting approximate solution using a splitting technique \cite{ZhZoChCa2024,ZhZoChZhCa2022} and demonstrate that this approximation converges to the exact solution with first-order accuracy in time. Furthermore, we employ N'{e}d'{e}lec's curl-conforming elements \cite{Mo1991,Mo2003,Ne1980} and Lagrange elements, combined with the backward Euler method, to establish a fully discretized scheme. We prove that the splitting numerical approximation converges to the exact solution with a convergence rate of $\mathcal{O}(\tau + h)$. Our comprehensive analysis, supported by numerical experiments, shows that the splitting method is several times faster than the standard finite element method.

The rest of this paper is organized as follows. In section \ref{s3.2}, we introduce the quasi-static electroporoelasticity equations and present a new definition of a weak solution. In section \ref{s3.3}, we prove the uniqueness and existence of the weak solution by Galerkin method. In section \ref{s3.4}, we construct a splitting approximate solution and obtain its convergence order. In section \ref{s3.5}, we set up a splitting backward Euler FEM scheme and derive error estimates. In section \ref{s3.6}, we carry out numerical experiments to demonstrate the theoretical results and the efficiency of the proposed algorithm.

\section{Electroporoelasticity equations and weak solution}\label{s3.2}
Let $T>0$ and $\Omega\subset \mathbb{R}^{3}$ be an open, bounded, simply connected and convex domain with Lipschitz continuous boundary $\partial\Omega$. Denote by $\mathbf{n}$ the unit outward normal vector to $\partial\Omega$. We consider a model of quasi-static electroporoelasticity equations in the temporal and spatial domain $[0,T]\times \Omega$: 
\begin{equation}\label{e3.2.1}
	\begin{aligned}
		& \epsilon\frac{\partial}{\partial t}\mathbf{E} + \sigma \mathbf{E} - \nabla\times \mathbf{H} - L\nabla p = \mathbf{j},
		\\
		& \mu\frac{\partial}{\partial t}\mathbf{H} + \nabla\times \mathbf{E} = \mathbf{0},
		\\
		& -\lambda_{c}\nabla(\nabla\cdot \mathbf{u}) - G\Delta \mathbf{u} + \alpha\nabla p = \mathbf{f},
		\\
		& \frac{\partial}{\partial t}(c_{0}p + \alpha\nabla\cdot \mathbf{u}) - \kappa\Delta p + L\nabla\cdot \mathbf{E} = g,
	\end{aligned}
\end{equation}
where  $\mathbf{E}$ is the electric field, $\mathbf{H}$ is the magnetic field, $\mathbf{u}$ is the displacement of the solid matrix and $p$ is the pressure in the fluid.  We assume the following initial and boundary conditions. For  $\mathbf{x}\in\Omega$, 
\begin{equation}\label{e3.2.2}
	\mathbf{E}(0,\mathbf{x}) \!=\! \mathbf{E}_{0}(\mathbf{x}), 
x	\quad \mathbf{H}(0,\mathbf{x}) \!=\! \mathbf{H}_{0}(\mathbf{x}),
	\quad \mathbf{u}(0,\mathbf{x}) \!=\! \mathbf{u}_{0}(\mathbf{x}), 
	\quad p(0,\mathbf{x}) \!=\! p_{0}(\mathbf{x}),
\end{equation}
and for $(t,\mathbf{x})\in[0,T]\times\partial \Omega$, 
\begin{equation}\label{e3.2.3}
	\mathbf{E}(t,\mathbf{x})\times \mathbf{n} = \mathbf{0},
	\quad
	\mathbf{u}(t,\mathbf{x})=\mathbf{0}, 
	\quad 
	p(t,\mathbf{x}) = 0.
\end{equation}
 The  physical significance of parameters $\epsilon$, $\mu$, $\sigma$, $L$, $\lambda_{c}$, $G$, $\alpha$, $c_{0}$ and $\kappa$ can be found in \cite{HuMe2022,HuMe2023}. For simplicity, we assume that these parameters are positive constants and $\mathbf{f}=\mathbf{0}$. Throughout this paper, we also denote $\mathbf{E}(t,\mathbf{x})$ by either $\mathbf{E}(t)$ or $\mathbf{E}$, etc.\ when no confusion occurs.

Let $L^{2}(\Omega)$ be the square integrable function space with usual inner product $(\cdot, \cdot)$ and norm $\|\cdot\|$. Throughout this paper, we will use bold italic letters to denote spaces of three-dimensional vector fields ($\bm{L}^{2}(\Omega)=(L^{2}(\Omega))^{3}$, etc.). We emphasize that $\|\cdot\|$ is also used to denote the norm in $\bm{L}^{2}(\Omega)$ when no confusion occurs. Likewise, $L^{2}(\partial \Omega)$ denotes the Hilbert space with inner product $\langle\cdot, \cdot\rangle_{\partial\Omega}$. For any integer $r>0$, denote by $H^{r}:=H^{r}(\Omega)$ the standard Sobolev space with norm $\|\cdot\|_{r}$ and $H^{r}_{0}=\{v\in H^{r}: D^{k}v|_{\partial\Omega}=0, \, |k|<r\}$ \cite{CaHoLi2020}. Let $C^{r}(\bar\Omega)$ be a space consisting of $r$th continuously differentiable functions on $\bar\Omega$ and $C^{r}_{0}(\Omega)$ be its subspace with compact support inside $\Omega$.
Define three spaces
\begin{align*}
	\bm{H}(div) &= \{ \mathbf{v}: \mathbf{v}\in \bm{L}^{2}(\Omega), \ \nabla\cdot \mathbf{v}\in L^{2}(\Omega) \},
	\\
	\bm{H}(\mathbf{curl}) &= \{ \mathbf{v}: \mathbf{v}\in \bm{L}^{2}(\Omega), \ \nabla\times \mathbf{v}\in \bm{L}^{2}(\Omega) \},
	\\
	\bm{H}_{0}(\mathbf{curl}) &= \{ \mathbf{v}: \mathbf{v}\in \bm{H}(\mathbf{curl}), \ \mathbf{v}\times \mathbf{n}=\mathbf{0} \ \text{on} \, \partial\Omega \}.
\end{align*}
Denote
\begin{align*}
	\bar{\mathbb{S}} :=&\ \bm{H}_{0}(\mathbf{curl}) \times \bm{H}(\mathbf{curl}) \times \bm{H}^{1}_{0} \times H^{1}_{0},
	\\
	\bar{\mathbb{H}} :=&\ \bm{L}^{2}(\Omega) \times \bm{L}^{2}(\Omega) \times \bm{L}^{2}(\Omega) \times L^{2}(\Omega),
	\\
	\mathbb{S} :=&\ \bm{H}_{0}(\mathbf{curl}) \times \bm{H}(\mathbf{curl}) \times H^{1}_{0},
	\\
	\mathbb{H} :=&\ \bm{L}^{2}(\Omega) \times \bm{L}^{2}(\Omega) \times L^{2}(\Omega).
\end{align*}

The variational form of \eqref{e3.2.1}--\eqref{e3.2.3} is to seek $(\mathbf{E},\mathbf{H},\mathbf{u},p)\in L^{2}(0,T;\bar{\mathbb{S}}) \cap H^{1}(0,T;\bar{\mathbb{H}})$ such that it satisfies \eqref{e3.2.2} and
\begin{equation}\label{e3.2.4}
	\begin{aligned}
		& \big(\epsilon\frac{\partial}{\partial t}\mathbf{E}(t), \mathbf{D}\big) + \big(\sigma \mathbf{E}(t), \mathbf{D}\big) - \big(\nabla\times \mathbf{H}(t), \mathbf{D}\big) - \big(L\nabla p(t), \mathbf{D}\big) = \big(\mathbf{j}(t), \mathbf{D}\big),
		\\
		& \big(\mu\frac{\partial}{\partial t}\mathbf{H}(t), \mathbf{B}\big) + \big(\nabla\times \mathbf{E}(t), \mathbf{B}\big) = 0,
		\\
		& \big(\lambda_{c}\nabla\cdot \mathbf{u}(t), \nabla\cdot \mathbf{v}\big) + \big(G\nabla \mathbf{u}(t), \nabla \mathbf{v}\big) - \big(p(t), \alpha \nabla\cdot \mathbf{v}\big) = 0,
		\\
		& \big(\frac{\partial}{\partial t}(c_{0}p(t) + \alpha\nabla\cdot \mathbf{u}(t)), q\big) + \big(\kappa\nabla p(t), \nabla q\big) - \big(L\mathbf{E}(t), \nabla q\big) = \big(g(t), q\big),
	\end{aligned}
\end{equation}
for all $t\in[0, T]$ and $(\mathbf{D},\mathbf{B},\mathbf{v},q)\in \bar{\mathbb{S}}$.

In comparison with the standard weak solution, Hu and Meir \cite{HuMe2023} defined a much weaker one which is only essentially bounded in time and square integrable in space. However, the test function is required to depend on temporal variable, resulting that most numerical methods are unable to be applied to approximate such solution. The main purpose of this work is to set up a new well-posedness theory and construct a fast numerical algorithm for quasi-static electroporoelasticity equations. We first provide a new definition of solution to \eqref{e3.2.1}--\eqref{e3.2.3} such that it is consistent with the framework of FEM. Then, we prove the uniqueness and existence and derive high order regularity estimates. Finally, we apply splitting technique and FEM to approximate \eqref{e3.2.4} and analyze its convergence order.

\begin{definition}\label{d3.2.1}
 $(\mathbf{E},\mathbf{H},\mathbf{u},p)\in L^{2}(0,T;\bar{\mathbb{S}}) \cap H^{1}(0,T;\bar{\mathbb{H}})$ is called a weak solution to \eqref{e3.2.1}--\eqref{e3.2.3} if it satisfies initial value conditions \eqref{e3.2.2} and variational equation \eqref{e3.2.4}.
\end{definition}

To prove the existence of the weak soluiton, we first eliminate  $\mathbf{u}$ in \eqref{e3.2.4} by expressing it with respect to $p$.  To this end, we define a bilinear form by
\begin{equation}\label{e3.2.5}
	a(\mathbf{w},\mathbf{v})=(\lambda_{c}\nabla\cdot \mathbf{w}, \nabla\cdot \mathbf{v}) + (G\nabla \mathbf{w}, \nabla \mathbf{v}), \quad \forall \, \mathbf{w},\mathbf{v}\in \bm{H}^{1}_{0},
\end{equation}
which satisfies
\begin{equation*}
	a(\mathbf{w},\mathbf{v})=a(\mathbf{v},\mathbf{w}), \quad \beta_{1} \|\mathbf{v}\|_{1}^{2} \leq |a(\mathbf{v},\mathbf{v})|, \quad |a(\mathbf{w},\mathbf{v})| \leq \beta_{2} \|\mathbf{w}\|_{1} \|\mathbf{v}\|_{1},
\end{equation*}
where $0< \beta_{1} \leq \beta_{2}$ are two constants. Then, for any $p\in L^{2}(\Omega)$, the following equation has a unique solution $\mathbf{u}\in \bm{H}^{1}_{0}$ due to Lax-Milgram theorem \cite{GiRa1986}. 
\begin{equation}\label{e3.2.6}
	a(\mathbf{u},\mathbf{v}) = (p, \alpha\nabla\cdot \mathbf{v}), \quad \forall \, \mathbf{v}\in \bm{H}^{1}_{0}.
\end{equation}
Thus, we can define an associated operator $\mathcal{B}: L^{2}(\Omega)\rightarrow L^{2}(\Omega)$ by $\mathcal{B}p = \alpha\nabla\cdot \mathbf{u}$. According to \cite{CaChMe2013,HuMe2023,Ph2005}, both of $\mathcal{B}$ and $c_{0}+\mathcal{B}$ are linear, bounded, self-adjoint and monotone $((\mathcal{B}p,p)\geq 0)$ operators, which also have bounded inverses and satisfy $\mathcal{B}\frac{\partial}{\partial t}=\frac{\partial}{\partial t}\mathcal{B}$. Furthermore,
\begin{equation}\label{e3.2.7}
	\|(c_{0}+\mathcal{B})p\|^{2} = (c_{0} p,c_{0} p) + 2(c_{0} p,\mathcal{B} p) +(\mathcal{B} p, \mathcal{B} p) \geq c_{0}^{2} \|p\|^{2}.
\end{equation}

\begin{lemma}\label{l3.2.2}
	For any $p\in H^{1}_{0}$ and $i=1,2,3$, there holds that $\mathcal{B}\frac{\partial}{\partial x_{i}}p = \frac{\partial}{\partial x_{i}}\mathcal{B}p$.
\end{lemma}

\begin{proof}
	According to the solvability of \eqref{e3.2.6}, for any $p\in H^{1}_{0}$, there exists a unique $\mathbf{u}\in \bm{H}^{2}\cap \bm{H}^{1}_{0}$ such that $\mathcal{B}p = \alpha\nabla\cdot \mathbf{u}$. Similarly, there exists a unique $\tilde{\mathbf{u}}_{i}\in \bm{H}^{1}_{0}$ such that $\mathcal{B}\frac{\partial}{\partial x_{i}}p = \alpha\nabla\cdot \tilde{\mathbf{u}}_{i}$. We will verify $\tilde{\mathbf{u}}_{i} = \frac{\partial}{\partial x_{i}}\mathbf{u}$. In fact, for any $\mathbf{v}\in \bm{H}^{2}_{0}$, noticing \eqref{e3.2.5}, \eqref{e3.2.6} and applying Green's formula yield
	\begin{align*}
		& a(\tilde{\mathbf{u}}_{i},\mathbf{v})
		= (\frac{\partial}{\partial x_{i}}p,\alpha\nabla\cdot \mathbf{v})
		= -(p,\alpha\nabla\cdot \frac{\partial}{\partial x_{i}}\mathbf{v})
		= -a(\mathbf{u},\frac{\partial}{\partial x_{i}}\mathbf{v})
		\\
		=& -\!(\lambda_{c}\nabla\cdot \mathbf{u}, \nabla\cdot \frac{\partial}{\partial x_{i}}\mathbf{v})\!-\!(G\nabla \mathbf{u}, \nabla \frac{\partial}{\partial x_{i}}\mathbf{v})
		\\
		=& -(\lambda_{c}\nabla\cdot \mathbf{u}, \frac{\partial}{\partial x_{i}}\nabla\cdot \mathbf{v})\!-\!(G\nabla \mathbf{u}, \frac{\partial}{\partial x_{i}}\nabla \mathbf{v})
		\\
		=& (\lambda_{c}\nabla\cdot \frac{\partial}{\partial x_{i}}\mathbf{u}, \nabla\cdot \mathbf{v}) + (G\nabla\frac{\partial}{\partial x_{i}} \mathbf{u}, \nabla \mathbf{v})
		= a(\frac{\partial}{\partial x_{i}}\mathbf{u},\mathbf{v}).
	\end{align*}
	Since $\bm{H}^{2}_{0}$ is dense in $\bm{H}^{1}_{0}$, we obtain $\tilde{\mathbf{u}}_{i} = \frac{\partial}{\partial x_{i}}\mathbf{u}$ in $\bm{H}^{1}_{0}$. Hence,
	\begin{equation*}
		\mathcal{B}\frac{\partial}{\partial x_{i}}p
		= \alpha\nabla\cdot \tilde{\mathbf{u}}_{i}
		= \alpha\nabla\cdot \frac{\partial}{\partial x_{i}}\mathbf{u}
		=\frac{\partial}{\partial x_{i}}\alpha\nabla\cdot \mathbf{u}
		=\frac{\partial}{\partial x_{i}}\mathcal{B}p.
	\end{equation*}
	This completes the proof.
\end{proof}

With the help of the operator $\mathcal{B}$, solving \eqref{e3.2.4} and \eqref{e3.2.2} is equivalent to first searching $\mathbf{U}:=(\mathbf{E},\mathbf{H},p)\in L^{2}(0,T;\mathbb{S}) \cap H^{1}(0,T;\mathbb{H})$ with $\mathbf{U}(0)=\mathbf{U}_{0}:=(\mathbf{E}_{0},\mathbf{H}_{0},p_{0})$ such that for all $t\in[0, T]$ and $(\mathbf{D},\mathbf{B},q)\in \mathbb{S}$,
\begin{equation}\label{e3.2.8}
	\begin{aligned}
		& \big(\epsilon\frac{\partial}{\partial t}\mathbf{E}(t), \mathbf{D}\big) + \big(\sigma \mathbf{E}(t), \mathbf{D}\big) - \big(\nabla\times \mathbf{H}(t), \mathbf{D}\big) - \big(L\nabla p(t), \mathbf{D}\big) = \big(\mathbf{j}(t), \mathbf{D}\big),
		\\
		& \big(\mu\frac{\partial}{\partial t}\mathbf{H}(t), \mathbf{B}\big) + \big(\nabla\times \mathbf{E}(t), \mathbf{B}\big) = 0,
		\\
		& \big(\frac{\partial}{\partial t}(c_{0} + \mathcal{B})p(t), q\big) + \big(\kappa\nabla p(t), \nabla q\big) - \big(L\mathbf{E}(t), \nabla q\big) = \big(g(t), q\big),
	\end{aligned}
\end{equation}
and then determining $\mathbf{u}$ in terms of \eqref{e3.2.6}. 
Therefore, we only need to study the well-posedness of \eqref{e3.2.8} and \eqref{e3.2.2}. We assume

\begin{itemize}
	\item [\bf{(H1)}] $0< L< \sqrt{\sigma\kappa}$.
	\item [\bf{(H2)}] $\mathbf{j}\in H^{1}(0,T;\bm{L}^{2}(\Omega))$ and $g\in H^{1}(0,T;L^{2}(\Omega))$.
	\item [\bf{(H3)}] $\mathbf{E}_{0}\in \bm{H}_{0}(\mathbf{curl})\cap \bm{H}(div)$, $\mathbf{H}_{0}\in \bm{H}(\mathbf{curl})$ and $p_{0}\in H^{2}\cap H^{1}_{0}$.
\end{itemize}

According to (H1), there exist a constant $C>0$ such that for all $\mathbf{E}\in \bm{L}^{2}(\Omega)$ and $p\in H^{1}_{0}$, there holds that
\begin{equation}\label{e3.2.9}
	(\sigma \mathbf{E},\mathbf{E}) - 2(L\nabla p,\mathbf{E}) + (\kappa\nabla p,\nabla p) \geq C(\|\mathbf{E}\|^{2} + \|\nabla p\|^{2}) \geq 0.
\end{equation}

\section{Uniqueness and existence of weak solution}\label{s3.3}
In this section, we establish the uniqueness and existence of weak solution to \eqref{e3.2.1}--\eqref{e3.2.3} and derive its regularity estimates.

\subsection{Uniqueness}\label{s3.3.1}
We begin with an a priori estimate.

\begin{lemma}\label{l3.3.1}
	Assume (H1)--(H3) and that $(\mathbf{E},\mathbf{H},\mathbf{u},p)$ is a weak solution to \eqref{e3.2.1}--\eqref{e3.2.3}. Then there exists a constant $C=C(\epsilon,\mu,c_{0},\lambda_{c},\alpha)>0$ such that for all $t\in [0,T]$,
	\begin{align*}
		&\ \|\mathbf{E}(t)\|^{2} + \|\mathbf{H}(t)\|^{2} + \|\mathbf{u}(t)\|_{1}^{2} + \|p(t)\|^{2}
		\\
		\leq&\ C \Big( \|\mathbf{E}_{0}\|^{2} + \|\mathbf{H}_{0}\|^{2} + \|p_{0}\|^{2} + \int_{0}^{t}{\|\mathbf{j}(\theta)\|^{2}}{\, \mathrm{d}\theta} + \int_{0}^{t}{\|g(\theta)\|^{2}}{\, \mathrm{d}\theta} \Big).
	\end{align*}
\end{lemma}

\begin{proof}
	Obviously, there holds $(\frac{\partial}{\partial t}(c_{0} + \mathcal{B})p, p) = \frac{1}{2}\frac{\mathrm{d}}{\mathrm{d}t}((c_{0} + \mathcal{B})p, p)$.
	Then, taking $(\mathbf{D},\mathbf{B},q)=\big(\mathbf{E}(t),\mathbf{H}(t),p(t)\big)$ in \eqref{e3.2.8} leads to
	\begin{align*}
		& \frac{1}{2}\frac{\mathrm{d}}{\mathrm{d}t}\big(\epsilon\|\mathbf{E}(t)\|^{2} + \mu\|\mathbf{H}(t)\|^{2} + ((c_{0}+\mathcal{B})p(t), p(t))\big)\\
		+ & \sigma\|\mathbf{E}(t)\|^{2} + \kappa\|\nabla p(t)\|^{2} - 2L(\mathbf{E}(t), \nabla p(t))
		= (\mathbf{j}(t), \mathbf{E}(t)) + (g(t), p(t)).
	\end{align*}
	Integrating over $[0,t]$ and noticing \eqref{e3.2.9} yields
	\begin{align*}
		&\ \epsilon\|\mathbf{E}(t)\|^{2} + \mu\|\mathbf{H}(t)\|^{2} + c_{0}\|p(t)\|^{2}
		\leq \epsilon \|\mathbf{E}_{0}\|^{2} + \mu \|\mathbf{H}_{0}\|^{2} + C \|p_{0}\|^{2}
		\\
		&\ + \int_{0}^{t}{\|\mathbf{j}(\theta)\|^{2}}{\, \mathrm{d}\theta} + \int_{0}^{t}{\|g(\theta)\|^{2}}{\, \mathrm{d}\theta}
		+ \int_{0}^{t}{\|\mathbf{E}(\theta)\|^{2}}{\, \mathrm{d}\theta} + \int_{0}^{t}{\|p(\theta)\|^{2}}{\, \mathrm{d}\theta}.
	\end{align*}
	Applying Gronwall's lemma, we obtain	
	\begin{align*}
		&\ \|\mathbf{E}(t)\|^{2} + \|\mathbf{H}(t)\|^{2} + \|p(t)\|^{2}
		\\
		\leq&\ C \Big( \|\mathbf{E}_{0}\|^{2} + \|\mathbf{H}_{0}\|^{2} + \|p_{0}\|^{2} + \int_{0}^{t}{\|\mathbf{j}(\theta)\|^{2}}{\, \mathrm{d}\theta} + \int_{0}^{t}{\|g(\theta)\|^{2}}{\, \mathrm{d}\theta} \Big).
	\end{align*}
	From \eqref{e3.2.6} and the coerciveness of $a(\cdot,\cdot)$, it follows that
	\begin{align*}
		\beta_{1} \|\mathbf{u}(t)\|_{1}^{2} \leq |a(\mathbf{u}(t),\mathbf{u}(t))| = |(p(t), \alpha \nabla\cdot \mathbf{u}(t))| \leq C\|p(t)\|\|\mathbf{u}(t)\|_{1},
	\end{align*}
	which implies $\|\mathbf{u}(t)\|_{1} \leq C\|p(t)\|$. Thus, the proof is completed.
\end{proof}

\begin{theorem}\label{t3.3.2}
	Assume (H1)--(H3). Then the initial-boundary value model \eqref{e3.2.1}--\eqref{e3.2.3} admits at most one weak solution.
\end{theorem}

\begin{proof}
	Assume that $\bar{\mathbf{U}}_{i}:=(\mathbf{E}_{i},\mathbf{H}_{i},\mathbf{u}_{i},p_{i}) \in L^{2}(0,T;\bar{\mathbb{S}}) \cap H^{1}(0,T;\bar{\mathbb{H}})$ $(i=1,2)$ are any two weak solutions to \eqref{e3.2.1}--\eqref{e3.2.3}. Then $\bar{\mathbf{U}}_{1}-\bar{\mathbf{U}}_{2}$ satisfies a homogeneous equations with homogeneous initial and boundary conditions. By virtue of Lemma \ref{l3.3.1}, we obtain for all $t\in [0,T]$,
	\begin{equation*}
		\|\mathbf{E}_{1}(t)-\mathbf{E}_{2}(t)\|^{2} + \|\mathbf{H}_{1}(t)-\mathbf{H}_{2}(t)\|^{2} + \|\mathbf{u}_{1}(t)-\mathbf{u}_{2}(t)\|_{1}^{2} + \|p_{1}(t)-p_{2}(t)\|^{2} \leq 0,
	\end{equation*}
	which implies $\bar{\mathbf{U}}_{1} = \bar{\mathbf{U}}_{2}$.
\end{proof}

\subsection{Existence}\label{s3.3.2}

In this subsection, we establish the existence of a unique weak solution to \eqref{e3.2.8} using the Galerkin method. First, we construct a family of finite-dimensional approximate problems, from which we derive a sequence of approximate solutions. Next, we demonstrate that these approximate solutions are uniformly bounded, ensuring the existence of a weak limit in an appropriate function space. Finally, we show that the limit function possesses additional regularity, confirming it as a solution to \eqref{e3.2.8}. Combined with \eqref{e3.2.6}, this result provides a weak solution to the system \eqref{e3.2.1}--\eqref{e3.2.3}.

The assumption (H3) implies $\mathbf{U}_{0}\in \mathbb{S}$. If $\mathbf{U}_{0}\neq\mathbf{0}$, we have a unique decomposition of space $\mathbb{S}=\operatorname{span}\{\mathbf{U}_{0}\}\oplus \operatorname{span}\{\mathbf{U}_{0}\}^{\perp}$. Let $\mathbf{\Phi}_{1}=\mathbf{U}_{0}/\|\mathbf{U}_{0}\|$ and $\{\mathbf{\Phi}_{i}\}_{i=2}^{\infty}\subset \operatorname{span}\{\mathbf{U}_{0}\}^{\perp}$ such that $\{\mathbf{\Phi}_{i}\}_{i=1}^{\infty}$ is an orthonormal basis of $\mathbb{H}$. If $\mathbf{U}_{0}=\mathbf{0}$, we may choose $\{\mathbf{\Phi}_{i}\}_{i=1}^{\infty}\subset \mathbb{S}$ to be any orthonormal basis of $\mathbb{H}$. Therefore, we always have $\mathbf{U}_{0}\in \operatorname{span}\{\mathbf{\Phi}_{1}\}$.

For any integer $m>0$, define a Galerkin finite dimensional subspace $\mathbf{V}_{m}= \operatorname{span}\{\mathbf{\Phi}_{1},\ldots,\mathbf{\Phi}_{m}\}$. Denote by $P_{m}$ the orthogonal projection operator from $\mathbb{H}$ to $\mathbf{V}_{m}$. Then there  holds that $P_{m}\mathbf{U}_{0} = \mathbf{U}_{0}$. 
A family of Galerkin approximations to \eqref{e3.2.8} is defined as seeking $\mathbf{U}_{m}(t):=(\mathbf{E}_{m}(t),\mathbf{H}_{m}(t),p_{m}(t))\in \mathbf{V}_{m}$ with $\mathbf{U}_{m}(0) = \mathbf{U}_{0}$ such that
\begin{equation}\label{e3.3.1}
	\begin{aligned}
		& \big(\epsilon \mathbf{E}_{m}^{\prime}(t),\mathbf{D}\big) \!\!+\!\! \big(\sigma \mathbf{E}_{m}(t),\mathbf{D}\big) \!\!-\!\! \big(\nabla\!\times\! \mathbf{H}_{m}(t),\mathbf{D}\big) \!\!-\!\! \big(L\nabla p_{m}(t),\mathbf{D}\big) \!\!=\!\! \big(\mathbf{j}(t),\mathbf{D}\big),
		\\
		& \big(\mu \mathbf{H}_{m}^{\prime}(t),\mathbf{B}\big) + \big(\nabla\times \mathbf{E}_{m}(t),\mathbf{B}\big) = 0,
		\\
		& \big((c_{0}+\mathcal{B})p_{m}^{\prime}(t),q\big) + \big(\kappa\nabla p_{m}(t),\nabla q\big) - \big(L \mathbf{E}_{m}(t),\nabla q\big) = \big(g(t),q\big),
	\end{aligned}
\end{equation}
for all $t\in [0,T]$ and $(\mathbf{D},\mathbf{B},q)\in \mathbf{V}_{m}$. According to the classical theory of ODEs \cite{Ar1992}, this equation possesses a unique solution $\mathbf{U}_{m}(t)$. 
We now study its uniform boundedness.

\begin{lemma}\label{l3.3.3}
	Assume (H1)--(H3). Then there exists a constant $C>0$ independent of $m$ such that for all $t\in [0,T]$,
	\begin{align*}
		\|\mathbf{E}_{m}(t)\| + \|\mathbf{H}_{m}(t)\| + \|p_{m}(t)\| + \|\mathbf{E}^{\prime}_{m}(t)\| + \|\mathbf{H}^{\prime}_{m}(t)\| + \|p^{\prime}_{m}(t)\|
		\leq&\ C,
		\\
		\|\nabla \times \mathbf{E}_{m}(t)\| + \|\nabla \times \mathbf{H}_{m}\|_{L^{2}(0,T;\bm{L}^{2}(\Omega))} + \|\nabla p_{m}\|_{L^{2}(0,T;\bm{L}^{2}(\Omega))}
		\leq&\ C.
	\end{align*}
\end{lemma}

\begin{proof}
	Taking $(\mathbf{D},\mathbf{B},q)=(\mathbf{E}_{m}(t),\mathbf{H}_{m}(t),p_{m}(t))$ in \eqref{e3.3.1} gives
	\begin{align*}
		&\ \frac{1}{2}\frac{\mathrm{d}}{\mathrm{d}t}\big(\epsilon\|\mathbf{E}_{m}(t)\|^{2} + \mu\|\mathbf{H}_{m}(t)\|^{2} + ((c_{0}+\mathcal{B})p_{m}(t), p_{m}(t))\big)\\
		+ &\ \sigma\|\mathbf{E}_{m}(t)\|^{2} + \kappa\|\nabla p_{m}(t)\|^{2} - 2L(\mathbf{E}_{m}(t), \nabla p_{m}(t))
		\\
		=&\ (\mathbf{j}(t), \mathbf{E}_{m}(t)) + (g(t), p_{m}(t)).
	\end{align*}
	Integrating over $[0,t]$ and noticing \eqref{e3.2.9}, we obtain
	\begin{align*}
		&\ \|\mathbf{E}_{m}(t)\|^{2} + \|\mathbf{H}_{m}(t)\|^{2} + \|p_{m}(t)\|^{2} + \int_{0}^{t}{\|\nabla p_{m}(\theta)\|^{2}}{\, \mathrm{d}\theta}
		\leq C \Big( \|\mathbf{E}_{0}\|^{2} + \|\mathbf{H}_{0}\|^{2}
		\\
		&\ + \|p_{0}\|^{2} + \int_{0}^{t}{\|\mathbf{j}(\theta)\|^{2}}{\, \mathrm{d}\theta}+\int_{0}^{t}{\|g(\theta)\|^{2}}{\, \mathrm{d}\theta}
		+ \int_{0}^{t}{\big(\|\mathbf{E}_{m}(\theta)\|^{2} \!+\! \|p_{m}(\theta)\|^{2}\big)}{\, \mathrm{d}\theta} \Big).
	\end{align*}
	By (H2), (H3),  and Gronwall's lemma, we have
	\begin{equation}\label{e3.3.2}
		\|\mathbf{E}_{m}(t)\|^{2}+\|\mathbf{H}_{m}(t)\|^{2}+\|p_{m}(t)\|^{2} \leq C,
		\quad
		\|\nabla p_{m}\|_{L^{2}(0,T;\bm{L}^{2}(\Omega))} \leq C,
	\end{equation}
	where $C=C(T)>0$ is a constant independent of $m$.
	
	Choosing $(\mathbf{D},\mathbf{B},q)=(\mathbf{E}_{m}^{\prime}(t),\mathbf{H}_{m}^{\prime}(t),p_{m}^{\prime}(t))$ in \eqref{e3.3.1} and letting  $t=0$ yield
	\begin{align*}
		& \big(\epsilon \mathbf{E}_{m}^{\prime}(0),\mathbf{E}_{m}^{\prime}(0)\big) + \big(\sigma \mathbf{E}_{0},\mathbf{E}_{m}^{\prime}(0)\big) - \big(\nabla\times \mathbf{H}_{0},\mathbf{E}_{m}^{\prime}(0)\big) - \big(L\nabla p_{0},\mathbf{E}_{m}^{\prime}(0)\big)
		\\
		& \quad = \big(\mathbf{j}(0),\mathbf{E}_{m}^{\prime}(0)\big),
		\\
		& \big(\mu \mathbf{H}_{m}^{\prime}(0),\mathbf{H}_{m}^{\prime}(0)\big) + \big(\nabla\times \mathbf{E}_{0},\mathbf{H}_{m}^{\prime}(0)\big) = 0,
		\\
		& \big((c_{0}+\mathcal{B})p_{m}^{\prime}(0),p_{m}^{\prime}(0)\big) \!-\! \big(\kappa\Delta p_{0},p_{m}^{\prime}(0)\big) \!+\! \big(L\nabla\cdot \mathbf{E}_{0}, p_{m}^{\prime}(0)\big) = \big(g(0),p_{m}^{\prime}(0)\big).
	\end{align*}
	By means of Cauchy--Schwarz inequality, we arrive at
	\begin{align*}
		& \epsilon\|\mathbf{E}_{m}^{\prime}(0)\| \leq \|\mathbf{j}(0)\|+\sigma\|\mathbf{E}_{0}\|+\|\nabla\times \mathbf{H}_{0}\|+L\|\nabla p_{0}\| \leq C,
		\\
		& \mu\|_{m}^{\prime}(0)\| \leq \|\nabla\times \mathbf{E}_{0}\| \leq C,
		\\
		& c_{0}\|p_{m}^{\prime}(0)\| \leq \|g(0)\|+\kappa\|\Delta p_{0}\|+L\|\nabla\cdot \mathbf{E}_{0}\| \leq C.
	\end{align*}
	
	Taking derivative with respect to $t$ on both sides of \eqref{e3.3.1} and then choosing $(\mathbf{D},\mathbf{B},q)=(\mathbf{E}_{m}^{\prime}(t),\mathbf{H}_{m}^{\prime}(t),p_{m}^{\prime}(t))$, we have
	\begin{align*}
		& \big(\epsilon \mathbf{E}_{m}^{\prime\prime}(t),\mathbf{E}_{m}^{\prime}(t)\big) \!+\! \big(\sigma \mathbf{E}_{m}^{\prime}(t),\mathbf{E}_{m}^{\prime}(t)\big) \!-\! \big(\nabla \!\times\! \mathbf{H}_{m}^{\prime}(t), \mathbf{E}_{m}^{\prime}(t)\big) \!-\! \big(L\nabla p_{m}^{\prime}(t),\mathbf{E}_{m}^{\prime}(t)\big)
		\\
		& \quad = \big(\mathbf{j}^{\prime}(t),\mathbf{E}_{m}^{\prime}(t)\big),
		\\
		& \big(\mu \mathbf{H}_{m}^{\prime\prime}(t),\mathbf{H}_{m}^{\prime}(t)\big) + \big(\nabla\times \mathbf{E}_{m}^{\prime}(t),\mathbf{H}_{m}^{\prime}(t)\big) = 0,
		\\
		& \big((c_{0}+\mathcal{B})p_{m}^{\prime\prime}(t),p_{m}^{\prime}(t)\big) + \big(\kappa\nabla p_{m}^{\prime}(t),\nabla p_{m}^{\prime}(t)\big) - \big(L \mathbf{E}_{m}^{\prime}(t),\nabla p_{m}^{\prime}(t)\big)
		\\
		& \quad = \big(g^{\prime}(t),p_{m}^{\prime}(t)\big).
	\end{align*}
	In a similar argument as proving \eqref{e3.3.2}, we obtain
	\begin{equation}\label{e3.3.3}
		\|\mathbf{E}_{m}^{\prime}(t)\|^{2}+\|\mathbf{H}_{m}^{\prime}(t)\|^{2}+\|p_{m}^{\prime}(t)\|^{2} \leq C,
	\end{equation}
	where $C>0$ is a constant independent of $m$.
	
	From the first equation of \eqref{e3.3.1}, it follows that
	\begin{equation*}
		\nabla\times \mathbf{H}_{m}(t)=\epsilon \mathbf{E}_{m}^{\prime}(t) + \sigma \mathbf{E}_{m}(t) - L\nabla p_{m}(t) - P_{m}\mathbf{j}(t) \ \text{in} \ \mathbf{V}_{m}.
	\end{equation*}
	According to \eqref{e3.3.2} and \eqref{e3.3.3}, we know $\|\nabla \times \mathbf{H}_{m}\|_{L^{2}(0,T;\bm{L}^{2}(\Omega))} \leq C$. Similarly, $\|\nabla \times \mathbf{E}_{m}(t)\| \leq C$ follows. 
	Thus, the proof is completed.
\end{proof}

This lemma implies that $\mathbf{U}_{m}=(\mathbf{E}_{m},\mathbf{H}_{m},p_{m})$, $m>0$, is a bounded sequence in $L^{\infty}(0,T;\mathbb{H})\subset L^{2}(0,T;\mathbb{H})$, and hence it possesses a weakly convergent subsequence (still denoted as $\mathbf{U}_{m}$) in $L^{2}(0,T;\mathbb{H})$ with limit $\mathbf{U}=(\mathbf{E},\mathbf{H},p)$, i.e.,
\begin{equation}\label{e3.3.4}
	\mathbf{U}_{m}\rightharpoonup \mathbf{U} \ \text{in} \ L^{2}(0,T;\mathbb{H})\ \text{as}\ m\to\infty.
\end{equation}
Similarly, there exist unique functions $\tilde{\mathbf{V}} :=(\tilde{\mathbf{V}}_{1},\tilde{\mathbf{V}}_{2},\tilde{V}_{3})\in L^{2}(0,T;\mathbb{H})$ and $\tilde{\mathbf{q}}, \tilde{\mathbf{W}},\tilde{\mathbf{M}}\in L^{2}(0,T;\bm{L}^{2}(\Omega))$ such that as $m\to\infty$, there hold
\begin{equation}\label{e3.3.5}
	\begin{aligned}
		& \mathbf{U}_{m}^{\prime}\rightharpoonup \tilde{\mathbf{V}} \ \text{in} \ L^{2}(0,T;\mathbb{H}),
		\\
		& \nabla p_{m}\rightharpoonup \tilde{\mathbf{q}},\ \nabla\times \mathbf{E}_{m}\rightharpoonup \tilde{\mathbf{W}},\ \nabla\times \mathbf{H}_{m}\rightharpoonup \tilde{\mathbf{M}} \ \text{in}\ L^{2}(0,T;\bm{L}^{2}(\Omega)).
	\end{aligned}
\end{equation}

For any $\boldsymbol{\phi}\in \bm{L}^{2}(\Omega)$ and $\psi\in L^{2}(0,T)$, by means of \eqref{e3.3.4}, we get
\begin{equation*}
	\lim\limits_{m\to\infty}\int_{0}^{T}{\big(\mathbf{E}_{m}(t),\boldsymbol{\phi}\big)\psi(t)}{\, \mathrm{d}t} = \int_{0}^{T}{\big(\mathbf{E}(t),\boldsymbol{\phi}\big)\psi(t)}{\, \mathrm{d}t}.
\end{equation*}
This leads to for a.e.\ $t\in [0,T]$,
\begin{equation*}
	\lim\limits_{m\to\infty}\big(\mathbf{E}_{m}(t),\boldsymbol{\phi}\big) = \big(\mathbf{E}(t),\boldsymbol{\phi}\big).
\end{equation*}
Similarly, we can prove for a.e.\ $t\in [0,T]$, $\varphi\in L^{2}(\Omega)$ and $\boldsymbol{\phi}\in \bm{L}^{2}(\Omega)$, there hold that
\begin{align*}
	\lim\limits_{m\to\infty}\big(\mathbf{H}_{m}(t),\boldsymbol{\phi}\big) &= \big(\mathbf{H}(t),\boldsymbol{\phi}\big), &
	\lim\limits_{m\to\infty}\big(p_{m}(t),\varphi\big) &= \big(p(t),\varphi\big),
	\\
	\lim\limits_{m\to\infty}\big(\mathbf{E}_{m}^{\prime}(t),\boldsymbol{\phi}\big) &= \big(\tilde{\mathbf{V}}_{1}(t),\boldsymbol{\phi}\big), &
	\lim\limits_{m\to\infty}\big(\mathbf{H}_{m}^{\prime}(t),\boldsymbol{\phi}\big) &= \big(\tilde{\mathbf{V}}_{2}(t),\boldsymbol{\phi}\big),
	\\
	\lim\limits_{m\to\infty}\big(p_{m}^{\prime}(t),\varphi\big) &= \big(\tilde{V}_{3}(t),\varphi\big), &
	\lim\limits_{m\to\infty}\big(\nabla p_{m}(t),\boldsymbol{\phi}\big) &= \big(\tilde{\mathbf{q}},\boldsymbol{\phi}\big),
	\\
	\lim\limits_{m\to\infty}\big(\nabla\times\mathbf{E}_{m}(t),\boldsymbol{\phi}\big) &= \big(\tilde{\mathbf{W}}(t),\boldsymbol{\phi}\big), &
	\lim\limits_{m\to\infty}\big(\nabla\times\mathbf{H}_{m}(t),\boldsymbol{\phi}\big) &= \big(\tilde{\mathbf{M}}(t),\boldsymbol{\phi}\big).
\end{align*}

The next lemma is about  the regularity of the weak limit function $(\mathbf{E},\mathbf{H},p)$.

\begin{lemma}\label{l3.3.4}
	Assume (H1)--(H3). Then
	\begin{equation*}
		\frac{\partial \mathbf{U}}{\partial t}=\tilde{\mathbf{V}},\ \nabla p=\tilde{\mathbf{q}},\ \nabla\times \mathbf{E}=\tilde{\mathbf{W}},\ \nabla\times \mathbf{H}=\tilde{\mathbf{M}}.
	\end{equation*}
	Furthermore, $p\in L^{2}(0,T;H^{1}_{0})$ and $\mathbf{E}\in L^{2}(0,T;\bm{H}_{0}(\mathbf{curl}))$.
\end{lemma}

\begin{proof}
	For any given $\mathbf{\Psi}_{1}\in C^{1}_{0}(0,T;\bm{L}^{2}(\Omega))$, integrating by parts gives
	\begin{equation*}
		\int_{0}^{T}{\big(\mathbf{E}_{m}^{\prime}(t), \mathbf{\Psi}_{1}(t)\big)}{\, \mathrm{d}t} = -\int_{0}^{T}{\big(\mathbf{E}_{m}(t), \frac{\partial\mathbf{\Psi}_{1}}{\partial t}(t)\big)}{\, \mathrm{d}t}.
	\end{equation*}
	By virtue of \eqref{e3.3.4} and \eqref{e3.3.5}, let $m\to \infty$ in above equation, we obtain
	\begin{equation*}
		\int_{0}^{T}{\big(\tilde{\mathbf{V}}_{1}(t), \mathbf{\Psi}_{1}(t)\big)}{\, \mathrm{d}t} = -\int_{0}^{T}{\big(\mathbf{E}(t), \frac{\partial\mathbf{\Psi}_{1}}{\partial t}(t)\big)}{\, \mathrm{d}t}.
	\end{equation*}
	Thus, we get $\frac{\partial \mathbf{E}}{\partial t}=\tilde{\mathbf{V}}_{1}$. 
	Similarly, we can prove $\frac{\partial \mathbf{H}}{\partial t}=\tilde{\mathbf{V}}_{2}$ and $\frac{\partial p}{\partial t}=\tilde V_{3}$.
	
	For any given $\mathbf{\Psi}_{2}\in L^{2}(0,T;\bm{C}^{1}_{0}(\Omega))$, integrating by parts leads to
	\begin{equation*}
		\int_{0}^{T}{\big(\nabla p_{m}(t), \mathbf{\Psi}_{2}(t)\big)}{\, \mathrm{d}t} = -\int_{0}^{T}{\big(p_{m}(t), \nabla\cdot \mathbf{\Psi}_{2}(t)\big)}{\, \mathrm{d}t}.
	\end{equation*}
	Let $m\to \infty$, we obtain
	\begin{equation*}
		\int_{0}^{T}{(\tilde{\mathbf{q}}(t), \mathbf{\Psi}_{2}(t))}{\, \mathrm{d}t} = -\int_{0}^{T}{(p(t), \nabla\cdot \mathbf{\Psi}_{2}(t))}{\, \mathrm{d}t},
	\end{equation*}
	which implies $\nabla p=\tilde{\mathbf{q}}$. In a similar way, we can prove $\nabla\times \mathbf{E}=\tilde{\mathbf{W}}$ and $\nabla\times \mathbf{H}=\tilde{\mathbf{M}}$.
	
	For any given $\mathbf{\Psi}_{3}\in L^{2}(0,T;\bm{C}^{1}(\bar\Omega))$, integrating by parts gives
	\begin{equation*}
		\int_{0}^{T}{\big(\nabla p_{m}(t), \mathbf{\Psi}_{3}(t)\big)}{\, \mathrm{d}t} = 
		- \int_{0}^{T}{\big(p_{m}(t), \nabla\cdot \mathbf{\Psi}_{3}(t)\big)}{\, \mathrm{d}t}.
	\end{equation*}
	By Green's formula, we get
	\begin{align*}
		& \int_{0}^{T}{\big\langle p(t), \mathbf{\Psi}_{3}(t)\cdot \mathbf{n}\big\rangle_{\partial\Omega}}{\, \mathrm{d}t}
		= \int_{0}^{T}{\big(p(t), \nabla\cdot \mathbf{\Psi}_{3}(t)\big)}{\, \mathrm{d}t}+\int_{0}^{T}{\big(\nabla p(t), \mathbf{\Psi}_{3}(t)\big)}{\, \mathrm{d}t}
		\\
		=& \lim_{m\to \infty} \Big( \int_{0}^{T}{\big(p_{m}(t), \nabla\cdot\! \mathbf{\Psi}_{3}(t)\big)}{\, \mathrm{d}t} 
		+ \int_{0}^{T}{\big(\nabla p_{m}(t), \mathbf{\Psi}_{3}(t)\big)}{\, \mathrm{d}t} \Big) 
		= 0,
	\end{align*}
	which implies $p\in L^{2}(0,T;H^{1}_{0})$ due to trace theory \cite{WuYiWa2006}. Similarly, we obtain $\mathbf{E}\in L^{2}(0,T;\bm{H}_{0}(\mathbf{curl}))$.
\end{proof}

We are now ready to establish the well-posedness of \eqref{e3.2.1}--\eqref{e3.2.3}. 

\begin{theorem}\label{t3.3.5}
	Assume (H1)--(H3). Then there exists a unique weak solution $(\mathbf{E},\mathbf{H},\mathbf{u},p)\in L^{2}(0,T;\bar{\mathbb{S}}) \cap H^{1}(0,T;\bar{\mathbb{H}})$ to \eqref{e3.2.1}--\eqref{e3.2.3}. Furthermore, there exists a constant $C=C(T)>0$ such that
	\begin{align*}
		&\ \|\mathbf{E}(t)\| + \|\mathbf{H}(t)\| + \|\mathbf{u}(t)\|_{1} + \|p(t)\| \leq C,
		\\
		&\ \|\frac{\partial}{\partial t} \mathbf{E}(t)\| + \|\frac{\partial}{\partial t} \mathbf{H}(t)\| + \|\frac{\partial}{\partial t} p(t)\| + \|\nabla \times \mathbf{E}(t)\| \leq C.
	\end{align*}
\end{theorem}

\begin{proof}
	Fixing an arbitrary integer $m_{0}>0$, then for any $m\geq m_{0}$, we have $\mathbf{V}_{m_{0}}\subset \mathbf{V}_{m}$.  Taking $(\mathbf{D},\mathbf{B},q)=(\mathbf{\Phi}_{m_{0}1},\mathbf{\Phi}_{m_{0}2},\Phi_{m_{0}3})\in \mathbf{V}_{m_0}\subset\mathbf{V}_{m}$ in \eqref{e3.3.1} and letting $m\to \infty$, we obtain
	\begin{align*}
		& \big(\epsilon \frac{\partial}{\partial t}\mathbf{E}(t),\mathbf{\Phi}_{m_{0}1}\big) + \big(\sigma \mathbf{E}(t),\mathbf{\Phi}_{m_{0}1}\big) - \big(\nabla\times \mathbf{H}(t),\mathbf{\Phi}_{m_{0}1}\big) - \big(L\nabla p(t),\mathbf{\Phi}_{m_{0}1}\big)
		\\
		& \quad =  \big(\mathbf{j}(t),\mathbf{\Phi}_{m_{0}1}\big),
		\\
		& \big(\mu \frac{\partial}{\partial t}\mathbf{H}(t),\mathbf{\Phi}_{m_{0}2}\big) + \big(\nabla\times \mathbf{E}(t),\mathbf{\Phi}_{m_{0}2}\big) = 0,
		\\
		& \big(\frac{\partial }{\partial t}(c_{0}\!+\!\mathcal{B})p(t),\Phi_{m_{0}3}\big) + \big(\kappa\nabla p(t),\nabla \Phi_{m_{0}3}\big) - \big(L \mathbf{E}(t),\nabla \Phi_{m_{0}3}\big) = \big(g(t),\Phi_{m_{0}3}\big).
	\end{align*}
	Due to the arbitrariness of $m_{0}$, $(\mathbf{E},\mathbf{H},p)$ is a solution to \eqref{e3.2.8}, which together with \eqref{e3.2.6} implies that $(\mathbf{E},\mathbf{H},\mathbf{u},p)$ is the unique weak solution to \eqref{e3.2.1}--\eqref{e3.2.3}.
\end{proof}

\subsection{Regularity estimates}\label{s3.3.3}
In the forthcoming analysis for  convergence order of approximate solutions, we require additional regularity for $(\mathbf{E},\mathbf{H},\mathbf{u},p)$ to \eqref{e3.2.1}--\eqref{e3.2.3}. 
In general, the smoothness of solution is determined by functions $\mathbf{j}$, $g$ and $(\mathbf{E}_{0},\mathbf{H}_{0},\mathbf{u}_{0},p_{0})$. Therefore, instead of assumption (H2), we assume
\begin{itemize}
	\item [\bf{(H2$^{\prime}$)}] The functions $\mathbf{j}$, $g$ and $(\mathbf{E}_{0},\mathbf{H}_{0},\mathbf{u}_{0},p_{0})$ are sufficiently smooth.
\end{itemize}

\begin{lemma}\label{l3.3.6}
	Assume (H1) and (H2$^{\prime}$). Let $(\mathbf{E},\mathbf{H},\mathbf{u},p)$ be a solution to \eqref{e3.2.4}. Then there exists a constant $C>0$ such that
	\begin{align*}
		& \|\nabla\cdot\mathbf{E}(t)\|^{2} + \|\nabla\cdot\mathbf{H}(t)\|^{2} + \|\nabla p(t)\|^{2} + \|\nabla \times \mathbf{H}(t)\|^{2}
		\\
		&\ + \int_{0}^{t}{\|\frac{\partial}{\partial t} \nabla\cdot \mathbf{u}(\theta)\|^{2}}{\, \mathrm{d}\theta}
		+ \int_{0}^{t}{\|\frac{\partial}{\partial t}p(\theta)\|^{2}}{\, \mathrm{d}\theta}
		+ \int_{0}^{t}{\|\Delta p(\theta)\|^{2}}{\, \mathrm{d}\theta} \leq C.
	\end{align*}
\end{lemma}

\begin{proof}
	According to the second equation in \eqref{e3.2.8}, we obtain a differential equation in $\bm{L}^{2}(\Omega)$
	\begin{equation*}
		\mu\frac{\partial}{\partial t}\mathbf{H}(t) + \nabla\times \mathbf{E}(t) = 0.
	\end{equation*}
	Taking divergence on both sides of this equation yields $\frac{\partial}{\partial t} \nabla\cdot \mathbf{H}(t)=0$, which leads to 
	\begin{equation*}
		\|\nabla\cdot\mathbf{H}(t)\|^{2} = \|\nabla\cdot\mathbf{H}_{0}\|^{2} = C.
	\end{equation*}
	Similarly, we obtain two differential equations in $\bm{L}^{2}(\Omega)$ and $L^{2}(\Omega)$, respectively,
	\begin{align}
		\label{e3.3.6}
		&\ \epsilon\frac{\partial}{\partial t}\mathbf{E}(t) + \sigma \mathbf{E}(t) - \nabla\times \mathbf{H}(t) = \mathbf{j}(t) + L\nabla p(t),
		\\
		\label{e3.3.7}
		&\ \frac{\partial}{\partial t}(c_{0}+\mathcal{B})p(t) - \kappa\Delta p(t) = g(t)-L\nabla\cdot \mathbf{E}(t).
	\end{align}
	Taking divergence on both sides of \eqref{e3.3.6}, then taking inner products with $\nabla\cdot\mathbf{E}(t)$ and utilizing Young's inequality, we get
	\begin{align*}
		&\ \epsilon\frac{1}{2}\frac{\mathrm{d}}{\mathrm{d}t}\|\nabla\cdot\mathbf{E}(t)\|^{2}
		+\sigma\|\nabla\cdot\mathbf{E}(t)\|^{2}
		\\
		\leq&\ \frac{1}{4\sigma}\|\nabla\cdot\mathbf{j}(t)\|^{2}+\sigma\|\nabla\cdot\mathbf{E}(t)\|^{2}
		+\frac{L^{2}}{2\sigma}\|\Delta p(t)\|^{2}+\frac{\sigma}{2}\|\nabla\cdot\mathbf{E}(t)\|^{2}.
	\end{align*}
	Integrating over $[0,t]$ yields
	\begin{equation}\label{e3.3.8}
		\begin{aligned}
			&\ \epsilon\|\nabla\cdot\mathbf{E}(t)\|^{2}
			\leq \epsilon\|\nabla\cdot\mathbf{E}_{0}\|^{2}
			+\frac{1}{2\sigma}\int_{0}^{t}{\|\nabla\cdot\mathbf{j}(\theta)\|^{2}}{\, \mathrm{d}\theta}
			\\
			&\qquad\qquad\qquad +\frac{L^{2}}{\sigma}\int_{0}^{t}{\|\Delta p(\theta)\|^{2}}{\, \mathrm{d}\theta}+\sigma\int_{0}^{t}{\|\nabla\cdot\mathbf{E}(\theta)\|^{2}}{\, \mathrm{d}\theta}.
		\end{aligned}
	\end{equation}
	From \eqref{e3.3.7}, it follows that
	\begin{equation*}
		\|\frac{\partial}{\partial t}(c_{0}+\mathcal{B})p(t) - \kappa\Delta p(t)\|^{2}
		\leq 2\|g(t)\|^{2}+2L^{2}\|\nabla\cdot \mathbf{E}(t)\|^{2}.
	\end{equation*}
	In terms of \eqref{e3.2.7}, direct computation yields
	\begin{align*}
		&\ \|\frac{\partial}{\partial t}(c_{0}+\mathcal{B})p(t) - \kappa\Delta p(t)\|^{2}
		\\
		=&\ \kappa\frac{\mathrm{d}}{\mathrm{d}t}\big((c_{0}+\mathcal{B})\nabla p(t),\nabla p(t)\big) + \|(c_{0}+\mathcal{B})\frac{\partial}{\partial t}p(t)\|^{2} + \kappa^{2}
		\|\Delta p(t)\|^{2}
		\\
		\geq&\ \kappa\frac{\mathrm{d}}{\mathrm{d}t}\big((c_{0}+\mathcal{B})\nabla p(t),\nabla p(t)\big) + c_{0}^{2} \|\frac{\partial}{\partial t}p(t)\|^{2} + \kappa^{2}
		\|\Delta p(t)\|^{2}.
	\end{align*}
	Integrating the above two inequalities over $[0,t]$, we have
	\begin{equation}\label{e3.3.9}
		\begin{aligned}
			&\ \big((c_{0}+\mathcal{B})\nabla p(t),\nabla p(t)\big) \!+\! \frac{c_{0}^{2}}{\kappa}\int_{0}^{t}{\|\frac{\partial}{\partial t}p(\theta)\|^{2}}{\, \mathrm{d}\theta} \!+\! \kappa\int_{0}^{t}{\|\Delta p(\theta)\|^{2}}{\, \mathrm{d}\theta}
			\\
			\leq&\ \big((c_{0}+\mathcal{B})\nabla p_{0},\nabla p_{0}\big) \!+\! \frac{2}{\kappa}\int_{0}^{t}{\|g(\theta)\|^{2}}{\, \mathrm{d}\theta} \!+\! \frac{2L^{2}}{\kappa}\int_{0}^{t}{\|\nabla\cdot \mathbf{E}(\theta)\|^{2}}{\, \mathrm{d}\theta}.
		\end{aligned}
	\end{equation}
	Summing up \eqref{e3.3.8} and \eqref{e3.3.9} gives
	\begin{align*}
		&\ \epsilon\|\nabla\!\cdot\!\mathbf{E}(t)\|^{2} \!+\! c_{0}\|\nabla p(t)\|^{2} \!+\! \frac{c_{0}^{2}}{\kappa}\int_{0}^{t}{\|\frac{\partial}{\partial t}p(\theta)\|^{2}}{\, \mathrm{d}\theta}
		\!+\! (\kappa-\frac{L^{2}}{\sigma})\!\int_{0}^{t}{\|\Delta p(\theta)\|^{2}}{\, \mathrm{d}\theta}
		\\
		\leq&\ \epsilon\|\nabla\cdot\mathbf{E}_{0}\|^{2}+C\|\nabla p_{0}\|^{2}
		+\frac{1}{2\sigma}\int_{0}^{t}{\|\nabla\cdot\mathbf{j}(\theta)\|^{2}}{\, \mathrm{d}\theta}+\frac{2}{\kappa}\int_{0}^{t}{\|g(\theta)\|^{2}}{\, \mathrm{d}\theta}
		\\
		&\ +(\sigma+\frac{2L^{2}}{\kappa})\int_{0}^{t}{\|\nabla\cdot\mathbf{E}(\theta)\|^{2}}.
	\end{align*}
	From assumption (H1), it follows that $\kappa-\frac{L^{2}}{\sigma}> 0$. Then, applying Gronwall's lemma, we conclude that  there exists a constant $C>0$ such that
	\begin{equation*}
		\|\nabla\cdot\mathbf{E}(t)\|^{2}+\|\nabla p(t)\|^{2} \leq C.
	\end{equation*}
	Thus, the rest of proof follows from this estimate immediately.
\end{proof}

Differentiating with respect to $t$ on both sides of \eqref{e3.2.4}, we get an equation similar to \eqref{e3.2.4} with unknowns $(\frac{\partial\mathbf{E}}{\partial t},\frac{\partial\mathbf{H}}{\partial t},\frac{\partial\mathbf{u}}{\partial t},\frac{\partial p}{\partial t})$. Under assumptions (H1) and (H2$^{\prime}$), we apply Theorem \ref{t3.3.5} and Lemma \ref{l3.3.6} to obtain estimates of higher order derivatives of the solution $(\mathbf{E},\mathbf{H},\mathbf{u},p)$ to \eqref{e3.2.4} with respect to $t$. We can repeat this procedure to raise the regularity of solution. In next lemma, we present the higher order regularity estimates and omit the proof.

\begin{lemma}\label{l3.3.7}
	Assume (H1) and (H2$^{\prime}$). Let $(\mathbf{E},\mathbf{H},\mathbf{u},p)$ be a solution to \eqref{e3.2.4}. Then there exists a constant $C>0$ such that for all $t\in[0,T]$,
	\begin{align*}
		\|\mathbf{E}(t)\|_{2}^{2}+\|\mathbf{H}(t)\|_{2}^{2}+\|p(t)\|_{4}^{2} \leq&\ C,
		\\
		\int_{0}^{t}{\!\|\frac{\partial}{\partial t}\mathbf{E}(\theta)\|_{2}^{2}}{\, \mathrm{d}\theta}
		\!+\!\! \int_{0}^{t}{\!\|\frac{\partial}{\partial t}\mathbf{H}(\theta)\|_{2}^{2}}{\, \mathrm{d}\theta}
		\!+\!\! \int_{0}^{t}{\!\|\frac{\partial}{\partial t}\mathbf{u}(\theta)\|_{3}^{2}}{\, \mathrm{d}\theta}
		\!+\!\! \int_{0}^{t}{\!\|\frac{\partial}{\partial t}p(\theta)\|_{2}^{2}}{\, \mathrm{d}\theta}
		\leq&\ C,
		\\
		\int_{0}^{t}{\|\nabla \!\times\! \frac{\partial}{\partial t}\mathbf{E}(\theta)\|^{2}}{\, \mathrm{d}\theta}
		\!+\! \int_{0}^{t}{\|\nabla \!\times\! \frac{\partial}{\partial t}\mathbf{H}(\theta)\|^{2}}{\, \mathrm{d}\theta}
		\!+\! \int_{0}^{t}{\|\frac{\partial^{2}}{\partial t^{2}}p(\theta)\|^{2}}{\, \mathrm{d}\theta}
		\leq&\ C.
	\end{align*}
\end{lemma}

\section{Splitting approximation and convergence analysis}\label{s3.4}
In this section, we introduce a splitting approximate solution to \eqref{e3.2.4} and derive a priori estimates. Then, we prove that the splitting approximation converges to the exact solution with first-order in time.

\subsection{Splitting approximate solution}\label{s3.4.1}
For any integer $N>0$, let $\tau=T/N$ be the temporal step-size and $t_{n}=n\tau$ $(0\leq n\leq N)$. Denote by $I_{n}:=(t_{n-1},t_{n}]$ $(1\leq n\leq N)$. We define the splitting approximate solution $(\tilde{\mathbf{E}}^{n},\tilde{\mathbf{H}}^{n},\tilde{\mathbf{u}}^{n},\tilde{p}^{n})$ $(0\leq n\leq N)$ to \eqref{e3.2.4} as follows. 

We first split \eqref{e3.2.8} into two set of variational equations: seek 
$(\mathbf{E}_{1},\mathbf{H}_{1},p_{1}) \in L^{2}(0,T;\mathbb{S}) \cap H^{1}(0,T;\mathbb{H})$ such that for all $t\in[0, T]$ and $(\mathbf{D},\mathbf{B},q)\in \mathbb{S}$,

\begin{equation}\label{e3.4.3}
	\begin{aligned}
		& \big(\epsilon\frac{\partial}{\partial t}\mathbf{E}_{1}(t), \mathbf{D}\big) \!\!+\!\! \big(\sigma \mathbf{E}_{1}(t), \mathbf{D}\big) \!\!-\!\! \big(\nabla \!\times\! \mathbf{H}_{1}(t), \mathbf{D}\big) \!\!-\!\! \big(L\nabla p_{1}(t), \mathbf{D}\big) \!=\! \big(\mathbf{j}(t), \mathbf{D}\big),
		\\
		& \big(\mu\frac{\partial}{\partial t}\mathbf{H}_{1}(t), \mathbf{B}\big) + \big(\nabla\times \mathbf{E}_{1}(t), \mathbf{B}\big) = 0,
		\\
		& \big(\frac{\partial}{\partial t}(c_{0} + \mathcal{B})p_{1}(t), q\big) = 0,
	\end{aligned}
\end{equation}
 and seek
$(\mathbf{E}_{2},\mathbf{H}_{2},p_{2}) \in L^{2}(0,T;\mathbb{S}) \cap H^{1}(0,T;\mathbb{H})$ such that for all $t\in[0, T]$ and $(\mathbf{D},\mathbf{B},q)\in \mathbb{S}$,
\begin{equation}\label{e3.4.4}
	\begin{aligned}
		& \big(\epsilon\frac{\partial}{\partial t}\mathbf{E}_{2}(t), \mathbf{D}\big) = 0,
		\\
		& \big(\mu\frac{\partial}{\partial t}\mathbf{H}_{2}(t), \mathbf{B}\big) = 0,
		\\
		& \big(\frac{\partial}{\partial t}(c_{0} + \mathcal{B})p_{2}(t), q\big) + \big(\kappa\nabla p_{2}(t), \nabla q\big) - \big(L\mathbf{E}_{2}(t), \nabla q\big) = \big(g(t), q\big).
	\end{aligned}
\end{equation}

Next, we define the splitting approximation $\tilde{\mathbf{U}}^{n}:=(\tilde{\mathbf{E}}^{n},\tilde{\mathbf{H}}^{n},\tilde{p}^{n})$ to \eqref{e3.2.8} by induction. Take $\tilde{\mathbf{U}}^{0} = (\mathbf{E}_{0},\mathbf{H}_{0},p_{0})$ and assume that we have already known approximations $\tilde{\mathbf{U}}^{n-1}$ for $1\leq n\leq N$. We define the next approximation $\tilde{\mathbf{U}}^{n}$  using  the following splitting procedure: 
\begin{enumerate}
	\item[(i)] solve \eqref{e3.4.3} with initial value $\tilde{\mathbf{U}}^{n-1}$ at $t_{n-1}$ to get a solution \begin{equation*}
		\mathbf{U}_{1\tau}(t):=(\mathbf{E}_{1\tau}(t),\mathbf{H}_{1\tau}(t),p_{1\tau}(t)), \quad t\in I_{n},
	\end{equation*}
	\item[(ii)] solve \eqref{e3.4.4} with initial value $\mathbf{U}_{1\tau}(t_{n})$ at $t_{n-1}$ to get a solution
	\begin{equation*}
		\mathbf{U}_{2\tau}(t):=(\mathbf{E}_{2\tau}(t),\mathbf{H}_{2\tau}(t),p_{2\tau}(t)), \quad t\in I_{n},
	\end{equation*}
	\item[(iii)] define $\tilde{\mathbf{U}}^{n} = \mathbf{U}_{2\tau}(t_{n})$.
\end{enumerate}
From \eqref{e3.4.3} and \eqref{e3.4.4}, it follows that $p_{1\tau}(t)\equiv \tilde{p}^{n-1}$, $\mathbf{E}_{2\tau}(t)\equiv \mathbf{E}_{1\tau}(t_{n})$ and $\mathbf{H}_{2\tau}(t)\equiv \mathbf{H}_{1\tau}(t_{n})$ on each $I_{n}$. Furthermore, the functions $\mathbf{E}_{1\tau}(t)$, $\mathbf{H}_{1\tau}(t)$ and $p_{2\tau}(t)$ are continuous in $t\in[0,T]$ and they interpolate the splitting approximate solution $\tilde{\mathbf{U}}^{n}$ at $t_{n}$, i.e., 
\begin{equation*}
	\tilde{\mathbf{U}}^{n} = (\mathbf{E}_{1\tau}(t_{n}),\mathbf{H}_{1\tau}(t_{n}),p_{2\tau}(t_{n})).
\end{equation*}

Finally, we solve \eqref{e3.2.6} with $p=p_{2\tau}(t)$ and obtain a unique continuous function $\mathbf{u}_{2\tau}(t)$ for $t\in[0,T]$. Let $\tilde{\mathbf{u}}^{0}=\mathbf{u}_{0}$ and $\tilde{\mathbf{u}}^{n}=\mathbf{u}_{2\tau}(t_{n})$, then we get a splitting approximate solution $(\tilde{\mathbf{E}}^{n},\tilde{\mathbf{H}}^{n},\tilde{\mathbf{u}}^{n},\tilde{p}^{n})$ $(0\leq n\leq N)$ to \eqref{e3.2.4}.

For the convenience of forthcoming analysis, we use the terminology of solution operators to reformulate the splitting approximate solution. For any $0\leq s\leq t\leq T$, let $\mathbf{R}_{s}^{t}:=\big((\mathcal{E}_{1})_{s}^{t},(\mathcal{H}_{1})_{s}^{t},(\mathcal{P}_{1})_{s}^{t}\big)$ with $\mathbf{R}_{s}^{s}=\mathbf{I}$ and $\mathbf{\Gamma}_{s}^{t}:=\big((\mathcal{E}_{2})_{s}^{t},(\mathcal{H}_{2})_{s}^{t},(\mathcal{P}_{2})_{s}^{t}\big)$ with $\mathbf{\Gamma}_{s}^{s}=\mathbf{I}$ be the solution operators to \eqref{e3.4.3} and \eqref{e3.4.4}, respectively, where $\mathbf{I}$ is the identity operator. Obviously, all of $(\mathcal{P}_{1})_{s}^{t}$, $(\mathcal{E}_{2})_{s}^{t}$ and $(\mathcal{H}_{2})_{s}^{t}$ are identity operators. Therefore, for any initial value $\mathbf{U}_{0}$ at time $s$, the functions
\begin{equation*}
	\mathbf{R}_{s}^{t}\mathbf{U}_{0} = \big((\mathcal{E}_{1})_{s}^{t}\mathbf{U}_{0},(\mathcal{H}_{1})_{s}^{t}\mathbf{U}_{0},p_{0}\big),
	\quad
	\mathbf{\Gamma}_{s}^{t}\mathbf{U}_{0} = \big(\mathbf{E}_{0},\mathbf{H}_{0},(\mathcal{P}_{2})_{s}^{t}\mathbf{U}_{0}\big).
\end{equation*}
are the unique solutions to \eqref{e3.4.3} and \eqref{e3.4.4}, respectively. 
For any $t>0$, there exists a unique integer $1\leq n\leq N$ such that $t\in I_{n}$ and then we have
\begin{equation*}
	\mathbf{U}_{1\tau}(t) = \mathbf{R}_{t_{n-1}}^{t}\tilde{\mathbf{U}}^{n-1},
	\
	\mathbf{U}_{2\tau}(t) = \mathbf{\Gamma}_{t_{n-1}}^{t}\mathbf{R}_{t_{n-1}}^{t_{n}}\tilde{\mathbf{U}}^{n-1}.
\end{equation*}
Then, the splitting approximate solution $\tilde{\mathbf{U}}^{n}$ can be rewritten as
\begin{equation*}
	\tilde{\mathbf{U}}^{0} = (\mathbf{E}_{0},\mathbf{H}_{0},p_{0}),
	\qquad \tilde{\mathbf{U}}^{n} = \mathbf{\Gamma}_{t_{n-1}}^{t_{n}}\mathbf{R}_{t_{n-1}}^{t_{n}}\tilde{\mathbf{U}}^{n-1}, \quad 1\leq n\leq N.
\end{equation*}

In fact, on each $I_{n}$, the function $(\mathbf{E}_{1\tau}(t),\mathbf{H}_{1\tau}(t))$ satisfies a variational equation with initial value $(\tilde{\mathbf{E}}^{n-1},\tilde{\mathbf{H}}^{n-1})$ at $t_{n-1}$ for all $(\mathbf{D},\mathbf{B})\in \bm{H}_{0}(\mathbf{curl}) \times \bm{H}(\mathbf{curl})$,
\begin{equation}\label{e3.4.5}
	\begin{aligned}
		& \big(\epsilon\frac{\partial}{\partial t}\mathbf{E}_{1\tau}(t),\mathbf{D}\big) 
		\!\!+\!\! \big(\sigma \mathbf{E}_{1\tau}(t),\mathbf{D}\big) 
		\!\!-\!\! \big(\nabla\!\!\times\!\! \mathbf{H}_{1\tau}(t),\mathbf{D}\big) 
		\!\!=\!\! \big(L\nabla \tilde{p}^{n-1},\mathbf{D}\big) \!\!+\!\! \big(\mathbf{j}(t),\mathbf{D}\big),
		\\
		& \big(\mu\frac{\partial}{\partial t}\mathbf{H}_{1\tau}(t),\mathbf{B}\big) 
		+ \big(\nabla\times \mathbf{E}_{1\tau}(t),\mathbf{B}\big) = 0, 
	\end{aligned}
\end{equation}
and  $p_{2\tau}(t)$ satisfies a variational equation with initial value $\tilde{p}^{n-1}$ at $t_{n-1}$:  for all $q\in H^{1}_{0}$,
\begin{equation}\label{e3.4.7}
	\big(\frac{\partial}{\partial t}(c_{0}+\mathcal{B})p_{2\tau}(t),q\big) + \big(\kappa\nabla p_{2\tau}(t),\nabla q\big) = \big(L\mathbf{E}_{1\tau}(t_{n}),\nabla q\big) + \big(g(t),q\big).
\end{equation}
Thus, the function $(\mathbf{u}_{2\tau}(t),p_{2\tau}(t))$ satisfies a variational equation with initial value $(\tilde{\mathbf{u}}^{n-1},\tilde{p}^{n-1})$ at $t_{n-1}$ for all $(\mathbf{v},q)\in \bm{H}^{1}_{0}\times H^{1}_{0}$,
\begin{equation}\label{e3.4.6}
	\begin{aligned}
		& (\lambda_{c}\nabla\cdot \mathbf{u}_{2\tau}(t), \nabla\cdot \mathbf{v}) + (G\nabla \mathbf{u}_{2\tau}(t), \nabla \mathbf{v}) - (p_{2\tau}(t), \alpha \nabla\cdot \mathbf{v}) = 0,
		\\
		& (\frac{\partial}{\partial t}(c_{0}p_{2\tau}(t) \!+\! \alpha\nabla \!\cdot\! \mathbf{u}_{2\tau}(t)), q) \!\!+\!\! (\kappa\nabla p_{2\tau}(t), \nabla q) \!\!-\!\! (L\mathbf{E}_{1\tau}(t_{n}), \nabla q) \!\!=\!\! (g(t), q).
	\end{aligned}
\end{equation}

\subsection{A priori estimates}\label{s3.4.2} 
In this part, we investigate the regularity of solutions to \eqref{e3.4.5} and \eqref{e3.4.6}.
We start with the estimates for solutions to \eqref{e3.4.3} and \eqref{e3.4.4}.

\begin{lemma}\label{l3.4.1}
	Assume (H2$^{\prime}$). Let $(\mathbf{E}_{1},\mathbf{H}_{1},p_{1})$ and $(\mathbf{E}_{2},\mathbf{H}_{2},p_{2})$ be solutions to \eqref{e3.4.3} and \eqref{e3.4.4}, respectively. Then for any $0\leq s\leq t\leq T$, there hold that
	\begin{align*}
		(i)\ &\ \epsilon\|\mathbf{E}_{1}(t)\|^{2}+\mu\|\mathbf{H}_{1}(t)\|^{2}
		\leq \epsilon\|\mathbf{E}_{1}(s)\|^{2}+ \mu\|\mathbf{H}_{1}(s)\|^{2}
		\\
		+&\ \frac{L^{2}}{\sigma}\int_{s}^{t}{\|\nabla p_{1}(\theta)\|^{2}}{\, \mathrm{d}\theta}
		+\sigma\int_{s}^{t}{\|\mathbf{E}_{1}(\theta)\|^{2}}{\, \mathrm{d}\theta}+\frac{1}{2\sigma}\int_{s}^{t}{\|\mathbf{j}(\theta)\|^{2}}{\, \mathrm{d}\theta},
		\\
		(ii)\ &\ \big((c_{0}+\mathcal{B})p_{2}(t),p_{2}(t)\big)+\kappa\int_{s}^{t}{\|\nabla p_{2}(\theta)\|^{2}}{\, \mathrm{d}\theta}
		\leq \big((c_{0}+\mathcal{B})p_{2}(s),p_{2}(s)\big)
		\\
		+&\ \frac{L^{2}}{\kappa}\int_{s}^{t}{\|\mathbf{E}_{2}(\theta)\|^{2}}{\, \mathrm{d}\theta}+\int_{s}^{t}{\|g(\theta)\|^{2}}{\, \mathrm{d}\theta}+\int_{s}^{t}{\|p_{2}(\theta)\|^{2}}{\, \mathrm{d}\theta},
		\\
		(iii)\ &\ \epsilon\|\nabla\cdot\mathbf{E}_{1}(t)\|^{2}
		\leq \epsilon\|\nabla\cdot\mathbf{E}_{1}(s)\|^{2}
		+\frac{1}{2\sigma}\int_{s}^{t}{\|\nabla\cdot\mathbf{j}(\theta)\|^{2}}{\, \mathrm{d}\theta}
		\\
		+&\ \frac{L^{2}}{\sigma}\int_{s}^{t}{\|\Delta p_{1}(\theta)\|^{2}}{\, \mathrm{d}\theta}+\sigma\int_{s}^{t}{\|\nabla\cdot\mathbf{E}_{1}(\theta)\|^{2}}{\, \mathrm{d}\theta},
		\\
		(iv)\ &\ \big((c_{0}+\mathcal{B})\nabla p_{2}(t),\nabla p_{2}(t)\big) \!+\! \frac{c_{0}^{2}}{\kappa}\int_{s}^{t}{\|\frac{\partial}{\partial t}p_{2}(\theta)\|^{2}}{\, \mathrm{d}\theta} \!+\! \kappa\int_{s}^{t}{\|\Delta p_{2}(\theta)\|^{2}}{\, \mathrm{d}\theta}
		\\
		\leq&\ \big((c_{0}+\mathcal{B})\nabla p_{2}(s),\nabla p_{2}(s)\big) \!+\! \frac{2}{\kappa}\int_{s}^{t}{\|g(\theta)\|^{2}}{\, \mathrm{d}\theta} \!+\! \frac{2L^{2}}{\kappa}\int_{s}^{t}{\|\nabla\cdot \mathbf{E}_{2}(\theta)\|^{2}}{\, \mathrm{d}\theta}.
	\end{align*}
\end{lemma}

\begin{proof}
	Taking $(\mathbf{D},\mathbf{B}) = (\mathbf{E}_{1}(t),\mathbf{H}_{1}(t))$ in the first two equations in \eqref{e3.4.3} and applying Young's inequality, we have
	\begin{align*}
		&\ \frac{1}{2}\frac{\mathrm{d}}{\mathrm{d}t}\big(\epsilon\|\mathbf{E}_{1}(t)\|^{2}+ \mu\|\mathbf{H}_{1}(t)\|^{2}\big)+\sigma\|\mathbf{E}_{1}(t)\|^{2}
		\\
		\leq&\ \frac{L^{2}}{2\sigma}\|\nabla p_{1}(t)\|^{2}+\frac{\sigma}{2}\|\mathbf{E}_{1}(t)\|^{2}
		+\frac{1}{4\sigma}\|\mathbf{j}(t)\|^{2}+\sigma\|\mathbf{E}_{1}(t)\|^{2}.
	\end{align*}
	Integrating  over $[s,t]$ yields $(i)$.
	Taking $q_{2}=p_{2}(t)$ in the third equation in \eqref{e3.4.4} and utilizing Young's inequality, we get
	\begin{align*}
		& \frac{1}{2}\frac{\mathrm{d}}{\mathrm{d}t}\big((c_{0}+\mathcal{B})p_{2}(t),p_{2}(t)\big)+\kappa\|\nabla p_{2}(t)\|^{2}
		\\
		\leq&\ \frac{L^{2}}{2\kappa}\|\mathbf{E}_{2}(t)\|^{2}+\frac{\kappa}{2}\|\nabla p_{2}(t)\|^{2}+\frac{1}{2}\|g(t)\|^{2}+\frac{1}{2}\|p_{2}(t)\|^{2}.
	\end{align*}
	Integrating over $[s,t]$, we obtain $(ii)$. 
	The  estimates in $(iii)$ and $(iv)$ follow directly from \eqref{e3.3.8} and \eqref{e3.3.9}.
\end{proof}

Then, we obtain regularity estimates for solutions to \eqref{e3.4.5} and \eqref{e3.4.6}.

\begin{lemma}\label{l3.4.2}
	Assume (H1) and (H2$^{\prime}$). Let $(\mathbf{E}_{1\tau},\mathbf{H}_{1\tau})$ and $(\mathbf{u}_{2\tau},p_{2\tau})$ be solutions to \eqref{e3.4.5} and \eqref{e3.4.6}, respectively. Then there exists a constant $C>0$ independent of $\tau$ such that for all $t\in [0,T]$,
	\begin{equation*}
		\|\mathbf{E}_{1\tau}(t)\|^{2} + \|\mathbf{H}_{1\tau}(t)\|^{2} + \|\mathbf{u}_{2\tau}(t)\|_{1}^{2} + \|p_{2\tau}(t)\|^{2} + \int_{0}^{t}{\|\nabla p_{2\tau}(\theta)\|^{2}}{\, \mathrm{d}\theta} \leq C.
	\end{equation*}
\end{lemma}

\begin{proof}
	For any $t>0$, there exists a unique integer $1\leq n\leq N$ such that $t\in I_{n}$. Applying $(i)$ of Lemma \ref{l3.4.1} in $[t_{n-1},t]$ and noticing $p_{1\tau}(\theta) \equiv \tilde{p}^{n-1} \equiv p_{2\tau}(t_{n-1})$ for any $\theta\in I_{n}$, we have
	\begin{equation}\label{e3.4.8}
		\begin{aligned}
			&\ \epsilon\|\mathbf{E}_{1\tau}(t)\|^{2}+\mu\|\mathbf{H}_{1\tau}(t)\|^{2}
			\leq \epsilon\|\mathbf{E}_{1\tau}(t_{n-1})\|^{2}+ \mu\|\mathbf{H}_{1\tau}(t_{n-1})\|^{2}
			\\
			&\ +\frac{L^{2}}{\sigma}\tau\|\nabla p_{2\tau}(t_{n-1})\|^{2}
			+\sigma\int_{t_{n-1}}^{t}{\|\mathbf{E}_{1\tau}(\theta)\|^{2}}{\, \mathrm{d}\theta}+\frac{1}{2\sigma}\int_{t_{n-1}}^{t}{\|\mathbf{j}(\theta)\|^{2}}{\, \mathrm{d}\theta}.
		\end{aligned}
	\end{equation}
	Similarly, using $(ii)$ of Lemma \ref{l3.4.1} in $[t_{n-1},t]$ and noting $\mathbf{E}_{2\tau}(\theta) \equiv \mathbf{E}_{1\tau}(t_{n})$ for any $\theta\in I_{n}$, we get
	\begin{equation}\label{e3.4.9}
		\begin{aligned}
			&\ \big((c_{0} + \mathcal{B})p_{2\tau}(t),p_{2\tau}(t)\big) + \kappa\int_{t_{n-1}}^{t}{\|\nabla p_{2\tau}(\theta)\|^{2}}{\, \mathrm{d}\theta}
			\\
			\leq&\ \big((c_{0} + \mathcal{B})p_{2\tau}(t_{n-1}),p_{2\tau}(t_{n-1})\big) + \frac{L^{2}}{\kappa}\tau\|\mathbf{E}_{1\tau}(t_{n})\|^{2}
			\\
			&\ + \int_{t_{n-1}}^{t}{\|g(\theta)\|^{2}}{\, \mathrm{d}\theta} + \int_{t_{n-1}}^{t}{\|p_{2\tau}(\theta)\|^{2}}{\, \mathrm{d}\theta}.
		\end{aligned}
	\end{equation}
	For $1\leq i\leq n-1$, utilizing $(i)$ and $(ii)$ of Lemma \ref{l3.4.1} in $I_{i}$ and observing $p_{1\tau}(\theta) \equiv \tilde{p}^{i-1} \equiv p_{2\tau}(t_{i-1})$ and  $\mathbf{E}_{2\tau}(\theta) \equiv \mathbf{E}_{1\tau}(t_{i})$ for any $\theta\in I_{i}$, we have
	\begin{equation}\label{e3.4.10}
		\begin{aligned}
			&\ \epsilon\|\mathbf{E}_{1\tau}(t_{i})\|^{2}+\mu\|\mathbf{H}_{1\tau}(t_{i})\|^{2}
			\leq \epsilon\|\mathbf{E}_{1\tau}(t_{i-1})\|^{2}+ \mu\|\mathbf{H}_{1\tau}(t_{i-1})\|^{2}
			\\
			&\ \quad +\frac{L^{2}}{\sigma}\tau\|\nabla p_{2\tau}(t_{i-1})\|^{2}
			+\sigma\int_{t_{i-1}}^{t_{i}}{\|\mathbf{E}_{1\tau}(\theta)\|^{2}}{\, \mathrm{d}\theta}+\frac{1}{2\sigma}\int_{t_{i-1}}^{t_{i}}{\|\mathbf{j}(\theta)\|^{2}}{\, \mathrm{d}\theta},
			\\
			&\ \big((c_{0} + \mathcal{B})p_{2\tau}(t_{i}),p_{2\tau}(t_{i})\big) + \kappa\int_{t_{i-1}}^{t_{i}}{\|\nabla p_{2\tau}(\theta)\|^{2}}{\, \mathrm{d}\theta}
			\\
			\leq&\ \big((c_{0} + \mathcal{B})p_{2\tau}(t_{i-1}),p_{2\tau}(t_{i-1})\big) + \frac{L^{2}}{\kappa}\tau\|\mathbf{E}_{1\tau}(t_{i})\|^{2}
			\\
			&\ +\int_{t_{i-1}}^{t_{i}}{\|g(\theta)\|^{2}}{\, \mathrm{d}\theta}+\int_{t_{i-1}}^{t_{i}}{\|p_{2\tau}(\theta)\|^{2}}{\, \mathrm{d}\theta}.
		\end{aligned}
	\end{equation}
	Summing up \eqref{e3.4.10} from $i=n-1$ to $1$ and additionally adding \eqref{e3.4.8} and \eqref{e3.4.9}, we obtain
	\begin{equation}\label{e3.4.11}
		\begin{aligned}
			&\ \epsilon\|\mathbf{E}_{1\tau}(t)\|^{2} \!\!+\! \mu\|\mathbf{H}_{1\tau}(t)\|^{2}
			\!\!+\!\! \big((c_{0}\!+\!\mathcal{B})p_{2\tau}(t),p_{2\tau}(t)\big)
			\!\!+\! \kappa \! \int_{0}^{t}{\!\!\|\nabla p_{2\tau}(\theta)\|^{2}}{\, \mathrm{d}\theta}
			\\
			\leq&\ \epsilon\|\mathbf{E}_{0}\|^{2}+\mu\|\mathbf{H}_{0}\|^{2}+C\|p_{0}\|^{2}+\frac{1}{2\sigma}\int_{0}^{t}{\|\mathbf{j}(\theta)\|^{2}}{\, \mathrm{d}\theta}+\int_{0}^{t}{\|g(\theta)\|^{2}}{\, \mathrm{d}\theta}
			\\
			&\ +\frac{L^{2}}{\kappa}\sum_{i=1}^{n}\tau\|\mathbf{E}_{1\tau}(t_{i})\|^{2}
			+\frac{L^{2}}{\sigma}\sum_{i=1}^{n}\tau\|\nabla p_{2\tau}(t_{i-1})\|^{2}
			\\
			&\ +\sigma\int_{0}^{t}{\|\mathbf{E}_{1\tau}(\theta)\|^{2}}{\, \mathrm{d}\theta}+\int_{0}^{t}{\|p_{2\tau}(\theta)\|^{2}}{\, \mathrm{d}\theta}.
		\end{aligned}
	\end{equation}
	
	Due to the continuity of $\mathbf{E}_{1\tau}(\theta)$ in $\theta$, for sufficiently small $\tau$,
	\begin{align*}
		&\ \Big|\sum_{i=1}^{n}\tau\|\mathbf{E}_{1\tau}(t_{i})\|^{2}
		-\int_{0}^{t}{\|\mathbf{E}_{1\tau}(\theta)\|^{2}}{\, \mathrm{d}\theta}\Big|
		\leq
		\frac{1}{2}.
	\end{align*}
	According to (H2$^{\prime}$) and \cite[Theorem 6, p.386]{Ev1998}, we deduce that $\nabla p_{2\tau}(\theta)$ is continuous in $\theta$. In a similar way, we get
	\begin{equation*}
		\Big|\sum_{i=1}^{n}\tau\|\nabla p_{2\tau}(t_{i-1})\|^{2}-\int_{0}^{t}{\|\nabla p_{2\tau}(\theta)\|^{2}}{\, \mathrm{d}\theta}\Big| \leq \frac{1}{2}.
	\end{equation*}
	Substitute the above two inequalities into \eqref{e3.4.11}, we obtain
	\begin{align*}
		&\ \epsilon\|\mathbf{E}_{1\tau}(t)\|^{2}+\mu\|\mathbf{H}_{1\tau}(t)\|^{2}
		+c_{0}\|p_{2\tau}(t)\|^{2}+(\kappa-\frac{L^{2}}{\sigma}) \int_{0}^{t}{\|\nabla p_{2\tau}(\theta)\|^{2}}{\, \mathrm{d}\theta}
		\\
		\leq&\ \epsilon\|\mathbf{E}_{0}\|^{2}+\mu\|\mathbf{H}_{0}\|^{2}+C\|p_{0}\|^{2}+\frac{1}{2\sigma}\int_{0}^{t}{\|\mathbf{j}(\theta)\|^{2}}{\, \mathrm{d}\theta}+\int_{0}^{t}{\|g(\theta)\|^{2}}{\, \mathrm{d}\theta}
		\\
		&\ +\frac{L^{2}}{2\kappa} +\frac{L^{2}}{2\sigma} +(\sigma+\frac{L^{2}}{\kappa})\int_{0}^{t}{\|\mathbf{E}_{1\tau}(\theta)\|^{2}}{\, \mathrm{d}\theta}
		+\int_{0}^{t}{\|p_{2\tau}(\theta)\|^{2}}{\, \mathrm{d}\theta}.
	\end{align*}
	Assumption (H1) implies $\kappa-\frac{L^{2}}{\sigma}>0$, then according to Gronwall's lemma, there exists a constant $C>0$ independent of $\tau$ such that
	\begin{equation*}
		\|\mathbf{E}_{1\tau}(t)\|^{2}+\|\mathbf{H}_{1\tau}(t)\|^{2}+\|p_{2\tau}(t)\|^{2} \leq C.
	\end{equation*}
	
	From \eqref{e3.4.6} and the coerciveness of $a(\cdot,\cdot)$, it follows that
	\begin{align*}
		\beta_{1} \|\mathbf{u}_{2\tau}(t)\|_{1}^{2} \leq\! |a(\mathbf{u}_{2\tau}(t),\mathbf{u}_{2\tau}(t))| \!=\! |(p_{2\tau}(t), \alpha \nabla \!\cdot\! \mathbf{u}_{2\tau}(t))| \leq\! C\|p_{2\tau}(t)\|\|\mathbf{u}_{2\tau}(t)\|_{1},
	\end{align*}
	which implies $\|\mathbf{u}_{2\tau}(t)\|_{1} \leq C\|p_{2\tau}(t)\| \leq C$. Then, the rest of proof follows immediately.
\end{proof}

Applying the same argument as in proving Lemma \ref{l3.4.2}, we can estimate high order derivatives of solutions to \eqref{e3.4.5} and \eqref{e3.4.6}. Here, we present the estimates in next lemma and omit the proof.

\begin{lemma}\label{l3.4.3}
	Assume (H1) and (H2$^{\prime}$). Let $(\mathbf{E}_{1\tau},\mathbf{H}_{1\tau})$ and $(\mathbf{u}_{2\tau},p_{2\tau})$ be solutions to \eqref{e3.4.5} and \eqref{e3.4.6}, respectively. Then there exists a constant $C>0$ independent of $\tau$ such that for all $t\in [0,T]$,
	\begin{align*}
		\|\nabla\cdot\mathbf{E}_{1\tau}(t)\|^{2} + \|\nabla\cdot\mathbf{H}_{1\tau}(t)\|^{2} + \|\nabla p_{2\tau}(t)\|^{2} \leq&\ C,
		\\
		\int_{0}^{t}{\|\frac{\partial}{\partial t} \nabla\cdot \mathbf{u}_{2\tau}(\theta)\|^{2}}{\, \mathrm{d}\theta} 
		+ \int_{0}^{t}{\!\|\frac{\partial}{\partial t}p_{2\tau}(\theta)\|^{2}}{\, \mathrm{d}\theta}
		+ \int_{0}^{t}{\!\|\Delta p_{2\tau}(\theta)\|^{2}}{\, \mathrm{d}\theta} \leq&\ C.
	\end{align*}
\end{lemma}

\begin{remark}\label{r3.4.1}
	Take $t=t_{n}$ in above two lemmas, we get uniform estimates for the splitting approximate solution, i.e., for all $0\leq n\leq N$, there hold
	\begin{align*}
		\|\tilde{\mathbf{E}}^{n}\|^{2} + \|\tilde{\mathbf{H}}^{n}\|^{2} + \|\tilde{\mathbf{u}}^{n}\|_{1}^{2} + \|\tilde{p}^{n}\|^{2} \leq&\ C,
		\\
		\|\nabla\cdot\tilde{\mathbf{E}}^{n}\|^{2} + \|\nabla\cdot\tilde{\mathbf{H}}^{n}\|^{2} + \|\nabla \tilde{p}^{n}\|^{2} \leq&\ C.
	\end{align*}
\end{remark}

The next lemma provides more estimates for solutions to \eqref{e3.4.5} and \eqref{e3.4.6}.

\begin{lemma}\label{l3.4.4}
	Assume (H1) and (H2$^{\prime}$). Let $(\mathbf{E}_{1\tau},\mathbf{H}_{1\tau})$ and $(\mathbf{u}_{2\tau},p_{2\tau})$ be solutions to \eqref{e3.4.5} and \eqref{e3.4.6}, respectively. Then there exists a constant $C>0$ independent of $\tau$ such that for all $t\in [0,T]$,
	\begin{align*}
		\|\frac{\partial}{\partial t}\mathbf{E}_{1\tau}(t)\|^{2} + \|\frac{\partial}{\partial t}\mathbf{H}_{1\tau}(t)\|^{2} + \|\frac{\partial}{\partial t} \nabla\cdot \mathbf{u}_{2\tau}(t)\|^{2} + \|\frac{\partial}{\partial t}p_{2\tau}(t)\|^{2} \leq&\ C,
		\\
		\|\nabla\times \mathbf{E}_{1\tau}(t)\|^{2} + \|\nabla\times \mathbf{H}_{1\tau}(t)\|^{2} + \|\Delta p_{2\tau}(t)\|^{2} + \int_{0}^{t}{\|\nabla \frac{\partial}{\partial t}p_{2\tau}(\theta)\|^{2}}{\, \mathrm{d}\theta} \leq&\ C.
	\end{align*}
\end{lemma}

\begin{proof}
	For any $1\leq n\leq N$ and $t\in I_{n}$, taking derivative on both sides of \eqref{e3.4.5} and \eqref{e3.4.7} with respect to $t$,  we obtain that $(\frac{\partial}{\partial t}\mathbf{E}_{1\tau}(t),\frac{\partial}{\partial t}\mathbf{H}_{1\tau}(t),\frac{\partial}{\partial t}p_{2\tau}(t))$ $(t\in I_{n})$ satisfies for all $(\mathbf{D},\mathbf{B},q)\in \mathbb{S}$,
	\begin{align*}
		& \big(\epsilon\frac{\partial^{2}}{\partial t^{2}}\mathbf{E}_{1\tau}(t),\mathbf{D}\big) + \big(\sigma \frac{\partial}{\partial t}\mathbf{E}_{1\tau}(t),\mathbf{D}\big) - \big(\nabla\times \frac{\partial}{\partial t}\mathbf{H}_{1\tau}(t),\mathbf{D}\big) = \big(\mathbf{j}^{\prime}(t),\mathbf{D}\big),
		\\
		& \big(\mu\frac{\partial^{2}}{\partial t^{2}}\mathbf{H}_{1\tau}(t),\mathbf{B}\big) + \big(\nabla\times \frac{\partial}{\partial t}\mathbf{E}_{1\tau}(t),\mathbf{B}\big) = 0,
		\\
		& \big(\frac{\partial^{2}}{\partial t^{2}}(c_{0}+\mathcal{B})p_{2\tau}(t),q\big) + \big(\kappa\nabla \frac{\partial}{\partial t}p_{2\tau}(t),\nabla q\big) = \big(g^{\prime}(t),q\big),
	\end{align*}
	with initial value $(\frac{\partial}{\partial t} \mathbf{E}_{1\tau}(t_{n-1}),\frac{\partial}{\partial t} \mathbf{H}_{1\tau}(t_{n-1}),\frac{\partial}{\partial t} p_{2\tau}(t_{n-1}))$ at $t_{n-1}$. Then, choosing $(\mathbf{D},\mathbf{B},q) = (\frac{\partial}{\partial t}\mathbf{E}_{1\tau}(t),\frac{\partial}{\partial t}\mathbf{H}_{1\tau}(t),\frac{\partial}{\partial t}p_{2\tau}(t))$ in above equation leads to
	\begin{equation}\label{e3.4.12}
		\begin{aligned}
			&\ \frac{1}{2}\frac{\mathrm{d}}{\mathrm{d}t}\big(\epsilon\|\frac{\partial}{\partial t}\mathbf{E}_{1\tau}(t)\|^{2} + \mu\|\frac{\partial}{\partial t}\mathbf{H}_{1\tau}(t)\|^{2} + ((c_{0}+\mathcal{B})\frac{\partial}{\partial t}p_{2\tau}(t),\frac{\partial}{\partial t}p_{2\tau}(t))\big)
			\\
			&\ +\sigma\|\frac{\partial}{\partial t}\mathbf{E}_{1\tau}(t)\|^{2} + \kappa\|\nabla \frac{\partial}{\partial t}p_{2\tau}(t)\|^{2}
			\\
			\leq&\ \frac{1}{2}\|\mathbf{j}^{\prime}(t)\|^{2} + \frac{1}{2}\|g^{\prime}(t)\|^{2} + \frac{1}{2}\|\frac{\partial}{\partial t}\mathbf{E}_{1\tau}(t)\|^{2} + \frac{1}{2}\|\frac{\partial}{\partial t}p_{2\tau}(t)\|^{2}.
		\end{aligned}
	\end{equation}
	Integrating over $[t_{n-1},t]$ yields
	\begin{equation}\label{e3.4.13}
		\begin{aligned}
			&\ \epsilon\|\frac{\partial}{\partial t}\mathbf{E}_{1\tau}(t)\|^{2} + \mu\|\frac{\partial}{\partial t}\mathbf{H}_{1\tau}(t)\|^{2} + \big((c_{0}+\mathcal{B})\frac{\partial}{\partial t}p_{2\tau}(t),\frac{\partial}{\partial t}p_{2\tau}(t)\big)
			\\
			&\ + 2\kappa \int_{t_{n-1}}^{t}{\|\nabla \frac{\partial}{\partial t}p_{2\tau}(\theta)\|^{2}}{\, \mathrm{d}\theta}
			\\
			\leq&\ \epsilon\|\frac{\partial}{\partial t}\mathbf{E}_{1\tau}(t_{n-1})\|^{2} + \mu\|\frac{\partial}{\partial t}\mathbf{H}_{1\tau}(t_{n-1})\|^{2}
			\\
			&\ + \big((c_{0}+\mathcal{B})\frac{\partial}{\partial t}p_{2\tau}(t_{n-1}),\frac{\partial}{\partial t}p_{2\tau}(t_{n-1})\big) + \int_{t_{n-1}}^{t}{\|\mathbf{j}^{\prime}(\theta)\|^{2}}{\, \mathrm{d}\theta}
			\\
			&\ + \int_{t_{n-1}}^{t}{\|g^{\prime}(\theta)\|^{2}}{\, \mathrm{d}\theta} + \int_{t_{n-1}}^{t}{\|\frac{\partial}{\partial t}\mathbf{E}_{1\tau}(\theta)\|^{2}}{\, \mathrm{d}\theta} + \int_{t_{n-1}}^{t}{\|\frac{\partial}{\partial t}p_{2\tau}(\theta)\|^{2}}{\, \mathrm{d}\theta}.
		\end{aligned}
	\end{equation}
	For any $1\leq i\leq n-1$, integrating \eqref{e3.4.12} over $I_{i}$, we have
	\begin{equation}\label{e3.4.14}
		\begin{aligned}
			&\ \epsilon\|\frac{\partial}{\partial t}\mathbf{E}_{1\tau}(t_{i})\|^{2} + \mu\|\frac{\partial}{\partial t}\mathbf{H}_{1\tau}(t_{i})\|^{2} + \big((c_{0}+\mathcal{B})\frac{\partial}{\partial t}p_{2\tau}(t_{i}),\frac{\partial}{\partial t}p_{2\tau}(t_{i})\big)
			\\
			&\ + 2\kappa \int_{t_{i-1}}^{t_{i}}{\|\nabla \frac{\partial}{\partial t}p_{2\tau}(\theta)\|^{2}}{\, \mathrm{d}\theta}
			\\
			\leq&\ \epsilon\|\frac{\partial}{\partial t}\mathbf{E}_{1\tau}(t_{i-1})\|^{2} + \mu\|\frac{\partial}{\partial t}\mathbf{H}_{1\tau}(t_{i-1})\|^{2}
			\\
			&\ + \big((c_{0}+\mathcal{B})\frac{\partial}{\partial t}p_{2\tau}(t_{i-1}),\frac{\partial}{\partial t}p_{2\tau}(t_{i-1})\big) + \int_{t_{i-1}}^{t_{i}}{\|\mathbf{j}^{\prime}(\theta)\|^{2}}{\, \mathrm{d}\theta}
			\\
			&\ + \int_{t_{i-1}}^{t_{i}}{\|g^{\prime}(\theta)\|^{2}}{\, \mathrm{d}\theta} + \int_{t_{i-1}}^{t_{i}}{\|\frac{\partial}{\partial t}\mathbf{E}_{1\tau}(\theta)\|^{2}}{\, \mathrm{d}\theta} + \int_{t_{i-1}}^{t_{i}}{\|\frac{\partial}{\partial t}p_{2\tau}(\theta)\|^{2}}{\, \mathrm{d}\theta}.
		\end{aligned}
	\end{equation}
	Summing up \eqref{e3.4.14} from $i=n-1$ to $1$ and additionally adding \eqref{e3.4.13}, we get
	\begin{equation}\label{e3.4.15}
		\begin{aligned}
			&\, \epsilon\|\frac{\partial}{\partial t}\mathbf{E}_{1\tau}(t)\|^{2} \!\!+\! \mu\|\frac{\partial}{\partial t}\mathbf{H}_{1\tau}(t)\|^{2} \!\!+\! c_{0}\|\frac{\partial}{\partial t}p_{2\tau}(t)\|^{2} \!\!+\! 2\kappa \! \int_{0}^{t}{\!\|\nabla \frac{\partial}{\partial t}p_{2\tau}(\theta)\|^{2}}{\, \mathrm{d}\theta}
			\\
			\leq&\, \epsilon\|\frac{\partial}{\partial t}\mathbf{E}_{1\tau}(0)\|^{2} + \mu\|\frac{\partial}{\partial t}\mathbf{H}_{1\tau}(0)\|^{2} + C\|\frac{\partial}{\partial t}p_{2\tau}(0)\|^{2} + \int_{0}^{t}{\|\mathbf{j}^{\prime}(\theta)\|^{2}}{\, \mathrm{d}\theta}
			\\
			&\, + \int_{0}^{t}{\|g^{\prime}(\theta)\|^{2}}{\, \mathrm{d}\theta} + \int_{0}^{t}{\|\frac{\partial}{\partial t}\mathbf{E}_{1\tau}(\theta)\|^{2}}{\, \mathrm{d}\theta} + \int_{0}^{t}{\|\frac{\partial}{\partial t}p_{2\tau}(\theta)\|^{2}}{\, \mathrm{d}\theta}.
		\end{aligned}
	\end{equation}
	Taking $(\mathbf{D},\mathbf{B})=(\frac{\partial}{\partial t}\mathbf{E}_{1\tau}(t),\frac{\partial}{\partial t}\mathbf{H}_{1\tau}(t))$ in \eqref{e3.4.5} and $q=\frac{\partial}{\partial t}p_{2\tau}(t)$ in \eqref{e3.4.7}, letting $t=0$ and integrating by parts, we obtain
	\begin{align*}
		\epsilon \|\frac{\partial}{\partial t}\mathbf{E}_{1\tau}(0)\| \leq&\ \|\mathbf{j}(0)\| + \sigma \|\mathbf{E}_{0}\| + \|\nabla\times \mathbf{H}_{0}\| + L \|\nabla p_{0}\| \leq C,
		\\
		\mu \|\frac{\partial}{\partial t}\mathbf{H}_{1\tau}(0)\| \leq&\ \|\nabla\times \mathbf{E}_{0}\| \leq C,
		\\
		c_{0} \|\frac{\partial}{\partial t}p_{2\tau}(0)\| \leq&\ \|g(0)\|+\kappa\|\Delta p_{0}\|+L\|\nabla\cdot \tilde{\mathbf{E}}^{1}\| \leq C,
	\end{align*}
	where we have used Cauchy-Schwarz inequality and Remark \ref{r3.4.1}. Then, applying Gronwall's lemma to \eqref{e3.4.15}, we get there exists a constant $C>0$ independent of $\tau$ such that the first estimate in this lemma holds.
	Integrating by parts in \eqref{e3.4.7} gives
	\begin{equation*}
		\kappa\Delta p_{2\tau}(t) = (c_{0}+\mathcal{B})\frac{\partial}{\partial t}p_{2\tau}(t) + L\nabla\cdot\mathbf{E}_{1\tau}(t_{n}) - g(t) \text{ in } L^{2}(\Omega).
	\end{equation*}
	Combining with Lemma \ref{l3.4.3}, we have $\|\Delta p_{2\tau}(t)\|^{2} \leq C$. Then, it is easy to prove the rest estimates.
\end{proof}

We can repeat the procedure in the proof of the above lemma by taking derivative with respect to $t$ to  obtain higher order estimates for solutions to \eqref{e3.4.5} and \eqref{e3.4.6}. Here, we present the estimates in next lemma and omit thfe proof.

\begin{lemma}\label{l3.4.5}
	Assume (H1) and (H2$^{\prime}$). Let $(\mathbf{E}_{1\tau},\mathbf{H}_{1\tau})$ and $(\mathbf{u}_{2\tau},p_{2\tau})$ be solutions to \eqref{e3.4.5} and \eqref{e3.4.6}, respectively. Then there exists a constant $C>0$ independent of $\tau$ such that for all $t\in [0,T]$,
	\begin{align*}
		\|\mathbf{E}_{1\tau}(t)\|_{2}^{2}+\|\mathbf{H}_{1\tau}(t)\|_{2}^{2}+\|p_{2\tau}(t)\|_{4}^{2} & \leq C,
		\\
		\int_{0}^{t}{\!\|\frac{\partial}{\partial t}\mathbf{E}_{1\tau}(\theta)\|_{2}^{2}}{\, \mathrm{d}\theta}
		\!+\!\! \int_{0}^{t}{\!\|\frac{\partial}{\partial t}\mathbf{H}_{1\tau}(\theta)\|_{2}^{2}}{\, \mathrm{d}\theta}
		\!+\!\! \int_{0}^{t}{\!\|\frac{\partial}{\partial t}\mathbf{u}_{2\tau}(\theta)\|_{3}^{2}}{\, \mathrm{d}\theta} & 
            \\
		\!+\!\! \int_{0}^{t}{\!\|\frac{\partial}{\partial t}p_{2\tau}(\theta)\|_{2}^{2}}{\, \mathrm{d}\theta}
		& \leq C,
		\\
		\int_{0}^{t}{\|\nabla \!\times\! \frac{\partial}{\partial t}\mathbf{E}_{1\tau}(\theta)\|^{2}}{\, \mathrm{d}\theta}
		\!+\! \int_{0}^{t}{\|\nabla \!\times\! \frac{\partial}{\partial t}\mathbf{H}_{1\tau}(\theta)\|^{2}}{\, \mathrm{d}\theta}
		\!+\! \int_{0}^{t}{\|\frac{\partial^{2}}{\partial t^{2}}p_{2\tau}(\theta)\|^{2}}{\, \mathrm{d}\theta}
		& \leq C.
	\end{align*}
\end{lemma}

\subsection{Convergence order of splitting approximate solution}\label{s3.4.3}
For any $0\leq s\leq t\leq T$, let $\mathbf{\Phi}_{s}^{t}:= (\mathcal{E}_{s}^{t},\mathcal{H}_{s}^{t},\mathcal{P}_{s}^{t})$ with $\mathbf{\Phi}_{s}^{s}=\mathbf{I}$ be the solution operator to \eqref{e3.2.8}. Then, for any $\mathbf{U}_{0}$, the function
\begin{equation*}
	\mathbf{U}(t)=\mathbf{\Phi}_{s}^{t}\mathbf{U}_{0}=(\mathcal{E}_{s}^{t}\mathbf{U}_{0},\mathcal{H}_{s}^{t}\mathbf{U}_{0},\mathcal{P}_{s}^{t}\mathbf{U}_{0})
\end{equation*}
is the unique solution to \eqref{e3.2.8} satisfying $\mathbf{U}(s)=\mathbf{U}_{0}$. 
For any $1\leq n\leq N$, noticing the semi-group property $\mathbf{\Phi}_{s}^{t}=\mathbf{\Phi}_{\iota}^{t}\mathbf{\Phi}_{s}^{\iota}$ for any $0\leq s\leq \iota\leq t$, we have a decomposition
\begin{equation}\label{e3.4.16}
	\begin{aligned}
		&\ \mathbf{U}(t_{n}) - \tilde{\mathbf{U}}^{n}
		= \mathbf{\Phi}_{t_{0}}^{t_{n}}\mathbf{U}_{0} - \mathbf{\Phi}_{t_{n}}^{t_{n}}\tilde{\mathbf{U}}^{n}
		\\
		=&\ \mathbf{\Phi}_{t_{0}}^{t_{n}}\tilde{\mathbf{U}}^{0} \mp \mathbf{\Phi}_{t_{1}}^{t_{n}}\tilde{\mathbf{U}}^{1} \mp \cdots \mp \mathbf{\Phi}_{t_{n-1}}^{t_{n}}\tilde{\mathbf{U}}^{n-1} - \mathbf{\Phi}_{t_{n}}^{t_{n}}\tilde{\mathbf{U}}^{n}
		\\
		=& \sum_{i=1}^{n}(\mathbf{\Phi}_{t_{i-1}}^{t_{n}}\tilde{\mathbf{U}}^{i-1} - \mathbf{\Phi}_{t_{i}}^{t_{n}}\tilde{\mathbf{U}}^{i})
		= \sum_{i=1}^{n}\mathbf{\Phi}_{t_{i}}^{t_{n}}(\mathbf{\Phi}_{t_{i-1}}^{t_{i}}\tilde{\mathbf{U}}^{i-1} - \tilde{\mathbf{U}}^{i}).
	\end{aligned}
\end{equation}

For any $1\leq i\leq n\leq N$ and $t\in I_{i}$, define
\begin{equation*}
	\Theta_{i-1}(t) \!=\! \mathbf{\Phi}_{t_{i-1}}^{t}\tilde{\mathbf{U}}^{i-1} - \mathbf{\Gamma}_{t_{i-1}}^{t}\mathbf{R}_{t_{i-1}}^{t}\tilde{\mathbf{U}}^{i-1}
	:=\! \big(\Theta_{(i-1)\mathbf{E}}(t),\Theta_{(i-1)\mathbf{H}}(t),\Theta_{(i-1)p}(t)\big).
\end{equation*}
Then, $\Theta_{i-1}(t_{i-1}) = (\mathbf{0},\mathbf{0},0)$.

\begin{lemma}\label{l3.4.6}
	Assume (H1) and (H2$^{\prime}$). Then there exists a constant $C>0$ independent of $\tau$ such that for all $1\leq i\leq N$ and $t\in I_{i}$,
	\begin{equation*}
		\|\Theta_{(i-1)\mathbf{E}}(t)\| + \|\Theta_{(i-1)\mathbf{H}}(t)\| + \|\Theta_{(i-1)p}(t)\| \leq C \tau^{2}.
	\end{equation*}
\end{lemma}

\begin{proof}
	For any $1\leq i\leq N$, noticing that $\mathbf{\Phi}_{s}^{t}$, $\mathbf{\Gamma}_{s}^{t}$ and $\mathbf{R}_{s}^{t}$ are solution operators, and then $\Theta_{(i-1)\mathbf{E}}(t)$ $(t\in I_{i})$ satisfies a variational equation for all $\mathbf{D}\in \bm{H}_{0}(\mathbf{curl})$,
	\begin{equation*}
		\big(\epsilon\frac{\partial}{\partial t}\Theta_{(i-1)\mathbf{E}}(t),\mathbf{D}\big) \!+\! \big(\sigma \Theta_{(i-1)\mathbf{E}}(t),\mathbf{D}\big) \!-\! \big(\nabla\times \Theta_{(i-1)\mathbf{H}}(t),\mathbf{D}\big) \!\!=\!\! \big(L\nabla \eta_{(i-1)p}(t),\mathbf{D}\big),
	\end{equation*}
	where $\eta_{(i-1)p}(t) = \mathcal{P}_{t_{i-1}}^{t}\tilde{\mathbf{U}}^{i-1}-\tilde{p}^{i-1}$. Taking $\mathbf{D}=\Theta_{(i-1)\mathbf{E}}(t)$ in above equation yields
	\begin{align*}
		&\ \big(\epsilon\frac{\partial}{\partial t}\Theta_{(i-1)\mathbf{E}}(t),\Theta_{(i-1)\mathbf{E}}(t)\big)
		+ \sigma\|\Theta_{(i-1)\mathbf{E}}(t)\|^{2}
		\\
		=&\ \big(\nabla\times \Theta_{(i-1)\mathbf{H}}(t),\Theta_{(i-1)\mathbf{E}}(t)\big)
		+ \big(L\nabla \eta_{(i-1)p}(t),\Theta_{(i-1)\mathbf{E}}(t)\big)
		\\
		\leq&\ \big(\|\nabla\times \Theta_{(i-1)\mathbf{H}}(t)\|+L\|\nabla \eta_{(i-1)p}(t)\|\big)\|\Theta_{(i-1)\mathbf{E}}(t)\|.
	\end{align*}
	It is easy to see $\big(\frac{\partial}{\partial t}\Theta_{(i-1)\mathbf{E}}(t),\Theta_{(i-1)\mathbf{E}}(t)\big)
	=\|\Theta_{(i-1)\mathbf{E}}(t)\|\frac{\mathrm{d}}{\mathrm{d}t}\|\Theta_{(i-1)\mathbf{E}}(t)\|$, then
	\begin{equation*}
		\|\Theta_{(i-1)\mathbf{E}}(t)\|
		\leq \frac{1}{\epsilon}\int_{t_{i-1}}^{t}{\|\nabla\times \Theta_{(i-1)\mathbf{H}}(\theta)\|}{\mathrm{d}\theta}
		+ \frac{L}{\epsilon} \int_{t_{i-1}}^{t}{\|\nabla \eta_{(i-1)p}(\theta)\|}{\mathrm{d}\theta}.
	\end{equation*}
	From Lemmas \ref{l3.3.7} and \ref{l3.4.5}, it follows that
	\begin{equation*}
		\|\nabla\times \Theta_{(i-1)\mathbf{H}}(\theta)\| \leq \frac{1}{\mu}\int_{t_{i-1}}^{\theta}
		{\!\!\big( \|\nabla \!\times\! \nabla \!\times\! \mathcal{E}_{t_{i-1}}^{\nu}\tilde{\mathbf{U}}^{i-1}\| \!+\! \|\nabla \!\times\! \nabla \!\times\! \mathbf{E}_{1\tau}(\nu)\| \big)}{\mathrm{d}\nu} 
		\leq C\tau.
	\end{equation*}
	According to the third equation in \eqref{e3.2.8} and the formula of integration by parts, we have
	\begin{equation*}
		(c_{0}+\mathcal{B})\eta_{(i-1)p}(\theta) = \int_{t_{i-1}}^{\theta}
		{\big( \kappa\Delta \mathcal{P}_{t_{i-1}}^{\nu}\tilde{\mathbf{U}}^{i-1}
			- L \nabla\cdot \mathcal{E}_{t_{i-1}}^{\nu}\tilde{\mathbf{U}}^{i-1} + g(\nu) \big)}{\mathrm{d}\nu}.
	\end{equation*}
	By virtue of Lemma \ref{l3.2.2}, we obtain
	\begin{align*}
		&\ (c_{0}+\mathcal{B})\nabla\eta_{(i-1)p}(\theta) = \nabla(c_{0}+\mathcal{B})\eta_{(i-1)p}(\theta)
		\\
		=&\ \int_{t_{i-1}}^{\theta}
		{\big( \kappa\nabla\Delta \mathcal{P}_{t_{i-1}}^{\nu}\tilde{\mathbf{U}}^{i-1}
			- L \nabla(\nabla\cdot \mathcal{E}_{t_{i-1}}^{\nu}\tilde{\mathbf{U}}^{i-1}) + \nabla g(\nu) \big)}{\mathrm{d}\nu},
	\end{align*}
	which together with \eqref{e3.2.7} and Lemma \ref{l3.3.7} leads to
	\begin{align*}
		\|\nabla\eta_{(i-1)p}(\theta)\| \leq \frac{1}{c_{0}}\|(c_{0}+\mathcal{B})\nabla\eta_{(i-1)p}(\theta)\|
		\leq C\tau.
	\end{align*}
	Thus, $\|\Theta_{(i-1)\mathbf{E}}(t)\| \leq C \tau^{2}$. In a similar way, we can prove $\|\Theta_{(i-1)\mathbf{H}}(t)\| \leq C \tau^{2}$.
	
	By virtue of solution operators $\mathbf{\Phi}_{s}^{t}$, $\mathbf{\Gamma}_{s}^{t}$ and $\mathbf{R}_{s}^{t}$ again, we get $\Theta_{(i-1)p}(t)$ $(t\in I_{i})$ satisfies a variational equation
	\begin{equation*}
		\big(\frac{\partial}{\partial t}(c_{0}+\mathcal{B})\Theta_{(i-1)p}(t),q\big) + \big(\kappa\nabla \Theta_{(i-1)p}(t),\nabla q\big) \!=\!  -\big(L \nabla\cdot\boldsymbol{\xi}_{(i-1)\mathbf{E}}(t),q\big), \ \forall q\in H^{1}_{0},
	\end{equation*}
	where $\boldsymbol{\xi}_{(i-1)\mathbf{E}}(t) = \mathcal{E}_{t_{i-1}}^{t}\tilde{\mathbf{U}}^{i-1}-(\mathcal{E}_{1})_{t_{i-1}}^{t}\tilde{\mathbf{U}}^{i-1}$.
	Taking $q=(c_{0}+\mathcal{B})\Theta_{(i-1)p}(t)$ in above equation, we have
	\begin{align*}
		& \big(\frac{\partial}{\partial t}(c_{0}\!+\!\mathcal{B})\Theta_{(i-1)p}(t),(c_{0}\!+\!\mathcal{B})\Theta_{(i-1)p}(t)\big) \!\!+\!\! \big(\kappa\nabla \Theta_{(i-1)p}(t),\nabla (c_{0}\!+\!\mathcal{B})\Theta_{(i-1)p}(t)\big)
		\\
		& \quad = -\big(L \nabla\cdot\boldsymbol{\xi}_{(i-1)\mathbf{E}}(t),(c_{0}+\mathcal{B})\Theta_{(i-1)p}(t)\big).
	\end{align*}
	According to Lemma \ref{l3.2.2}, we get
	\begin{equation*}
		\big(\kappa\nabla \Theta_{(i-1)p}(t),\nabla (c_{0}+\mathcal{B})\Theta_{(i-1)p}(t)\big)
		\!=\! \big(\kappa\nabla \Theta_{(i-1)p}(t),(c_{0}+\mathcal{B})\nabla \Theta_{(i-1)p}(t)\big) \geq 0.
	\end{equation*}
	By means of \eqref{e3.2.7}, we obtain
	\begin{equation*}
		\|\Theta_{(i-1)p}(t)\|
		\leq \frac{1}{c_{0}} \|(c_{0}+\mathcal{B})\Theta_{(i-1)p}(t)\|
		\leq \frac{L}{c_{0}} \int_{t_{i-1}}^{t}{\|\nabla\cdot \boldsymbol{\xi}_{(i-1)\mathbf{E}}(\theta)\|}{\mathrm{d}\theta}.
	\end{equation*}
	Applying Lemmas \ref{l3.3.7} and \ref{l3.4.5} yields
	\begin{align*}
		&\ \|\nabla\cdot \boldsymbol{\xi}_{(i-1)\mathbf{E}}(\theta)\|
		\\
		\leq&\ \frac{1}{\epsilon}\int_{t_{i-1}}^{\theta}
		{\!\!\big( \sigma \|\nabla \!\cdot\! \mathcal{E}_{t_{i-1}}^{\nu}\tilde{\mathbf{U}}^{i-1}\|
		\!+\! L \|\Delta \mathcal{P}_{t_{i-1}}^{\nu}\tilde{\mathbf{U}}^{i-1}\|
		\!+\! \sigma \|\nabla \!\cdot\! \mathbf{E}_{1\tau}(\nu)\|
		\!+\! L \|\Delta \tilde{p}^{i-1}\| \big)}{\mathrm{d}\nu}
		\\
		\leq&\ C\tau,
	\end{align*}
	which leads to $\|\Theta_{(i-1)p}(t)\| \leq C \tau^{2}$. 
	Thus, the proof is completed.
\end{proof}

Now, we are ready to study the error estimates between the exact solution and the  splitting approximation.

\begin{theorem}\label{t3.4.7}
	Assume (H1) and (H2$^{\prime}$). Let $(\mathbf{E},\mathbf{H},\mathbf{u},p)$ and $(\tilde{\mathbf{E}}^{n},\tilde{\mathbf{H}}^{n},\tilde{\mathbf{u}}^{n},\tilde{p}^{n})$ be exact and splitting approximate solutions to \eqref{e3.2.4}, respectively. Then there exists a constant $C>0$ independent of $\tau$ such that for all $0\leq n\leq N$,
	\begin{equation*}
		\|\mathbf{E}(t_{n})-\tilde{\mathbf{E}}^{n}\| + \|\mathbf{H}(t_{n})-\tilde{\mathbf{H}}^{n}\| + \|\mathbf{u}(t_{n})-\tilde{\mathbf{u}}^{n}\|_{1} + \|p(t_{n})-\tilde{p}^{n}\| \leq C\tau.
	\end{equation*}
\end{theorem}

\begin{proof}
	For any $1\leq n\leq N$, from \eqref{e3.4.16} and Lemmas \ref{l3.3.1}, \ref{l3.4.6}, it follows that
	\begin{align*}
		&\ \|\mathbf{E}(t_{n})-\tilde{\mathbf{E}}^{n}\|^{2}
		+\|\mathbf{H}(t_{n})-\tilde{\mathbf{H}}^{n}\|^{2}
		+\|p(t_{n})-\tilde{p}^{n}\|^{2}
		\\
		=&\ \Big( \|\sum_{i=1}^{n}\mathbf{\Phi}_{t_{i}}^{t_{n}}(\mathbf{\Phi}_{t_{i-1}}^{t_{i}}\tilde{\mathbf{U}}^{i-1} - \tilde{\mathbf{U}}^{i})\| \Big)^{2}
		\leq \Big( \sum_{i=1}^{n}\|\mathbf{\Phi}_{t_{i}}^{t_{n}}(\mathbf{\Phi}_{t_{i-1}}^{t_{i}}\tilde{\mathbf{U}}^{i-1} - \tilde{\mathbf{U}}^{i})\| \Big)^{2}
		\\
		\leq&\ n \sum_{i=1}^{n}\|\mathbf{\Phi}_{t_{i}}^{t_{n}}(\mathbf{\Phi}_{t_{i-1}}^{t_{i}}\tilde{\mathbf{U}}^{i-1} - \tilde{\mathbf{U}}^{i})\|^{2} 
		\leq C N \sum_{i=1}^{n}\|\mathbf{\Phi}_{t_{i-1}}^{t_{i}}\tilde{\mathbf{U}}^{i-1} - \tilde{\mathbf{U}}^{i}\|^{2}
		\\
		=&\ C N \sum_{i=1}^{n}\Big(\|\Theta_{(i-1)\mathbf{E}}(t_{i})\|^{2}+\|\Theta_{(i-1)\mathbf{H}}(t_{i})\|^{2}
		+\|\Theta_{(i-1)p}(t_{i})\|^{2}\Big)
		\\
		\leq&\ C N^{2} \tau^{4}
		\leq C \tau^{2}.
	\end{align*}
	From the third equation in \eqref{e3.2.4}, the first equation in \eqref{e3.4.6} and the coerciveness of $a(\cdot,\cdot)$, it follows that
	\begin{align*}
		\beta_{1} \|\mathbf{u}(t_{n})-\mathbf{u}_{2\tau}(t_{n})\|_{1}^{2} \leq&\ |a(\mathbf{u}(t_{n})-\mathbf{u}_{2\tau}(t_{n}),\mathbf{u}(t_{n})-\mathbf{u}_{2\tau}(t_{n}))| 
		\\
		=&\ \big|\big(p(t_{n})-p_{2\tau}(t_{n}), \alpha \nabla\cdot (\mathbf{u}(t_{n})-\mathbf{u}_{2\tau}(t_{n}))\big)\big| 
		\\
		\leq&\ C\|p(t_{n})-p_{2\tau}(t_{n})\|\|\mathbf{u}(t_{n})-\mathbf{u}_{2\tau}(t_{n})\|_{1},
	\end{align*}
	which implies $\|\mathbf{u}(t_{n})-\tilde{\mathbf{u}}^{n}\|_{1} \leq C \|p(t_{n})-\tilde{p}^{n}\| \leq C\tau$.
\end{proof}

\section{Finite element discretization and error estimates}\label{s3.5}
In this section, we propose a numerical approximation to \eqref{e3.2.1}--\eqref{e3.2.3} based on splitting technique and derive its convergence order. We first introduce N\'{e}d\'{e}lec's finite element spaces \cite{Mo1991,Mo2003,Ne1980}, and apply FEM and backward Euler method to construct full-discretized schemes to \eqref{e3.4.5} and \eqref{e3.4.6}, respectively. Then, we define a splitting numerical solution and study its stability. Finally, we derive error estimates.

\subsection{Splitting numerical scheme}\label{s3.5.1}
For simplicity, we assume that the domain $\Omega$ is a polyhedron, otherwise, we can construct a polyhedron to approximate $\Omega$. Let $\mathcal{T}_{h}$ be a regular tetrahedral partition of $\Omega$. For each element $K\in \mathcal{T}_{h}$ with edges $\mathcal{E}_{K}$ and faces $\mathcal{F}_{K}$, let $h_{K}$ be its diameter and $h=\max_{K\in \mathcal{T}_{h}}h_{K}$ be the mesh size of $\mathcal{T}_{h}$. For any integer $k\geq 0$, denote by $\mathbb{P}_{k}(\Omega)$ and $\tilde{\mathbb{P}}_{k}(\Omega)$ the spaces of polynomials with degree no more than $k$ and homogeneous polynomials with degree $k$, respectively.
For any $k\geq 1$, define two subspaces of $(\mathbb{P}_{k}(K))^{3}$ as follows. 
\begin{align*}
	\mathbb{S}_{k}(K) =&\ \{\mathbf{P}\in (\tilde{\mathbb{P}}_{k}(K))^{3}: \mathbf{P}(\mathbf{x})\cdot\mathbf{x}=0, \ \forall\, \mathbf{x}\in K\},
	\\
	\mathbb{R}_{k}(K) =&\ (\mathbb{P}_{k-1}(K))^{3}\oplus \mathbb{S}_{k}(K).
\end{align*}
Then, we construct four finite element spaces by
\begin{align*}
	\mathbb{E}_{h} =&\ \{\mathbf{E}_{h}\in \bm{H}(\mathbf{curl}): \mathbf{E}_{h}|_{K}\in \mathbb{R}_{k}(K), \ \forall\, K\in \mathcal{T}_{h}\},
	\\
	\mathbb{H}_{h} =&\ \{\mathbf{H}_{h}\in \bm{L}^{2}(\Omega): \mathbf{H}_{h}|_{K}\in (\mathbb{P}_{k-1}(K))^{3}, \ \forall\, K\in \mathcal{T}_{h}\},
	\\
	\mathbb{U}_{h} =&\ \{\mathbf{u}_{h}\in \bm{C}(\Omega): \mathbf{u}_{h}|_{\partial\Omega}=0,\, \mathbf{u}_{h}|_{K}\in (\mathbb{P}_{k}(K))^{3}, \ \forall\, K\in \mathcal{T}_{h}\},
	\\
	\mathbb{P}_{h} =&\ \{p_{h}\in C(\Omega): p_{h}|_{\partial\Omega}=0,\, p_{h}|_{K}\in \mathbb{P}_{k}(K), \ \forall\, K\in \mathcal{T}_{h}\},
\end{align*}
where $\mathbb{E}_{h}$ is the space of N\'{e}d\'{e}lec edge elements \cite{Ne1980}. In this work, we fix $k=1$.

For notation simplicity, we simultaneously denote by $\mathcal{P}_{h}$ the $L^{2}$-orthogonal projection operator from $\bm{H}(\mathbf{curl})$ to $\mathbb{E}_{h}$, or from $\bm{L}^{2}(\Omega)$ to $\mathbb{H}_{h}$, or from $\bm{H}^{1}_{0}$ to $\mathbb{U}_{h}$, or from $H^{1}_{0}$ to $\mathbb{P}_{h}$ when no confusion occurs. Then, for any $(\mathbf{E},\mathbf{H},\mathbf{u},p) \in \bm{H}^{2} \times \bm{H}^{1} \times \bm{H}^{2} \times H^{2}$, there hold \cite{ErGu2021}
	\begin{equation}\label{e3.5.1}
		\begin{aligned}
			\|\mathbf{E}-\mathcal{P}_{h}\mathbf{E}\| \leq&\ Ch^{2}\|\mathbf{E}\|_{2},
			\\
			\|\mathbf{H}-\mathcal{P}_{h}\mathbf{H}\| \leq&\ Ch\|\mathbf{H}\|_{1},
			\\
			\|\mathbf{u}-\mathcal{P}_{h}\mathbf{u}\| + h \|\mathbf{u}-\mathcal{P}_{h}\mathbf{u}\|_{1} 
			\leq&\ Ch^{2}\|\mathbf{u}\|_{2},
			\\
			\|p-\mathcal{P}_{h}p\| + h \|\nabla (p-\mathcal{P}_{h}p)\| 
			\leq&\ Ch^{2}\|p\|_{2}.
		\end{aligned}
	\end{equation}

Next, we construct a fully discretized FEM approximation  and define a splitting numerical solution $(\mathbf{E}_{h}^{n},\mathbf{H}_{h}^{n},\mathbf{u}_{h}^{n},p_{h}^{n})$ $(0\leq n\leq N)$ to \eqref{e3.2.4}. Take initial setting
\begin{equation*}
	(\mathbf{E}_{h}^{0},\mathbf{H}_{h}^{0},\mathbf{u}_{h}^{0},p_{h}^{0})
	:=(\mathcal{P}_{h}\mathbf{E}_{0},\mathcal{P}_{h}\mathbf{H}_{0},\mathcal{P}_{h}\mathbf{u}_{0},\mathcal{P}_{h}p_{0}).
\end{equation*}
Assume that we have already known $(\mathbf{E}_{h}^{n-1},\mathbf{H}_{h}^{n-1},\mathbf{u}_{h}^{n-1},p_{h}^{n-1})$ for $n\geq 1$, we define the next approximation $\mathbf{U}_{h}^{n}$ via approximating \eqref{e3.4.5} and \eqref{e3.4.6}.

A fully discretized scheme to \eqref{e3.4.5} on $I_{n}$ is defined as  seeking $(\mathbf{E}_{1\tau,h}^{n},\mathbf{H}_{1\tau,h}^{n})\in \mathbb{E}_{h}\times \mathbb{H}_{h}$ such that for all $(\mathbf{D}_{h},\mathbf{B}_{h})\in \mathbb{E}_{h}\times \mathbb{H}_{h}$,
\begin{equation}\label{e3.5.2}
	\begin{aligned}
		& (\epsilon \bar\partial_{t}\mathbf{E}_{1\tau,h}^{n},\mathbf{D}_{h}) \!+\! (\sigma \mathbf{E}_{1\tau,h}^{n},\mathbf{D}_{h}) \!-\! (\mathbf{H}_{1\tau,h}^{n},\nabla\times \mathbf{D}_{h})
		\!=\! (L\nabla p_{h}^{n-1},\mathbf{D}_{h}) \!+\! (\mathbf{j}(t_{n}),\mathbf{D}_{h}),
		\\
		& (\mu \bar\partial_{t}\mathbf{H}_{1\tau,h}^{n},\mathbf{B}_{h}) + (\nabla\times \mathbf{E}_{1\tau,h}^{n},\mathbf{B}_{h}) = 0,
	\end{aligned}
\end{equation}
where $\bar\partial_{t}\mathbf{E}_{1\tau,h}^{n} = (\mathbf{E}_{1\tau,h}^{n}-\mathbf{E}_{1\tau,h}^{n-1})/ \tau$ and $(\mathbf{E}_{1\tau,h}^{n-1},\mathbf{H}_{1\tau,h}^{n-1})=(\mathbf{E}_{h}^{n-1},\mathbf{H}_{h}^{n-1})$. Here we have used the formula of integration by parts to convert $\big(\nabla\times \mathbf{H}_{1\tau},\mathbf{D}\big)$ to $\big(\mathbf{H}_{1\tau},\nabla\times \mathbf{D}\big)$ in \eqref{e3.4.5} in order to match the choice of finite element for $\mathbf{H}_{1\tau}$. 
Similarly, a fully discretized scheme to \eqref{e3.4.6} on $I_{n}$ is defined as searching $(\mathbf{u}_{2\tau,h}^{n},p_{2\tau,h}^{n})\in \mathbb{U}_{h}\times  \mathbb{P}_{h}$ such that for all $(\mathbf{v}_{h},q_{h})\in \mathbb{U}_{h} \times \mathbb{P}_{h}$,
\begin{equation}\label{e3.5.3}
	\begin{aligned}
		& (\lambda_{c}\nabla\cdot \mathbf{u}_{2\tau,h}^{n}, \nabla\cdot \mathbf{v}_{h}) + (G\nabla \mathbf{u}_{2\tau,h}^{n}, \nabla \mathbf{v}_{h}) - (p_{2\tau,h}^{n}, \alpha \nabla\cdot \mathbf{v}_{h}) = 0,
		\\
		& (\bar\partial_{t}(c_{0}p_{2\tau,h}^{n} \!\!+\! \alpha\nabla \!\cdot\! \mathbf{u}_{2\tau,h}^{n}),q_{h}) \!\!+\!\! (\kappa\nabla p_{2\tau,h}^{n},\nabla q_{h}) \!\!=\!\! (L\mathbf{E}_{1\tau,h}^{n},\nabla q_{h}) \!\!+\!\! (g(t_{n}),q_{h}),
	\end{aligned}
\end{equation}
where $(\mathbf{u}_{2\tau,h}^{n-1}, p_{2\tau,h}^{n-1}) = (\mathbf{u}_{h}^{n-1}, p_{h}^{n-1})$. 
Then, the splitting numerical solution to \eqref{e3.2.4} at $t_{n}$ is defined by
\begin{equation*}
	(\mathbf{E}_{h}^{n},\mathbf{H}_{h}^{n},\mathbf{u}_{h}^{n},p_{h}^{n}) 
	= (\mathbf{E}_{1\tau,h}^{n},\mathbf{H}_{1\tau,h}^{n},\mathbf{u}_{2\tau,h}^{n},p_{2\tau,h}^{n}), \quad 1\leq n\leq N.
\end{equation*}

In order to study the stability of splitting numerical solution, we first establish a relationship between $\mathbf{u}_{h}^{n}$ and $p_{h}^{n}$ in terms of the first equation in \eqref{e3.5.3}. For any given $p_{h}^{n}=p_{h}\in \mathbb{P}_{h}$, there exists a unique function $\mathbf{u}_{h}\in \mathbb{U}_{h}$ such that
\begin{equation*}
	a(\mathbf{u}_{h},\mathbf{v}_{h}) = (p_{h}, \alpha\nabla\cdot \mathbf{v}_{h}), \quad \forall \, \mathbf{v}_{h}\in \mathbb{U}_{h}.
\end{equation*}
Then, we define an operator $\mathcal{B}_{h}: \mathbb{P}_{h}\rightarrow \mathbb{P}_{h}$ by $\mathcal{B}_{h}p_{h} = \alpha\nabla\cdot \mathbf{u}_{h}$. Similar to the study for the operator $\mathcal{B}$, we can verify that $\mathcal{B}_{h}$ is a linear, bounded, self-adjoint and monotone operator and satisfies $\mathcal{B}_{h}\frac{\partial}{\partial t}=\frac{\partial}{\partial t}\mathcal{B}_{h}$ and $\mathcal{B}_{h}\frac{\partial}{\partial x_{i}} = \frac{\partial}{\partial x_{i}}\mathcal{B}_{h}$ $(i=1,2,3)$. Furthermore, $\|(c_{0}+\mathcal{B}_{h})p_{h}\|^{2} \geq c_{0}^{2} \|p_{h}\|^{2}$. 
Thus, the second equation in \eqref{e3.4.3} is equivalent to
\begin{equation}\label{e3.5.4}
	(\bar\partial_{t}(c_{0} \!+\! \mathcal{B}_{h}) p_{h}^{n},q_{h}) \!+\! (\kappa\nabla p_{h}^{n},\nabla q_{h}) \!=\! (L\mathbf{E}_{h}^{n},\nabla q_{h}) \!+\! (g(t_{n}),q_{h}), \quad \forall \, q_{h}\in \mathbb{P}_{h}.
\end{equation}

\begin{lemma}\label{l3.5.1}
	Assume (H1) and (H2$^{\prime}$). Let $(\mathbf{E}_{h}^{n},\mathbf{H}_{h}^{n},\mathbf{u}_{h}^{n},p_{h}^{n})$ $(0\leq n\leq N)$ be splitting numerical solution to \eqref{e3.2.4}. Then there exists a constant $C>0$ independent of both $\tau$ and $h$ such that
	\begin{equation*}
		\|\mathbf{E}_{h}^{n}\|^{2} + \|\mathbf{H}_{h}^{n}\|^{2} + \|\mathbf{u}_{h}^{n}\|_{1}^{2} + \|p_{h}^{n}\|^{2} + \tau\|\nabla p_{h}^{n}\|^{2} \leq C.
	\end{equation*}
\end{lemma}

\begin{proof}
	Taking $(\mathbf{D}_{h},\mathbf{B}_{h})=(\mathbf{E}_{h}^{n},\mathbf{H}_{h}^{n})$ in \eqref{e3.5.2} and $q_{h}=p_{h}^{n}$ in \eqref{e3.5.4}, respectively, we obtain
	\begin{equation}\label{e3.5.5}
		\begin{aligned}
			& (\epsilon \bar\partial_{t}\mathbf{E}_{h}^{n},\mathbf{E}_{h}^{n}) \!+\! (\mu \bar\partial_{t}\mathbf{H}_{h}^{n},\mathbf{H}_{h}^{n}) \!+\! (\bar\partial_{t}(c_{0} \!+\! \mathcal{B}_{h}) p_{h}^{n},p_{h}^{n}) \!+\! \sigma \|\mathbf{E}_{h}^{n}\|^{2} \!+\! \kappa \|\nabla p_{h}^{n}\|^{2} 
			\\
			&\quad - (L \nabla p_{h}^{n-1},\mathbf{E}_{h}^{n}) - (L \mathbf{E}_{h}^{n},\nabla p_{h}^{n})
			= (\mathbf{j}(t_{n}),\mathbf{E}_{h}^{n}) + (g(t_{n}),p_{h}^{n}).
		\end{aligned}
	\end{equation}
	Direct computation gives
	\begin{equation}\label{e3.5.6}
		\begin{aligned}
			\tau(\bar\partial_{t}\mathbf{E}_{h}^{n},\mathbf{E}_{h}^{n}) =&\ \frac{\tau^{2}}{2}(\bar\partial_{t}\mathbf{E}_{h}^{n},\bar\partial_{t}\mathbf{E}_{h}^{n}) + \frac{1}{2}(\mathbf{E}_{h}^{n},\mathbf{E}_{h}^{n}) - \frac{1}{2}(\mathbf{E}_{h}^{n-1},\mathbf{E}_{h}^{n-1})
			\\
			\geq&\ \frac{1}{2}\|\mathbf{E}_{h}^{n}\|^{2} - \frac{1}{2}\|\mathbf{E}_{h}^{n-1}\|^{2}.
		\end{aligned}
	\end{equation}
	Similarly, we get $\tau(\bar\partial_{t}\mathbf{H}_{h}^{n},\mathbf{H}_{h}^{n})\geq \frac{1}{2}\|\mathbf{H}_{h}^{n}\|^{2} - \frac{1}{2}\|\mathbf{H}_{h}^{n-1}\|^{2}$. 
	By virtue of the monotonicity and self-adjointness of operator $\mathcal{B}_{h}$, we have
	\begin{align*}
		\tau(\bar\partial_{t}(c_{0}+\mathcal{B}_{h}) p_{h}^{n},p_{h}^{n}) =&\ \frac{\tau^{2}}{2}((c_{0}+\mathcal{B}_{h})\bar\partial_{t}p_{h}^{n},\bar\partial_{t}p_{h}^{n}) + \frac{1}{2}((c_{0}+\mathcal{B}_{h}) p_{h}^{n},p_{h}^{n})
		\\
		&\ - \frac{1}{2}((c_{0}+\mathcal{B}_{h}) p_{h}^{n-1},p_{h}^{n-1})
		\\
		\geq&\ \frac{1}{2}((c_{0}+\mathcal{B}_{h}) p_{h}^{n},p_{h}^{n}) - \frac{1}{2}((c_{0}+\mathcal{B}_{h}) p_{h}^{n-1},p_{h}^{n-1}).
	\end{align*}
	From assumption (H1), it follows that
	\begin{align*}
		(L \nabla p_{h}^{n-1},\mathbf{E}_{h}^{n}) <&\ (\sqrt{\sigma\kappa} \nabla p_{h}^{n-1},\mathbf{E}_{h}^{n}) 
		\leq \frac{\kappa}{2}\|\nabla p_{h}^{n-1}\|^{2} + \frac{\sigma}{2}\|\mathbf{E}_{h}^{n}\|^{2},
		\\
		(L \mathbf{E}_{h}^{n},\nabla p_{h}^{n}) <&\ (\sqrt{\sigma\kappa} \mathbf{E}_{h}^{n},\nabla p_{h}^{n}) 
		\leq \frac{\sigma}{2}\|\mathbf{E}_{h}^{n}\|^{2} + \frac{\kappa}{2}\|\nabla p_{h}^{n}\|^{2}.
	\end{align*}
	Noticing \eqref{e3.5.5}, we get
	\begin{align*}
		&\ \epsilon\|\mathbf{E}_{h}^{n}\|^{2} - \epsilon\|\mathbf{E}_{h}^{n-1}\|^{2} + \mu\|\mathbf{H}_{h}^{n}\|^{2} - \mu\|\mathbf{H}_{h}^{n-1}\|^{2} + ((c_{0}+\mathcal{B}_{h}) p_{h}^{n},p_{h}^{n}) 
		\\
		&\quad - ((c_{0}+\mathcal{B}_{h}) p_{h}^{n-1},p_{h}^{n-1}) + 2\sigma \tau \|\mathbf{E}_{h}^{n}\|^{2} + 2\kappa \tau \|\nabla p_{h}^{n}\|^{2} 
		\\
		&\quad - \kappa \tau \|\nabla p_{h}^{n-1}\|^{2} - \sigma \tau \|\mathbf{E}_{h}^{n}\|^{2} 
		- \sigma \tau \|\mathbf{E}_{h}^{n}\|^{2} - \kappa \tau \|\nabla p_{h}^{n}\|^{2}
		\\
		\leq&\ \tau\|\mathbf{j}(t_{n})\|^{2} + \tau\|\mathbf{E}_{h}^{n}\|^{2} + \tau\|g(t_{n})\|^{2} + \tau\|p_{h}^{n}\|^{2}.
	\end{align*}
	Inductively for $n\geq 1$ in above inequality, we obtain
	\begin{align*}
		&\ \|\mathbf{E}_{h}^{n}\|^{2} + \|\mathbf{H}_{h}^{n}\|^{2} + \|p_{h}^{n}\|^{2} + \tau\|\nabla p_{h}^{n}\|^{2}
		\\
		\leq&\ C(\|\mathbf{E}_{h}^{0}\|^{2} \!+\! \|\mathbf{H}_{h}^{0}\|^{2} \!+\! \|p_{h}^{0}\|^{2} \!+\! \tau\|\nabla p_{h}^{0}\|^{2}) + C\sum_{i=1}^{n}\tau\|\mathbf{j}(t_{i})\|^{2} + C\sum_{i=1}^{n}\tau\|g(t_{i})\|^{2}
		\\
		&\ + C\sum_{i=1}^{n}\tau(\|\mathbf{E}_{h}^{i}\|^{2} + \|p_{h}^{i}\|^{2}).
	\end{align*}
	Eliminating the item $\tau\|\nabla p_{h}^{n}\|^{2}$ in above inequality and then applying discrete Gronwall's lemma, we have
	\begin{equation*}
		\|\mathbf{E}_{h}^{n}\|^{2} + \|\mathbf{H}_{h}^{n}\|^{2} + \|p_{h}^{n}\|^{2}
		\leq C e^{C\sum_{i=1}^{n}\tau}
		\leq C(T).
	\end{equation*}
	This directly leads to $\tau\|\nabla p_{h}^{n}\|^{2} \leq C$.
	Then, applying \eqref{e3.5.3} and the coerciveness of $a(\cdot,\cdot)$ yields
	\begin{equation*}
		\beta_{1} \|\mathbf{u}_{h}^{n}\|_{1}^{2} \leq |a(\mathbf{u}_{h}^{n},\mathbf{u}_{h}^{n})| = |(p_{h}^{n}, \alpha \nabla\cdot \mathbf{u}_{h}^{n})| \leq C \|p_{h}^{n}\|\|\mathbf{u}_{h}^{n}\|_{1}.
	\end{equation*}
	Thus, we complete the proof.
\end{proof}

\subsection{Error estimates}\label{s3.5.2}
In this part, we derive the error estimates for the splitting numerical approximation to \eqref{e3.2.4} and investigate its convergence order. We will set up an error equation according to \eqref{e3.4.5} and \eqref{e3.5.2} together with \eqref{e3.4.6} and \eqref{e3.5.3} without using the operator $\mathcal{B}_{h}$ and analyze the error estimation. 
Now, we introduce some operators onto finite element spaces. 
Define three sets of moments for $\mathbf{w}\in \bm{H}^{2}$ on $K\in \mathcal{T}_{h}$ \cite{Mo1991,Mo2003,Ne1980}
\begin{align*}
	M_{e_{K}}(\mathbf{w}) =&\ \left\{ \int_{e_{K}}{(\mathbf{w}\cdot\boldsymbol{\tau}_{e_{K}})\chi}{\, \mathrm{d}e_{K}}, \ \forall \, \chi\in \mathbb{P}_{k-1}(e_{K}), \, e_{K}\in \mathcal{E}_{K} \right\},
	\\
	M_{f_{K}}(\mathbf{w}) =&\ \left\{ \int_{f_{K}}{(\mathbf{w}\times\mathbf{n})\cdot\boldsymbol{\chi}}{\, \mathrm{d}f_{K}}, \ \forall \, \boldsymbol{\chi}\in (\mathbb{P}_{k-2}(f_{K}))^{2}, \, f_{K}\in \mathcal{F}_{K} \right\},
	\\
	M_{K}(\mathbf{w}) =&\ \left\{ \int_{K}{\mathbf{w}\cdot\boldsymbol{\chi}}{\, \mathrm{d}K}, \ \forall \, \boldsymbol{\chi}\in (\mathbb{P}_{k-3}(K))^{3} \right\},
\end{align*}
where $\boldsymbol{\tau}_{e_{K}}$ denotes a unit vector parallel to $e_{K}$. N\'{e}d\'{e}lec \cite{Ne1980} proved that these sets are unisolvent on $\mathbb{R}_{k}$ and conforming in $\bm{H}(\mathbf{curl})$. 
Denote by $\mathcal{I}_{h}: \bm{H}^{2}\to \mathbb{E}_{h}$ the interpolation operator such that $(\mathcal{I}_{h}\mathbf{E})|_{K}$ is the unique polynomial in $\mathbb{R}_{k}(K)$ having the same moments as $\mathbf{E}|_{K}$ \cite{GiRa1986,Mo1991,Ne1980}. 
Define two bilinear forms: for any $\mathbf{v}\in \bm{H}^{1}_{0}$ and $\chi,q \in H^{1}_{0}$,
\begin{equation*}
	b(\mathbf{v},\chi) = (\chi,\alpha \nabla\cdot \mathbf{v}),
	\quad
	c(\chi,q) = (\kappa \nabla\chi, \nabla q).
\end{equation*}
We also define two Ritz projection operators $\mathcal{Q}_{h}:\bm{H}^{1}_{0}\rightarrow \mathbb{U}_{h}$ and $\mathcal{R}_{h}:H^{1}_{0}\rightarrow \mathbb{P}_{h}$ satisfying for any $\mathbf{u}\in \bm{H}^{1}_{0}$ and $p\in H^{1}_{0}$,
\begin{align}
	\label{e3.5.7}
	a(\mathcal{Q}_{h}\mathbf{u},\mathbf{v}_{h})-b(\mathbf{v}_{h},\mathcal{R}_{h}p) =&\ a(\mathbf{u},\mathbf{v}_{h})-b(\mathbf{v}_{h},p), \quad \forall \, \mathbf{v}_{h}\in \mathbb{U}_{h},
	\\
	\label{e3.5.8}
	c(\mathcal{R}_{h}p,q_{h}) =&\ c(p,q_{h}), \quad \forall \, q_{h}\in \mathbb{P}_{h}.
\end{align}

Define a norm by $\|\cdot\|_{a}=\sqrt{a(\cdot,\cdot)}$ and it is equivalent to the norm $\|\cdot\|_{1}$, which will be used in forthcoming error estimates.

\begin{lemma}\label{l3.5.2}
	For any $\mathbf{E}\in \bm{H}^{2}$, $\mathbf{u}\in \bm{H}^{2}\cap \bm{H}^{1}_{0}$ and $p\in H^{2}\cap H^{1}_{0}$, there hold
	\begin{align*}
		\|\mathcal{I}_{h}\mathbf{E}-\mathbf{E}\| + \|\nabla\times (\mathcal{I}_{h}\mathbf{E}-\mathbf{E})\| \leq&\ Ch\|\mathbf{E}\|_{2},
		\\
		\|\mathcal{R}_{h}p-p\| + h \|\nabla (\mathcal{R}_{h}p-p)\| \leq&\ Ch^{2}\|p\|_{2},
		\\
		\|\mathcal{Q}_{h}\mathbf{u}-\mathbf{u}\|_{1} \leq&\ Ch\|\mathbf{u}\|_{2} + Ch^{2}\|p\|_{2},
	\end{align*}
	where $C>0$ is a constant independent of $h$.
\end{lemma}

\begin{proof}
	The first and second inequalities are two classical estimates, cf.\ \cite[Theorem 2]{Ne1980} and \cite[Lemma 1.1]{Th2006}, respectively. 
	From \eqref{e3.5.7}, it follows that
	\begin{equation*}
		a(\mathcal{Q}_{h}\mathbf{u}-\mathcal{P}_{h}\mathbf{u},\mathbf{v}_{h}) = a(u-\mathcal{P}_{h}\mathbf{u},\mathbf{v}_{h}) + b(\mathbf{v}_{h},\mathcal{R}_{h}p-p), \quad \forall \, \mathbf{v}_{h}\in \mathbb{U}_{h}.
	\end{equation*}
	Taking $\mathbf{v}_{h} = \mathcal{Q}_{h}\mathbf{u}-\mathcal{P}_{h}\mathbf{u} \in \mathbb{U}_{h}$ in above equation yields
	\begin{align*}
		&\ \beta_{1} \|\mathcal{Q}_{h}\mathbf{u}-\mathcal{P}_{h}\mathbf{u}\|_{1}^{2} 
           \\
        \leq&\ |a(\mathcal{Q}_{h}\mathbf{u}-\mathcal{P}_{h}\mathbf{u},\mathcal{Q}_{h}\mathbf{u}-\mathcal{P}_{h}\mathbf{u})|
		\\
		\leq&\ |a(u-\mathcal{P}_{h}\mathbf{u},\mathcal{Q}_{h}\mathbf{u}-\mathcal{P}_{h}\mathbf{u})| + |b(\mathcal{Q}_{h}\mathbf{u}-\mathcal{P}_{h}\mathbf{u},\mathcal{R}_{h}p-p)|
		\\
		\leq&\ C\|\mathbf{u}-\mathcal{P}_{h}\mathbf{u}\|_{1} \|\mathcal{Q}_{h}\mathbf{u}-\mathcal{P}_{h}\mathbf{u}\|_{1} + C\|\mathcal{Q}_{h}\mathbf{u}-\mathcal{P}_{h}\mathbf{u}\|_{1} \|\mathcal{R}_{h}p-p\|,
	\end{align*}
	from which we get
	\begin{equation*}
		\|\mathcal{Q}_{h}\mathbf{u}-\mathcal{P}_{h}\mathbf{u}\|_{1} \leq C\|\mathbf{u}-\mathcal{P}_{h}\mathbf{u}\|_{1} + C\|\mathcal{R}_{h}p-p\| \leq Ch\|\mathbf{u}\|_{2} + Ch^{2}\|p\|_{2}.
	\end{equation*}
	Then, we have
	\begin{equation*}
		\|\mathcal{Q}_{h}\mathbf{u}-\mathbf{u}\|_{1} \leq \|\mathcal{Q}_{h}\mathbf{u}-\mathcal{P}_{h}\mathbf{u}\|_{1} + \|\mathcal{P}_{h}\mathbf{u}-\mathbf{u}\|_{1} \leq Ch\|\mathbf{u}\|_{2} + Ch^{2}\|p\|_{2}.
	\end{equation*}
	Thus, the proof is completed.
\end{proof}

Now, we derive the error estimates for the splitting numerical approximation to \eqref{e3.2.4}.

\begin{theorem}\label{t3.5.3}
	Assume (H1) and (H2$^{\prime}$). Let $(\mathbf{E},\mathbf{H},\mathbf{u},p)$ and $(\mathbf{E}_{h}^{n},\mathbf{H}_{h}^{n},\mathbf{u}_{h}^{n},p_{h}^{n})$ $(0\leq n\leq N)$ be the exact solution and splitting numerical approximation to \eqref{e3.2.4}, respectively. Then there exists a constant $C>0$ independent of both $\tau$ and $h$ such that
	\begin{equation*}
		\|\mathbf{E}(t_{n})-\mathbf{E}_{h}^{n}\| + \|\mathbf{H}(t_{n})-\mathbf{H}_{h}^{n}\| + \|\mathbf{u}(t_{n})-\mathbf{u}_{h}^{n}\|_{1} + \|p(t_{n})-p_{h}^{n}\|  \leq C(\tau+h).
	\end{equation*}
\end{theorem}

\begin{proof}
	Define four error functions by 
	\begin{align*}
		\mathbf{E}_{1\tau}(t_{n})-\mathbf{E}_{h}^{n} =&\ (\mathbf{E}_{1\tau}(t_{n})-\mathcal{I}_{h}\mathbf{E}_{1\tau}(t_{n})) + (\mathcal{I}_{h}\mathbf{E}_{1\tau}(t_{n})-\mathbf{E}_{h}^{n}) := \boldsymbol{\rho}(t_{n}) + \boldsymbol{\xi}_{h}^{n}, 
		\\
		\mathbf{H}_{1\tau}(t_{n})-\mathbf{H}_{h}^{n} =&\ (\mathbf{H}_{1\tau}(t_{n})\!-\!\mathcal{P}_{h}\mathbf{H}_{1\tau}(t_{n})) \!+\! (\mathcal{P}_{h}\mathbf{H}_{1\tau}(t_{n})\!-\mathbf{H}_{h}^{n}) := \boldsymbol{\phi}(t_{n}) + \boldsymbol{\psi}_{h}^{n}, 
		\\
		\mathbf{u}_{2\tau}(t_{n})-\mathbf{u}_{h}^{n} =&\ (\mathbf{u}_{2\tau}(t_{n})-\mathcal{Q}_{h}\mathbf{u}_{2\tau}(t_{n})) + (\mathcal{Q}_{h}\mathbf{u}_{2\tau}(t_{n})-\mathbf{u}_{h}^{n}) := \boldsymbol{\pi}(t_{n}) + \boldsymbol{\gamma}_{h}^{n}, 
		\\
		p_{2\tau}(t_{n})-p_{h}^{n} =&\ (p_{2\tau}(t_{n})-\mathcal{R}_{h}p_{2\tau}(t_{n})) + (\mathcal{R}_{h}p_{2\tau}(t_{n})-p_{h}^{n}) := \zeta(t_{n}) + \eta_{h}^{n}.
	\end{align*}
	By virtue of \eqref{e3.5.1} and Lemmas \ref{l3.4.5}, \ref{l3.5.2}, we arrive at
	\begin{equation}\label{e3.5.9}
		\begin{aligned}
			\|\boldsymbol{\rho}(t_{n})\| =&\ \|\mathbf{E}_{1\tau}(t_{n})-\mathcal{I}_{h}\mathbf{E}_{1\tau}(t_{n})\| 
			\leq Ch\|\mathbf{E}_{1\tau}(t_{n})\|_{2} \leq Ch, 
			\\
			\|\boldsymbol{\phi}(t_{n})\| =&\ \|\mathbf{H}_{1\tau}(t_{n})-\mathcal{P}_{h}\mathbf{H}_{1\tau}(t_{n})\| 
			\leq Ch\|\mathbf{H}_{1\tau}(t_{n})\|_{1} \leq Ch, 
			\\
			\|\boldsymbol{\pi}(t_{n})\|_{1} =&\ \|\mathbf{u}_{2\tau}(t_{n})-\mathcal{Q}_{h}\mathbf{u}_{2\tau}(t_{n})\|_{1} 
			\leq Ch\|\mathbf{u}_{2\tau}(t_{n})\|_{2} + Ch^{2}\|p_{2\tau}(t_{n})\|_{2}
			\\
			\leq&\ Ch \big( \|\mathbf{u}_{0}\|_{2} + \int_{0}^{t_{n}}{\|\frac{\partial}{\partial t} \mathbf{u}_{2\tau}(\theta)\|_{3}}{\, \mathrm{d}\theta} \big) + Ch^{2}\|p_{2\tau}(t_{n})\|_{2} \leq Ch, 
			\\
			\|\zeta(t_{n})\| =&\ \|p_{2\tau}(t_{n})-\mathcal{R}_{h}p_{2\tau}(t_{n})\| 
			\leq Ch^{2}\|p_{2\tau}(t_{n})\|_{2} \leq Ch^{2}.
		\end{aligned}
	\end{equation}
	According to the first equations in \eqref{e3.4.5} and \eqref{e3.5.2}, we have that for any $\mathbf{D}_{h}\in \mathbb{E}_{h}$,
	\begin{align*}
		&\ (\epsilon\bar\partial_{t}\boldsymbol{\xi}_{h}^{n},\mathbf{D}_{h}) = (\epsilon\bar\partial_{t}(\mathcal{I}_{h}\mathbf{E}_{1\tau}(t_{n})-\mathbf{E}_{h}^{n}),\mathbf{D}_{h}) 
		\\
		=&\ \big( \epsilon(\mathcal{I}_{h}\bar\partial_{t}\mathbf{E}_{1\tau}(t_{n})-\bar\partial_{t}\mathbf{E}_{1\tau}(t_{n})),\mathbf{D}_{h} \big) 
		+ \big( \epsilon(\bar\partial_{t}\mathbf{E}_{1\tau}(t_{n})-\frac{\partial}{\partial t}\mathbf{E}_{1\tau}(t_{n})),\mathbf{D}_{h}\big)
		\\
		&\ + \big( \epsilon\frac{\partial}{\partial t}\mathbf{E}_{1\tau}(t_{n}),\mathbf{D}_{h}\big) 
		- (\epsilon\bar\partial_{t}\mathbf{E}_{h}^{n},\mathbf{D}_{h})
		\\
		=&\ \big( \epsilon(\mathcal{I}_{h}\bar\partial_{t}\mathbf{E}_{1\tau}(t_{n})-\bar\partial_{t}\mathbf{E}_{1\tau}(t_{n})),\mathbf{D}_{h} \big) 
		+ \big( \epsilon(\bar\partial_{t}\mathbf{E}_{1\tau}(t_{n})-\frac{\partial}{\partial t}\mathbf{E}_{1\tau}(t_{n})),\mathbf{D}_{h}\big)
		\\
		&\ - (\sigma \boldsymbol{\rho}(t_{n}),\mathbf{D}_{h}) - (\sigma \boldsymbol{\xi}_{h}^{n},\mathbf{D}_{h}) 
		+ (\boldsymbol{\phi}(t_{n}),\nabla\times \mathbf{D}_{h}) 
		\\
		&\ + (\boldsymbol{\psi}_{h}^{n},\nabla\times \mathbf{D}_{h}) 
		+ (L\nabla \zeta(t_{n-1}),\mathbf{D}_{h}) 
		+ (L\nabla \eta_{h}^{n-1},\mathbf{D}_{h}).
	\end{align*}
	Denote by $\boldsymbol{\omega}_{1}^{n}:= \mathcal{I}_{h}\bar\partial_{t}\mathbf{E}_{1\tau}(t_{n})-\bar\partial_{t}\mathbf{E}_{1\tau}(t_{n})$ and $\boldsymbol{\omega}_{2}^{n}:= \bar\partial_{t}\mathbf{E}_{1\tau}(t_{n})-\frac{\partial}{\partial t}\mathbf{E}_{1\tau}(t_{n})$. 
	From the definitions of the finite element spaces $\mathbb{E}_{h}$ and $\mathbb{H}_{h}$, it follows that $\nabla\times \mathbf{D}_{h}|_{K} \in \mathbb{H}_{h}|_{K}$. This together with $\mathcal{P}_{h}$ leads to $(\boldsymbol{\phi}(t_{n}),\nabla\times \mathbf{D}_{h}) = 0$. 
	By means of the second equations in \eqref{e3.4.5} and \eqref{e3.5.2}, we get that for any $\mathbf{B}_{h}\in \mathbb{H}_{h}$,
	\begin{align*}
		&\ (\mu\bar\partial_{t}\boldsymbol{\psi}_{h}^{n},\mathbf{B}_{h}) = (\mu\bar\partial_{t}(\mathcal{P}_{h}\mathbf{H}_{1\tau}(t_{n})-\mathbf{H}_{h}^{n}),\mathbf{B}_{h}) 
		\\
		=&\ \big( \mu(\mathcal{P}_{h}\bar\partial_{t}\mathbf{H}_{1\tau}(t_{n})-\bar\partial_{t}\mathbf{H}_{1\tau}(t_{n})), \mathbf{B}_{h} \big) 
		+ \big( \mu(\bar\partial_{t}\mathbf{H}_{1\tau}(t_{n})-\frac{\partial}{\partial t}\mathbf{H}_{1\tau}(t_{n})),\mathbf{B}_{h}\big)
		\\
		&\ + \big( \mu\frac{\partial}{\partial t}\mathbf{H}_{1\tau}(t_{n}),\mathbf{B}_{h}\big) 
		- (\mu\bar\partial_{t}\mathbf{H}_{h}^{n},\mathbf{B}_{h})
		\\
		=&\ \big( \mu(\mathcal{P}_{h}\bar\partial_{t}\mathbf{H}_{1\tau}(t_{n})-\bar\partial_{t}\mathbf{H}_{1\tau}(t_{n})), \mathbf{B}_{h} \big) 
		+ \big( \mu(\bar\partial_{t}\mathbf{H}_{1\tau}(t_{n})-\frac{\partial}{\partial t}\mathbf{H}_{1\tau}(t_{n})),\mathbf{B}_{h}\big)
		\\
		&\ - (\nabla\times \boldsymbol{\rho}(t_{n}),\mathbf{B}_{h}) - (\nabla\times \boldsymbol{\xi}_{h}^{n},\mathbf{B}_{h}).
	\end{align*}
	Denote by $\boldsymbol{\omega}_{3}^{n}:= \mathcal{P}_{h}\bar\partial_{t}\mathbf{H}_{1\tau}(t_{n})-\bar\partial_{t}\mathbf{H}_{1\tau}(t_{n})$ and $\boldsymbol{\omega}_{4}^{n}:= \bar\partial_{t}\mathbf{H}_{1\tau}(t_{n})-\frac{\partial}{\partial t}\mathbf{H}_{1\tau}(t_{n})$. 
	By \eqref{e3.4.6}, \eqref{e3.5.3} and \eqref{e3.5.7}, we obtain that for any $\mathbf{v}_{h}\in \mathbb{U}_{h}$,
	\begin{align*}
		&\ a(\boldsymbol{\gamma}_{h}^{n},\mathbf{v}_{h}) - b(\mathbf{v}_{h},\eta_{h}^{n})
		\\
		=&\ a(\mathcal{Q}_{h}\mathbf{u}_{2\tau}(t_{n}),\mathbf{v}_{h}) - a(\mathbf{u}_{h}^{n},\mathbf{v}_{h}) - b(\mathbf{v}_{h},\mathcal{R}_{h}p_{2\tau}(t_{n})) + b(\mathbf{v}_{h},p_{h}^{n})
		\\
		=&\ a(\mathbf{u}_{2\tau}(t_{n}),\mathbf{v}_{h}) - b(\mathbf{v}_{h},p_{2\tau}(t_{n})) - a(\mathbf{u}_{h}^{n},\mathbf{v}_{h}) + b(\mathbf{v}_{h},p_{h}^{n}) = 0.
	\end{align*}
	By virtue of \eqref{e3.4.6}, \eqref{e3.5.3} and \eqref{e3.5.8}, we have that for any $q_{h}\in \mathbb{P}_{h}$,
	\begin{align*}
		&\ (c_{0}\bar\partial_{t}\eta_{h}^{n},q_{h}) + b(\bar\partial_{t}\boldsymbol{\gamma}_{h}^{n},q_{h}) + c(\eta_{h}^{n},q_{h})
		\\
		=&\ (c_{0}\bar\partial_{t}\mathcal{R}_{h}p_{2\tau}(t_{n}),q_{h}) - (c_{0}\bar\partial_{t}p_{h}^{n},q_{h}) + b(\bar\partial_{t}\mathcal{Q}_{h}\mathbf{u}_{2\tau}(t_{n}),q_{h}) - b(\bar\partial_{t}\mathbf{u}_{h}^{n},q_{h})
		\\
		&\ + c(\mathcal{R}_{h}p_{2\tau}(t_{n}),q_{h}) -c(p_{h}^{n},q_{h})
		\\
		=&\ \big( c_{0}(\mathcal{R}_{h}\bar\partial_{t}p_{2\tau}(t_{n})-\bar\partial_{t}p_{2\tau}(t_{n})),q_{h} \big) + \big( c_{0}(\bar\partial_{t}p_{2\tau}(t_{n})-\frac{\partial}{\partial t}p_{2\tau}(t_{n})),q_{h} \big)
		\\
		&\ + \big( c_{0}\frac{\partial}{\partial t}p_{2\tau}(t_{n}),q_{h} \big) - (c_{0}\bar\partial_{t}p_{h}^{n},q_{h})
		\\
		&\ + \big( \alpha(\nabla\cdot \mathcal{Q}_{h}\bar\partial_{t}\mathbf{u}_{2\tau}(t_{n})-\nabla\cdot \bar\partial_{t}\mathbf{u}_{2\tau}(t_{n})),q_{h} \big)
		\\
		&\ + \big( \alpha(\nabla\cdot \bar\partial_{t}\mathbf{u}_{2\tau}(t_{n})-\nabla\cdot \frac{\partial}{\partial t}\mathbf{u}_{2\tau}(t_{n})),q_{h} \big)
		\\
		&\ + \big( \alpha\nabla\cdot \frac{\partial}{\partial t}\mathbf{u}_{2\tau}(t_{n}),q_{h} \big) - (\alpha\nabla\cdot \bar\partial_{t}\mathbf{u}_{h}^{n},q_{h}) + c(p_{2\tau}(t_{n}),q_{h}) -c(p_{h}^{n},q_{h})
		\\
		=&\ \big( c_{0}(\mathcal{R}_{h}\bar\partial_{t}p_{2\tau}(t_{n})-\bar\partial_{t}p_{2\tau}(t_{n})),q_{h} \big) + \big( c_{0}(\bar\partial_{t}p_{2\tau}(t_{n})-\frac{\partial}{\partial t}p_{2\tau}(t_{n})),q_{h} \big)
		\\
		&\ + \big( \alpha(\nabla\cdot \mathcal{Q}_{h}\bar\partial_{t}\mathbf{u}_{2\tau}(t_{n})-\nabla\cdot \bar\partial_{t}\mathbf{u}_{2\tau}(t_{n})),q_{h} \big)
		\\
		&\ + \big( \alpha(\nabla\cdot \bar\partial_{t}\mathbf{u}_{2\tau}(t_{n})-\nabla\cdot \frac{\partial}{\partial t}\mathbf{u}_{2\tau}(t_{n})),q_{h} \big) 
		+ (L\boldsymbol{\rho}(t_{n}),\nabla q_{h}) + (L\boldsymbol{\xi}_{h}^{n},\nabla q_{h})
		\\
		:=&\ (c_{0}\omega_{5}^{n},q_{h}) \!+\! (c_{0}\omega_{6}^{n},q_{h}) \!+\! (\alpha\omega_{7}^{n},q_{h}) \!+\! (\alpha\omega_{8}^{n},q_{h}) \!+\! (L\boldsymbol{\rho}(t_{n}),\nabla q_{h}) \!+\! (L\boldsymbol{\xi}_{h}^{n},\nabla q_{h}).
	\end{align*}
	Taking $\mathbf{D}_{h}=\boldsymbol{\xi}_{h}^{n}$, $\mathbf{B}_{h}=\boldsymbol{\psi}_{h}^{n}$, $\mathbf{v}_{h}=\bar\partial_{t}\boldsymbol{\gamma}_{h}^{n}$ and $q_{h}=\eta_{h}^{n}$ in above four equations, respectively, and then summing up, we obtain an error equation
	\begin{align*}
		&\ (\epsilon\bar\partial_{t}\boldsymbol{\xi}_{h}^{n},\boldsymbol{\xi}_{h}^{n}) + (\mu\bar\partial_{t}\boldsymbol{\psi}_{h}^{n},\boldsymbol{\psi}_{h}^{n}) + a(\bar\partial_{t}\boldsymbol{\gamma}_{h}^{n},\boldsymbol{\gamma}_{h}^{n}) + (c_{0}\bar\partial_{t}\eta_{h}^{n},\eta_{h}^{n}) + c(\eta_{h}^{n},\eta_{h}^{n}) 
		\\
		=&\ (\epsilon\boldsymbol{\omega}_{1}^{n},\boldsymbol{\xi}_{h}^{n}) 
		+ (\epsilon\boldsymbol{\omega}_{2}^{n},\boldsymbol{\xi}_{h}^{n}) 
		- (\sigma \boldsymbol{\rho}(t_{n}),\boldsymbol{\xi}_{h}^{n}) 
		- (\sigma \boldsymbol{\xi}_{h}^{n},\boldsymbol{\xi}_{h}^{n}) 
		+ (L\nabla \zeta(t_{n-1}),\boldsymbol{\xi}_{h}^{n}) 
		\\
		&\ + (L\nabla \eta_{h}^{n-1},\boldsymbol{\xi}_{h}^{n}) 
		+ (\mu\boldsymbol{\omega}_{3}^{n},\boldsymbol{\psi}_{h}^{n}) 
		+ (\mu\boldsymbol{\omega}_{4}^{n},\boldsymbol{\psi}_{h}^{n}) 
		- (\nabla\times \boldsymbol{\rho}(t_{n}),\boldsymbol{\psi}_{h}^{n})
		\\
		&\ + (c_{0}\omega_{5}^{n},\eta_{h}^{n}) + (c_{0}\omega_{6}^{n},\eta_{h}^{n}) + (\alpha\omega_{7}^{n},\eta_{h}^{n}) + (\alpha\omega_{8}^{n},\eta_{h}^{n})
            \\
		&\ + (L\boldsymbol{\rho}(t_{n}),\nabla \eta_{h}^{n}) + (L\boldsymbol{\xi}_{h}^{n},\nabla \eta_{h}^{n}).
	\end{align*}
	Similar to the proof of \eqref{e3.5.6}, we can also estimate $\tau(\bar\partial_{t}\boldsymbol{\xi}_{h}^{n},\boldsymbol{\xi}_{h}^{n})$, $\tau(\bar\partial_{t}\boldsymbol{\psi}_{h}^{n},\boldsymbol{\psi}_{h}^{n})$, $\tau a(\bar\partial_{t}\boldsymbol{\gamma}_{h}^{n},\boldsymbol{\gamma}_{h}^{n})$ and $\tau(\bar\partial_{t}\eta_{h}^{n},\eta_{h}^{n})$, and then utilizing assumption (H1) yields
	\begin{align*}
		&\ \frac{\epsilon}{2}\big(\|\boldsymbol{\xi}_{h}^{n}\|^{2}-\|\boldsymbol{\xi}_{h}^{n-1}\|^{2}\big) 
		+ \frac{\mu}{2}\big(\|\boldsymbol{\psi}_{h}^{n}\|^{2}-\|\boldsymbol{\psi}_{h}^{n-1}\|^{2}\big) 
		+ \frac{1}{2}\big(\|\boldsymbol{\gamma}_{h}^{n}\|_{a}^{2}-\|\boldsymbol{\gamma}_{h}^{n-1}\|_{a}^{2}\big) 
		\\
		&\ + \frac{c_{0}}{2}\big(\|\eta_{h}^{n}\|^{2}-\|\eta_{h}^{n-1}\|^{2}\big) 
		+ \tau \kappa \|\nabla \eta_{h}^{n}\|^{2}
		\\
		\leq&\ \tau (\epsilon\bar\partial_{t}\boldsymbol{\xi}_{h}^{n},\boldsymbol{\xi}_{h}^{n}) 
		+ \tau (\mu\bar\partial_{t}\boldsymbol{\psi}_{h}^{n},\boldsymbol{\psi}_{h}^{n}) 
		+ \tau a(\bar\partial_{t}\boldsymbol{\gamma}_{h}^{n},\boldsymbol{\gamma}_{h}^{n}) 
		+ \tau (c_{0}\bar\partial_{t}\eta_{h}^{n},\eta_{h}^{n}) 
		+ \tau c(\eta_{h}^{n},\eta_{h}^{n})
		\\
		\leq&\ \frac{\epsilon}{2}\tau \|\boldsymbol{\omega}_{1}^{n}\|^{2} 
		+ \frac{\epsilon}{2}\tau \|\boldsymbol{\xi}_{h}^{n}\|^{2} 
		+ \frac{\epsilon}{2}\tau \|\boldsymbol{\omega}_{2}^{n}\|^{2} 
		+ \frac{\epsilon}{2}\tau \|\boldsymbol{\xi}_{h}^{n}\|^{2} 
		+ \frac{\sigma}{2}\tau \|\boldsymbol{\rho}(t_{n})\|^{2} 
		+ \frac{\sigma}{2}\tau \|\boldsymbol{\xi}_{h}^{n}\|^{2} 
		\\
		&\ + \sigma\tau \|\boldsymbol{\xi}_{h}^{n}\|^{2} 
		+ \frac{\kappa}{2}\tau \|\nabla \zeta(t_{n-1})\|^{2} 
		+ \frac{\sigma}{2}\tau \|\boldsymbol{\xi}_{h}^{n}\|^{2} 
		+ \frac{\kappa}{2}\tau \|\nabla \eta_{h}^{n-1}\|^{2} 
		+ \frac{\sigma}{2}\tau \|\boldsymbol{\xi}_{h}^{n}\|^{2} 
		\\
		&\ + \frac{\mu}{2}\tau \|\boldsymbol{\omega}_{3}^{n}\|^{2} 
		+ \frac{\mu}{2}\tau \|\boldsymbol{\psi}_{h}^{n}\|^{2} 
		+ \frac{\mu}{2}\tau \|\boldsymbol{\omega}_{4}^{n}\|^{2} 
		+ \frac{\mu}{2}\tau \|\boldsymbol{\psi}_{h}^{n}\|^{2} 
		+ \frac{1}{2}\tau \|\nabla\times \boldsymbol{\rho}(t_{n})\|^{2} 
		\\
		&\ + \frac{1}{2}\tau \|\boldsymbol{\psi}_{h}^{n}\|^{2} 
		+ \frac{c_{0}}{2}\tau \|\omega_{5}^{n}\|^{2} + \frac{c_{0}}{2}\tau \|\eta_{h}^{n}\|^{2} + \frac{c_{0}}{2}\tau \|\omega_{6}^{n}\|^{2} + \frac{c_{0}}{2}\tau \|\eta_{h}^{n}\|^{2} 
		\\
		&\ + \frac{\alpha}{2}\tau \|\omega_{7}^{n}\|^{2} + \frac{\alpha}{2}\tau \|\eta_{h}^{n}\|^{2} + \frac{\alpha}{2}\tau \|\omega_{8}^{n}\|^{2} + \frac{\alpha}{2}\tau \|\eta_{h}^{n}\|^{2}
		\\
		&\ + \sigma\tau \|\boldsymbol{\rho}(t_{n})\|^{2} + \frac{\kappa}{4}\tau \|\nabla \eta_{h}^{n}\|^{2} 
		+ \sigma\tau \|\boldsymbol{\xi}_{h}^{n}\|^{2} + \frac{\kappa}{4}\tau \|\nabla \eta_{h}^{n}\|^{2}.
	\end{align*}
	Inductively for $n\geq 1$ in above inequality, we obtain
	\begin{equation}\label{e3.5.10}
		\begin{aligned}
			&\ \epsilon\|\boldsymbol{\xi}_{h}^{n}\|^{2} + \mu\|\boldsymbol{\psi}_{h}^{n}\|^{2} 
			+ \|\boldsymbol{\gamma}_{h}^{n}\|_{a}^{2} + c_{0}\|\eta_{h}^{n}\|^{2} 
			+ \kappa \tau \|\nabla \eta_{h}^{n}\|^{2} 
			\\
			\leq&\ \epsilon\|\boldsymbol{\xi}_{h}^{0}\|^{2} + \mu\|\boldsymbol{\psi}_{h}^{0}\|^{2} 
			+ \|\boldsymbol{\gamma}_{h}^{0}\|_{a}^{2} + c_{0}\|\eta_{h}^{0}\|^{2} 
			+ \kappa \tau \|\nabla \eta_{h}^{0}\|^{2} 
			\\
			&\ + \epsilon\sum_{j=1}^{n}\tau\|\boldsymbol{\omega}_{1}^{j}\|^{2} 
			\!+\! \epsilon\sum_{j=1}^{n}\tau\|\boldsymbol{\omega}_{2}^{j}\|^{2}
			\!+\! 3\sigma\sum_{j=1}^{n}\tau\|\boldsymbol{\rho}(t_{j})\|^{2} 
			\!+\! \kappa\sum_{j=1}^{n}\tau\|\nabla \zeta(t_{j-1})\|^{2} 
			\\
			&\ + \mu\sum_{j=1}^{n}\tau\|\boldsymbol{\omega}_{3}^{j}\|^{2} 
			+ \mu\sum_{j=1}^{n}\tau\|\boldsymbol{\omega}_{4}^{j}\|^{2} 
			+ \sum_{j=1}^{n}\tau\|\nabla\times \boldsymbol{\rho}(t_{j})\|^{2}
			\\
			&\ + c_{0}\sum_{j=1}^{n}\tau\|\omega_{5}^{j}\|^{2} + c_{0}\sum_{j=1}^{n}\tau\|\omega_{6}^{j}\|^{2} 
			+ \alpha\sum_{j=1}^{n}\tau\|\omega_{7}^{j}\|^{2} + \alpha\sum_{j=1}^{n}\tau\|\omega_{8}^{j}\|^{2} 
			\\
			&\ + (2\epsilon\!+\!7\sigma)\sum_{j=1}^{n}\tau\|\boldsymbol{\xi}_{h}^{j}\|^{2} 
			\!+\! (2\mu\!+\!1)\sum_{j=1}^{n}\tau\|\boldsymbol{\psi}_{h}^{j}\|^{2} 
			\!+\! 2(c_{0}\!+\!\alpha)\sum_{j=1}^{n}\tau\|\eta_{h}^{j}\|^{2}.
		\end{aligned}
	\end{equation}
	From Lemma \ref{l3.5.2} and \eqref{e3.5.1}, it follows that
	\begin{align*}
		\|\boldsymbol{\xi}_{h}^{0}\| =&\ \|\mathcal{I}_{h}\mathbf{E}_{0}-\mathbf{E}_{h}^{0}\| 
		\leq C\|\mathcal{I}_{h}\mathbf{E}_{0}-\mathbf{E}_{0}\| + C\|\mathbf{E}_{0}-\mathcal{P}_{h}\mathbf{E}_{0}\| 
		\\
		\leq&\ Ch\|\mathbf{E}_{0}\|_{2} + Ch^{2}\|\mathbf{E}_{0}\|_{2} \leq Ch,
		\\
		\|\boldsymbol{\psi}_{h}^{0}\| =&\ \|\mathcal{P}_{h}\mathbf{H}_{0}-\mathbf{H}_{h}^{0}\| = 0, 
		\\
		\|\boldsymbol{\gamma}_{h}^{0}\|_{a} =&\ \|\mathcal{Q}_{h}\mathbf{u}_{0}-\mathbf{u}_{h}^{0}\|_{a} \leq C\|\mathcal{Q}_{h}\mathbf{u}_{0}-\mathbf{u}_{h}^{0}\|_{1}
		\\
		\leq&\ C\|\mathcal{Q}_{h}\mathbf{u}_{0}-\mathbf{u}_{0}\|_{1} + C\|\mathbf{u}_{0}-\mathcal{P}_{h}\mathbf{u}_{0}\|_{1}
		\\
		\leq&\ Ch\|\mathbf{u}_{0}\|_{2} + Ch^{2}\|p_{0}\|_{2} + Ch\|\mathbf{u}_{0}\|_{2} \leq Ch,
		\\
		\|\eta_{h}^{0}\| =&\ \|\mathcal{R}_{h}p_{0}-p_{h}^{0}\| \leq \|\mathcal{R}_{h}p_{0}-p_{0}\|+\|p_{0}-\mathcal{P}_{h}p_{0}\| 
		\\
		\leq&\ Ch^{2}\|p_{0}\|_{2} + Ch^{2}\|p_{0}\|_{2} \leq Ch^{2},
		\\
		\|\nabla \eta_{h}^{0}\| =&\ \|\nabla (\mathcal{R}_{h}p_{0}-p_{h}^{0})\| \leq \|\nabla (\mathcal{R}_{h}p_{0}-p_{0})\|+\|\nabla (p_{0}-\mathcal{P}_{h}p_{0})\| 
		\\
		\leq&\ Ch\|p_{0}\|_{2} + Ch\|p_{0}\|_{2} \leq Ch.
	\end{align*}
	Direct computation yields
	\begin{equation*}
		\boldsymbol{\omega}_{1}^{j}= \mathcal{I}_{h}\bar\partial_{t}\mathbf{E}_{1\tau}(t_{j})-\bar\partial_{t}\mathbf{E}_{1\tau}(t_{j}) 
		=\frac{1}{\tau}\int_{t_{j-1}}^{t_{j}}{\big(\mathcal{I}_{h}\frac{\partial}{\partial t} \mathbf{E}_{1\tau}(\theta)-\frac{\partial}{\partial t}\mathbf{E}_{1\tau}(\theta)\big)}{\, \mathrm{d}\theta}.
	\end{equation*}
	According to Cauchy--Schwarz inequality and Lemma \ref{l3.5.2}, we get
	\begin{align*}
		\|\boldsymbol{\omega}_{1}^{j}\|^{2}
		=&\ \int_{\Omega}{\Big( \frac{1}{\tau} \int_{t_{j-1}}^{t_{j}}{\big(\mathcal{I}_{h}\frac{\partial}{\partial t} \mathbf{E}_{1\tau}(\theta)-\frac{\partial}{\partial t} \mathbf{E}_{1\tau}(\theta)\big)}{\, \mathrm{d}\theta} \Big)^{2}}{\, \mathrm{d}\mathbf{x}}
		\\
		\leq&\ \frac{1}{\tau^{2}}\int_{\Omega}{\Big( \int_{t_{j-1}}^{t_{j}}{1^{2}}{\, \mathrm{d}\theta} \int_{t_{j-1}}^{t_{j}}\big(\mathcal{I}_{h}\frac{\partial}{\partial t} \mathbf{E}_{1\tau}(\theta)-\frac{\partial}{\partial t} \mathbf{E}_{1\tau}(\theta)\big)^{2}{\, \mathrm{d}\theta} \Big)}{\, \mathrm{d}\mathbf{x}}
		\\
		\leq&\ \frac{1}{\tau}\int_{t_{j-1}}^{t_{j}}{\|\mathcal{I}_{h}\frac{\partial}{\partial t} \mathbf{E}_{1\tau}(\theta)-\frac{\partial}{\partial t} \mathbf{E}_{1\tau}(\theta)\|^{2}}{\, \mathrm{d}\theta}
		\\
		\leq&\ C\frac{1}{\tau} h^{2} \int_{t_{j-1}}^{t_{j}}{\|\frac{\partial}{\partial t} \mathbf{E}_{1\tau}(\theta)\|_{2}^{2}}{\, \mathrm{d}\theta}.
	\end{align*}
	Similarly, applying Cauchy--Schwarz inequality, \eqref{e3.5.1} and Lemma \ref{l3.5.2}, we obtain
	\begin{align*}
		\|\boldsymbol{\omega}_{3}^{j}\|^{2} \leq&\ \frac{1}{\tau}\int_{t_{j-1}}^{t_{j}}{\!\!\|\mathcal{P}_{h}\frac{\partial}{\partial t} \mathbf{H}_{1\tau}(\theta)\!-\!\frac{\partial}{\partial t} \mathbf{H}_{1\tau}(\theta)\|^{2}}{\, \mathrm{d}\theta} 
		\leq C\frac{1}{\tau} h^{2} \int_{t_{j-1}}^{t_{j}}{\!\!\|\frac{\partial}{\partial t} \mathbf{H}_{1\tau}(\theta)\|_{1}^{2}}{\, \mathrm{d}\theta},
		\\
		\|\omega_{5}^{j}\|^{2} \leq&\ \frac{1}{\tau}\int_{t_{j-1}}^{t_{j}}{\|\mathcal{R}_{h}\frac{\partial}{\partial t} p_{2\tau}(\theta)-\frac{\partial}{\partial t} p_{2\tau}(\theta)\|^{2}}{\, \mathrm{d}\theta} 
		\leq C\frac{1}{\tau} h^{4} \int_{t_{j-1}}^{t_{j}}{\|\frac{\partial}{\partial t} p_{2\tau}(\theta)\|_{2}^{2}}{\, \mathrm{d}\theta},
		\\
		\|\omega_{7}^{j}\|^{2} \leq&\ \frac{1}{\tau}\int_{t_{j-1}}^{t_{j}}{\|\nabla\cdot \big(\mathcal{Q}_{h}\frac{\partial}{\partial t} \mathbf{u}_{2\tau}(\theta)-\frac{\partial}{\partial t} \mathbf{u}_{2\tau}(\theta)\big)\|^{2}}{\, \mathrm{d}\theta} 
		\\
		\leq&\ \frac{1}{\tau}\int_{t_{j-1}}^{t_{j}}{\|\mathcal{Q}_{h}\frac{\partial}{\partial t} \mathbf{u}_{2\tau}(\theta)-\frac{\partial}{\partial t} \mathbf{u}_{2\tau}(\theta)\|_{1}^{2}}{\, \mathrm{d}\theta}
		\\
		\leq&\ C\frac{1}{\tau} h^{2} \int_{t_{j-1}}^{t_{j}}{\|\frac{\partial}{\partial t} \mathbf{u}_{2\tau}(\theta)\|_{2}^{2}}{\, \mathrm{d}\theta} + C\frac{1}{\tau} h^{4} \int_{t_{j-1}}^{t_{j}}{\|\frac{\partial}{\partial t} p_{2\tau}(\theta)\|_{2}^{2}}{\, \mathrm{d}\theta}.
	\end{align*}
	According to Taylor expansion, we obtain
	\begin{equation*}
		\boldsymbol{\omega}_{2}^{j}= \bar\partial_{t}\mathbf{E}_{1\tau}(t_{j})-\frac{\partial}{\partial t}\mathbf{E}_{1\tau}(t_{j}) 
		=\frac{1}{\tau}\int_{t_{j-1}}^{t_{j}}{(t_{j-1}-\theta)\frac{\partial^{2}}{\partial t^{2}} \mathbf{E}_{1\tau}(\theta)}{\, \mathrm{d}\theta}.
	\end{equation*}
	Utilizing Cauchy--Schwarz inequality yields
	\begin{align*}
		\|\boldsymbol{\omega}_{2}^{j}\|^{2}
		&= \int_{\Omega}{\Big( \frac{1}{\tau} \int_{t_{j-1}}^{t_{j}}{(t_{j-1}-\theta)\frac{\partial^{2}}{\partial t^{2}} \mathbf{E}_{1\tau}(\theta)}{\, \mathrm{d}\theta} \Big)^{2}}{\, \mathrm{d}\mathbf{x}}
		\\
		&\leq \frac{1}{\tau^{2}}\int_{\Omega}{\Big( \int_{t_{j-1}}^{t_{j}}{(t_{j-1}-\theta)^{2}}{\, \mathrm{d}\theta} \int_{t_{j-1}}^{t_{j}}{\big(\frac{\partial^{2}}{\partial t^{2}} \mathbf{E}_{1\tau}(\theta)\big)^{2}}{\, \mathrm{d}\theta} \Big)}{\, \mathrm{d}\mathbf{x}}
		\\
		&\leq C\tau\int_{t_{j-1}}^{t_{j}}{\|\frac{\partial^{2}}{\partial t^{2}} \mathbf{E}_{1\tau}(\theta)\|^{2}}{\, \mathrm{d}\theta}.
	\end{align*}
	In a similar way, we have
	\begin{align*}
		\|\boldsymbol{\omega}_{4}^{j}\|^{2} \leq&\ C\tau\int_{t_{j-1}}^{t_{j}}{\|\frac{\partial^{2}}{\partial t^{2}} \mathbf{H}_{1\tau}(\theta)\|^{2}}{\, \mathrm{d}\theta},
		\\
		\|\omega_{6}^{j}\|^{2} \leq&\ C\tau\int_{t_{j-1}}^{t_{j}}{\|\frac{\partial^{2}}{\partial t^{2}} p_{2\tau}(\theta)\|^{2}}{\, \mathrm{d}\theta},
		\\
		\|\omega_{8}^{j}\|^{2} \leq&\ C\tau\int_{t_{j-1}}^{t_{j}}{\|\nabla\cdot \frac{\partial^{2}}{\partial t^{2}} \mathbf{u}_{2\tau}(\theta)\|^{2}}{\, \mathrm{d}\theta}
		\leq C\tau\int_{t_{j-1}}^{t_{j}}{\|\frac{\partial^{2}}{\partial t^{2}} \mathbf{u}_{2\tau}(\theta)\|_{1}^{2}}{\, \mathrm{d}\theta}
		\\
		\leq&\ C\tau\int_{t_{j-1}}^{t_{j}}{\|\frac{\partial^{2}}{\partial t^{2}} p_{2\tau}(\theta)\|^{2}}{\, \mathrm{d}\theta}.
	\end{align*}
	Applying Lemma \ref{l3.5.2}, we arrive at
	\begin{align*}
		\|\boldsymbol{\rho}(t_{j})\|^{2} =&\ \|\mathbf{E}_{1\tau}(t_{j})-\mathcal{I}_{h}\mathbf{E}_{1\tau}(t_{j})\|^{2} 
		\leq Ch^{2}\|\mathbf{E}_{1\tau}(t_{j})\|_{2}^{2},
		\\
		\|\nabla \zeta(t_{j-1})\|^{2} =&\ \|\nabla (p_{2\tau}(t_{j-1})-\mathcal{R}_{h}p_{2\tau}(t_{j-1}))\|^{2} 
		\leq Ch^{2}\|p_{2\tau}(t_{j-1})\|_{2}^{2},
		\\
		\|\nabla\times \boldsymbol{\rho}(t_{j})\|^{2} =&\ \|\nabla\times (\mathbf{E}_{1\tau}(t_{j})-\mathcal{I}_{h}\mathbf{E}_{1\tau}(t_{j}))\|^{2} 
		\leq Ch^{2}\|\mathbf{E}_{1\tau}(t_{j})\|_{2}^{2}.
	\end{align*}
	Combining with \eqref{e3.5.10} and Lemma \ref{l3.4.5}, we get
	\begin{align*}
		&\, \|\boldsymbol{\xi}_{h}^{n}\|^{2} + \|\boldsymbol{\psi}_{h}^{n}\|^{2} 
		+ \|\boldsymbol{\gamma}_{h}^{n}\|_{a}^{2} + \|\eta_{h}^{n}\|^{2} 
		\\
		\leq&\, Ch^{2} \!+\! Ch^{2} \!+\! Ch^{4} \!+\! C\tau h^{2} 
		\!+\! C h^{2} \!\! \int_{0}^{t_{n}}{\!\!\|\frac{\partial}{\partial t} \mathbf{E}_{1\tau}(\theta)\|_{2}^{2}}{\mathrm{d}\theta} 
            \!+\! C\tau^{2} \!\! \int_{0}^{t_{n}}{\!\!\|\frac{\partial^{2}}{\partial t^{2}} \mathbf{E}_{1\tau}(\theta)\|^{2}}{\mathrm{d}\theta} 
            \\
		&\, + C h^{2} \sum_{j=1}^{n}\tau \|\mathbf{E}_{1\tau}(t_{j})\|_{2}^{2} 
            \!+\! C h^{2} \sum_{j=1}^{n}\tau \|p_{2\tau}(t_{j-1})\|_{2}^{2} 
            \!+\! C h^{2} \!\! \int_{0}^{t_{n}}{\!\|\frac{\partial}{\partial t} \mathbf{H}_{1\tau}(\theta)\|_{1}^{2}}{\, \mathrm{d}\theta}
		\\
		&\, + C\tau^{2} \!\! \int_{0}^{t_{n}}{\!\|\frac{\partial^{2}}{\partial t^{2}} \mathbf{H}_{1\tau}(\theta)\|^{2}}{\mathrm{d}\theta} 
		\!+\! C h^{2} \sum_{j=1}^{n}\tau \|\mathbf{E}_{1\tau}(t_{j})\|_{2}^{2} 
		\!+\! Ch^{4} \!\! \int_{0}^{t_{n}}{\|\frac{\partial}{\partial t} p_{2\tau}(\theta)\|_{2}^{2}}{\mathrm{d}\theta} 
		\\
		&\, +\! C\tau^{2} \!\! \int_{0}^{t_{n}}{\!\!\!\|\frac{\partial^{2}}{\partial t^{2}} p_{2\tau}(\theta)\|^{2}}{\mathrm{d}\theta} 
		\!+\! \big( Ch^{2} \!\! \int_{0}^{t_{n}}{\!\!\!\|\frac{\partial}{\partial t} \mathbf{u}_{2\tau}(\theta)\|_{2}^{2}}{\mathrm{d}\theta} 
		\!+\! Ch^{4} \!\! \int_{0}^{t_{n}}{\!\!\!\|\frac{\partial}{\partial t} p_{2\tau}(\theta)\|_{2}^{2}}{\mathrm{d}\theta} \big)
		\\
		&\, + C\tau^{2} \int_{0}^{t_{n}}{\|\frac{\partial^{2}}{\partial t^{2}} p_{2\tau}(\theta)\|^{2}}{\, \mathrm{d}\theta}
		+ C \sum_{j=1}^{n}\tau \big( \|\boldsymbol{\xi}_{h}^{j}\|^{2} + \|\boldsymbol{\psi}_{h}^{j}\|^{2} + \|\eta_{h}^{j}\|^{2} \big)
		\\
		\leq&\, C (\tau^{2} + h^{2}) + C \sum_{j=1}^{n}\tau \big( \|\boldsymbol{\xi}_{h}^{j}\|^{2} + \|\boldsymbol{\psi}_{h}^{j}\|^{2} + \|\eta_{h}^{j}\|^{2} \big).
	\end{align*}
	By discrete Gronwall's lemma, we obtain
	\begin{equation*}
		\|\boldsymbol{\xi}_{h}^{n}\|^{2} + \|\boldsymbol{\psi}_{h}^{n}\|^{2} 
		+ \|\boldsymbol{\gamma}_{h}^{n}\|_{a}^{2} + \|\eta_{h}^{n}\|^{2} 
		\leq C (\tau^{2} + h^{2}) e^{C\sum_{j=1}^{n}\tau} \leq C (\tau^{2} + h^{2}),
	\end{equation*}
	which leads to
	\begin{equation}\label{e3.5.11}
		\|\boldsymbol{\xi}_{h}^{n}\| + \|\boldsymbol{\psi}_{h}^{n}\| 
		+ \|\boldsymbol{\gamma}_{h}^{n}\|_{1} + \|\eta_{h}^{n}\| \leq C (\tau + h).
	\end{equation}
	From \eqref{e3.5.9}, \eqref{e3.5.11} and Theorem \ref{t3.4.7}, it follows that
	\begin{equation*}
		\|\mathbf{E}(t_{n})-\mathbf{E}_{h}^{n}\| + \|\mathbf{H}(t_{n})-\mathbf{H}_{h}^{n}\| + \|\mathbf{u}(t_{n})-\mathbf{u}_{h}^{n}\|_{1} + \|p(t_{n})-p_{h}^{n}\|  \leq C(\tau+h),
	\end{equation*}
	which completes the proof.
\end{proof}

\section{Numerical experiments}\label{s3.6}
Let $\Omega=[0, 1]^{3}$, $T=0.1$ and fix physical parameters $\epsilon=1$, $\sigma=2$, $L=1$, $\mu=1$, $\lambda=2$, $G=1$, $\alpha=1$, $c_{0}=1$ and $\kappa=2$. Set time step-size $\tau=(1/20)^{2}$ and take uniform tetrahedral partition for $\Omega$ with mesh size $h$. We employ the FEniCS \cite{AlBlHaJoKeLoRiRiRoWe2015,LoMaWe2012} computing platform to carry out numerical experiments. 

\begin{example}\label{eg3.6.1}
	The initial value $(\mathbf{E}_{0},\mathbf{H}_{0},\mathbf{u}_{0},p_{0})$ and right-hand functions $\mathbf{j}$, $\mathbf{f}$, $g$ are chosen such that the exact solution to \eqref{e3.2.1}--\eqref{e3.2.3} reads
	\begin{align*}
		\mathbf{E} =
		\left(
		\begin{array}{c}
			\sin(\pi x)\sin(\pi y)\sin(\pi z)\sin(t)
			\\
			\sin(\pi x)\sin(\pi y)\sin(\pi z)\sin(t)
			\\
			\sin(\pi x)\sin(\pi y)\sin(\pi z)\sin(t)
		\end{array}
		\right),
	\end{align*}
	\begin{align*}
		\mathbf{H} =
		\left(
		\begin{array}{c}
			\pi(\sin(\pi x)\cos(\pi y)\sin(\pi z)-\sin(\pi x)\sin(\pi y)\cos(\pi z))\cos(t)
			\\
			\pi(\sin(\pi x)\sin(\pi y)\cos(\pi z)-\cos(\pi x)\sin(\pi y)\sin(\pi z))\cos(t)
			\\
			\pi(\cos(\pi x)\sin(\pi y)\sin(\pi z)-\sin(\pi x)\cos(\pi y)\sin(\pi z))\cos(t)
		\end{array}
		\right),
	\end{align*}
	\begin{align*}
		\mathbf{u} =
		\left(
		\begin{array}{c}
			\sin(\pi x)\sin(\pi y)\sin(\pi z)e^{-t}
			\\
			\sin(\pi x)\sin(\pi y)\sin(\pi z)e^{-t}
			\\
			\sin(\pi x)\sin(\pi y)\sin(\pi z)e^{-t}
		\end{array}
		\right),
	\end{align*}
	\begin{equation*}
		p = \sin(\pi x)\sin(\pi y)\sin(\pi z)e^{-t}.
	\end{equation*}
\end{example}
	Here, $\mathbf{E}$ and $\mathbf{H}$ are periodic in $t$. 
	We implement numerical computations with splitting backward Euler FEM \eqref{e3.5.2} and \eqref{e3.5.3} for $h= \frac{1}{4},\frac{1}{8},\frac{1}{12},\frac{1}{15},\frac{1}{18}$, respectively. The corresponding errors and convergence orders are listed in Table \ref{tab3.6.1}. We observe that the errors for $\mathbf{u}$ and $p$ in $L^{2}$-norm exhibit extra convergence orders comparing with the theoretical results in Theorem \ref{t3.5.3}. We also apply a standard FEM to carry out numerical experiments \cite{HuMe2022} to compare with our method. The errors of our method and standard FEM in log-log scale are shown in Figure \ref{fig3.6.1}, which indicates that both of these two methods have the same convergence order. Table \ref{tab3.6.2} shows that the splitting backward Euler FEM improves the computational efficiency.

	\begin{table}[H]
		\belowrulesep=0pt
		\aboverulesep=0pt
		\renewcommand{\arraystretch}{1.2}
		\setlength{\arrayrulewidth}{0.1mm}
		\centering
		\caption{Errors and convergence orders of splitting numerical solution.}
		\label{tab3.6.1}
		\begin{tabular}{@{}c|cc|cc|cc@{}}
			\toprule
			$h$ & $\|\mathbf{E}(t_{n})\!-\!\mathbf{E}_{h}^{n}\|$ & Order & $\|\mathbf{H}(t_{n})\!-\!\mathbf{H}_{h}^{n}\|$ & Order & - & - 
			\\
			\midrule
			1/4 & 3.5223E-02 & - & 8.4312E-01 & - & - & - 
			\\
			1/8 & 1.7263E-02 & 1.0292 & 4.2855E-01 & 0.9763 & - & - 
			\\
			1/12 & 1.1197E-02 & 1.0678 & 2.8632E-01 & 0.9947 & - & - 
			\\
			1/15 & 8.8741E-03 & 1.0418 & 2.2915E-01 & 0.9981 & - & - 
			\\
			1/18 & 7.3633E-03 & 1.0236 & 1.9098E-01 & 0.9993 & - & - 
			\\
			\midrule[0.7pt]
			$h$ & $\|\mathbf{u}(t_{n})\!-\!\mathbf{u}_{h}^{n}\|$ & Order & $\|\mathbf{u}(t_{n})\!-\!\mathbf{u}_{h}^{n}\|_{1}$ & Order & $\|p(t_{n})\!-\!p_{h}^{n}\|$ & Order
			\\
			\midrule[0.2pt]
			1/4 & 1.1322E-01 & - & 1.4423E-00 & - & 8.0152E-02 & - 
			\\
			1/8 & 2.9383E-02 & 1.9460 & 7.5397E-01 & 0.9358 & 2.2321E-02 & 1.8443
			\\
			1/12 & 1.3135E-02 & 1.9857 & 5.0650E-01 & 0.9812 & 1.0114E-02 & 1.9523
			\\
			1/15 & 8.4201E-03 & 1.9928 & 4.0608E-01 & 0.9903 & 6.5009E-03 & 1.9808
			\\
			1/18 & 5.8522E-03 & 1.9954 & 3.3879-01 & 0.9936 & 4.5192E-03 & 1.9942
			\\
			\bottomrule
		\end{tabular}
	\end{table}
	
	\begin{figure}[H]
		\centering   
		\mbox{ \includegraphics [width=0.48\textwidth,trim=0 0 0 0,clip]{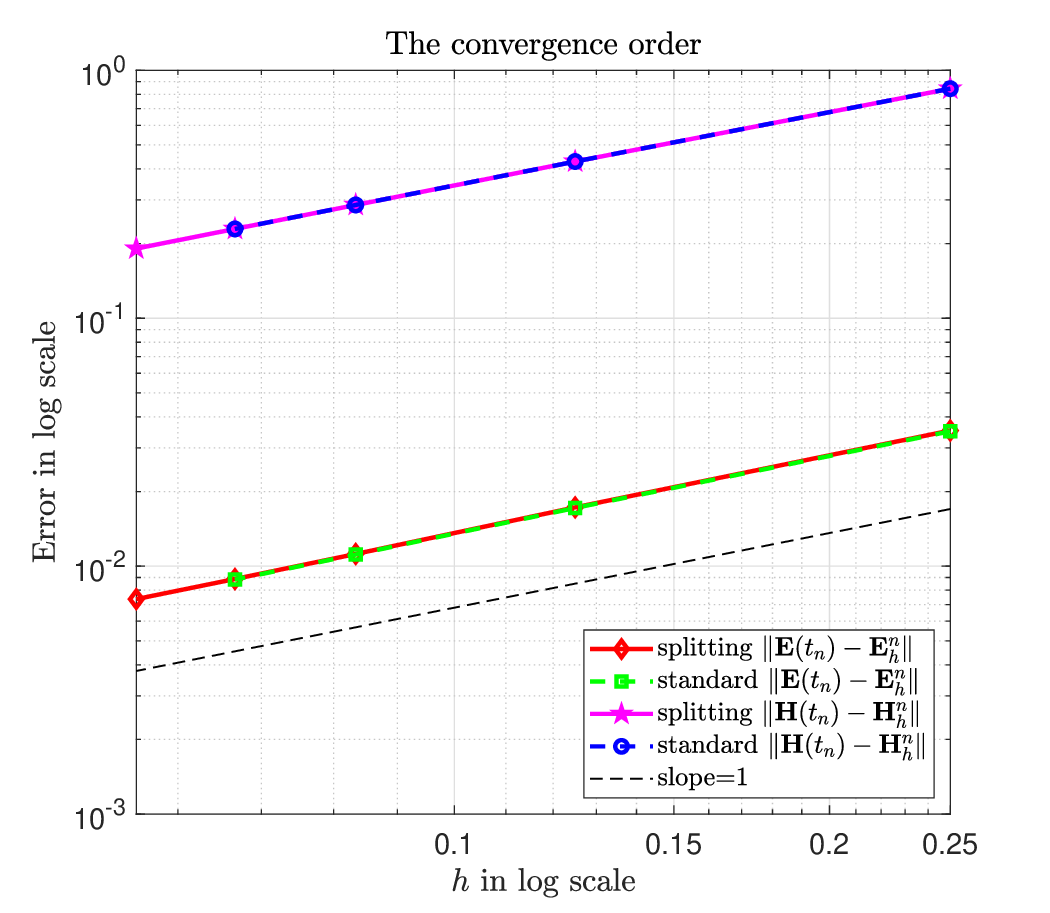}
		\includegraphics [width=0.48\textwidth,trim=0 0 0 0,clip]{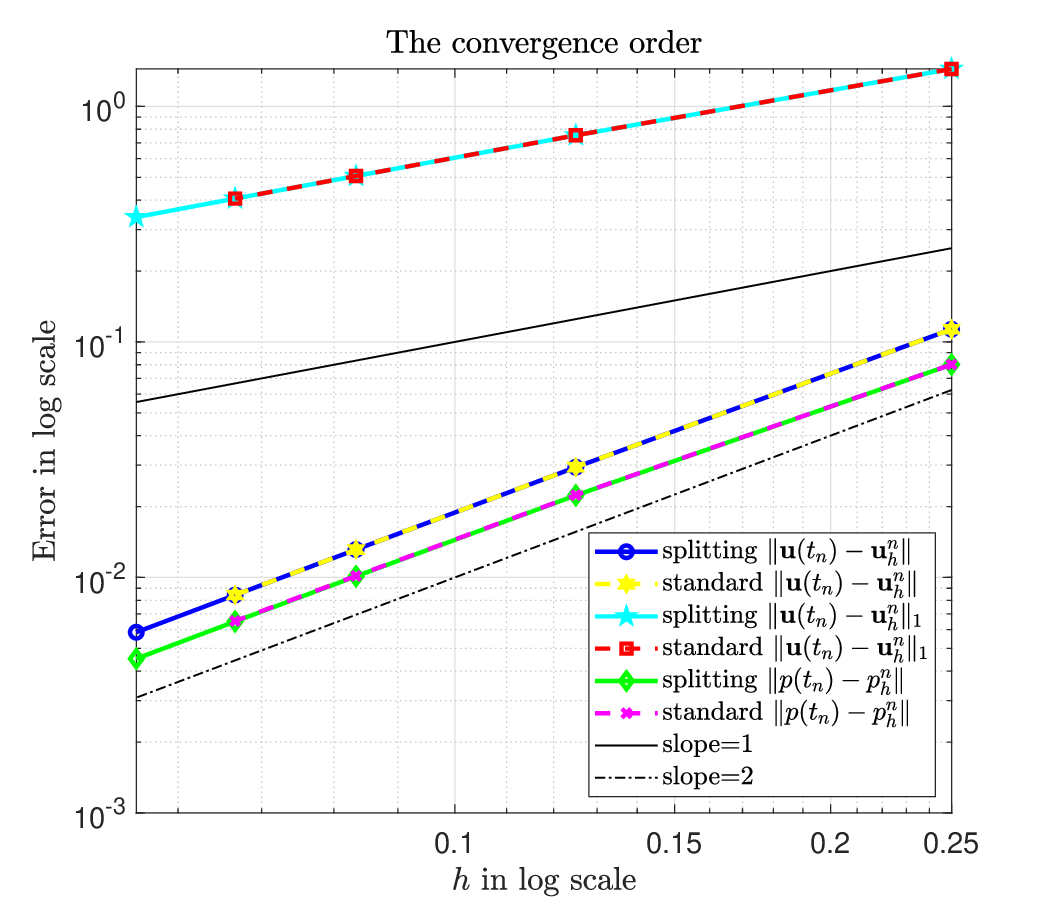} }
		\caption{Errors in log-log scale.}
		\label{fig3.6.1}
	\end{figure}
	
	\begin{table}[H]
		\belowrulesep=0pt
		\aboverulesep=0pt
		\renewcommand{\arraystretch}{1.2}
		\centering
		\caption{Comparison of computational complexity.}
		\label{tab3.6.2}
		\begin{tabular}{@{}|c|c|c|@{}}
			\toprule
			\diagbox{\qquad \qquad \quad \ \ $h$}{run time (s)} & Our method & Standard FEM
			\\
			\midrule
			1/4 & 11 & 13
			\\
			1/8 & 92 & 153
			\\
			1/12 & 516 & 1374
			\\
			1/15 & 1745 & 7098
			\\
			\bottomrule
		\end{tabular}
	\end{table}

\section{Conclusions}\label{s3.7}

In this paper, we established a new well-posedness theory for the quasi-static electroporoelasticity model and developed a splitting fully discretized finite element method (FEM) approximation. We derived the convergence order in both temporal and spatial variables for the finite element approximation. The efficiency of the splitting scheme was validated through numerical experiments.  In future work, we aim to derive the second-order convergence for the pressure $p$ and displacement $\mathbf{u}$, as suggested by our numerical results. Additionally, we plan to investigate locking-free numerical approximations in the regime where the Lam\'{e} constant becomes large.



%
\section*{Compliance with Ethical Standards}
The authors declare that there is no conflict of interest. Research does not involve human participants or animals. 
\section*{Data Availability Statement}
No data set is used in the reesearch

\bibliographystyle{spmpsci}      
\bibliography{electroporoelasticity_splitting_references}   

%
%

\end{document}